\documentclass[12pt, reqno]{article}
\usepackage{amsthm, amsmath, amstext, mathtools, mathrsfs, bm}
\usepackage[margin=1in]{geometry}
\usepackage[colorlinks=true, pdfstartview=FitV, linkcolor=blue, citecolor=blue, urlcolor=blue]{hyperref}
\usepackage{enumerate}
\usepackage{newtxtext}
\usepackage{newtxmath}
\usepackage[toc, page]{appendix}
\allowdisplaybreaks

\newtheorem{theorem}{Theorem}[section]
\newtheorem{lemma}[theorem]{Lemma}
\newtheorem{proposition}[theorem]{Proposition}
\newtheorem{corollary}[theorem]{Corollary}
\newtheorem{definition}[theorem]{Definition}

\newtheorem{remark}[theorem]{Remark}    
\newtheorem*{claim}{Claim}
\numberwithin{equation}{section}

\makeatletter
\@tfor\next:=abcdefghijklmnopqrstuvwxyzABCDEFGHIJKLMNOPQRSTUVWXYZ\do{
\def\command@factory#1{%
\expandafter\def\csname b#1\endcsname{\mathbf{#1}}
\expandafter\def\csname fk#1\endcsname{\mathfrak{#1}}
\expandafter\def\csname bb#1\endcsname{\mathbb{#1}}
\expandafter\def\csname cl#1\endcsname{\mathcal{#1}}
\expandafter\def\csname bcl#1\endcsname{\mathbfcal{#1}}
}
\expandafter\command@factory\next
}
\def\pounds{\mbox{\textsterling}}
\newcommand{\rmd}{\textnormal{d}}

\title{\textbf{Solution properties of the incompressible Euler system with rough path advection }}
\author{
Dan Crisan\thanks{\footnotesize Department of Mathematics, Imperial College London, SW7 2AZ, UK}\thanks{Emails: \href{mailto:d.crisan@imperial.ac.uk}{d.crisan@imperial.ac.uk},  \href{mailto:d.holm@imperial.ac.uk }{d.holm@imperial.ac.uk}, \href{mailto:j.leahy@imperial.ac.uk}{j.leahy@imperial.ac.uk}}
\and
Darryl D. Holm\footnotemark[1]\footnotemark[2]
\and
James-Michael Leahy\footnotemark[1]\footnotemark[2]\thanks{\footnotesize Department of Applied Mathematics, University of Twente, 7522 NB Enschede
}
\and 
Torstein Nilssen\thanks{\footnotesize Institute of Mathematics, University of Agder, 4604 Kristiansand, NO. Email: \href{mailto:torstein.nilssen@uia.no}{torstein.nilssen@uia.no}}
}

\date{}

\begin{document}

\maketitle

\begin{abstract}
The present paper aims to establish the local well-posedness of Euler's fluid equations on geometric rough paths. In particular, 
we consider the Euler equations for the incompressible flow of an ideal fluid whose Lagrangian transport velocity possesses an additional rough-in-time, divergence-free vector field. In recent work, we have demonstrated that this system can be derived from Clebsch and Hamilton-Pontryagin variational principles that possess a perturbative geometric rough path Lie-advection constraint. In this paper, we prove the local well-posedness of the system in $L^2$-Sobolev spaces $H^m$ with integer regularity $m\ge \lfloor d/2\rfloor+2$ and establish a Beale-Kato-Majda (BKM) blow-up criterion in terms of the $L^1_tL^\infty_x$-norm of the vorticity. In dimension two, we show that the $L^p$-norms of the vorticity are conserved, which yields global well-posedness and a Wong-Zakai approximation theorem for the stochastic version of the equation.
\end{abstract}

\renewcommand{\baselinestretch}{0.9}\normalsize
\setcounter{tocdepth}{3}
\tableofcontents
\renewcommand{\baselinestretch}{1.0}\normalsize

\section{Introduction}

Mathematical models of the interaction dynamics of disparate space and time scales in fluid flows remains one of the most active research areas, emboldened not just by theoretical considerations but also by practical applications such as the need to understand the dynamics of Earth's oceans and atmosphere in the context of global climate change. Incorporating fine-scale perturbations into the fluid motion equations and then analysing their effects on the coarse-scale flow has become one of the primary objectives of mainstream fluid models, particularly in the last two decades. Approaches toward this objective include deterministic perturbations \cite{smagorinsky1963general, germano1991dynamic, mason1994large,piomelli1999large, xie2020modeling} and stochastic perturbations \cite{piomelli1991subgrid, mason1992stochastic, schumann1995stochastic, piomelli1996large, pope2001turbulent, mikulevicius2001equations, shutts2005kinetic, berselli2009stochastic,majda2001mathematical, imkeller2012stochastic, memin2014fluid, holm2015variational, berner2017stochastic, demaeyer2018stochastic, skamarock2019description, street2021semi}. Such perturbations can be exogenously introduced into the fluid model to account for (possibly unknown) external forces. They can also be introduced endogenously, for example, to model the effects of unresolved fast sub-grid scale physics or other uncertain processes. In geophysical fluid dynamics, this trend has led to many numerical developments, including the introduction of parameterization schemes used to represent model uncertainties in the interaction of disparate space and time scales to improve the probabilistic skill of the ensemble weather forecasts \cite{shutts2005kinetic,berner2017stochastic, skamarock2019description, cotter2019numerically, cotter2020data, resseguier2020data, resseguier2020new, cotter2020particle}. 

Many fluid equations can be characterized as critical points of action functionals \cite{arnold1998differential, holm1998euler} that incorporate fluid physics via a Lagrangian whose kinetic energy is defined in terms of velocity vector fields which are right-invariant under the diffeomorphisms and whose potential energy is defined in terms of advected quantities which evolve under pushforward by the flows generated by these vector fields. Perhaps the most well-known fluid model that arises in this manner is the perfect incompressible Euler system \cite{arnold1998differential}, which describes geodesic flow on the manifold of diffeomorphisms endowed with the $H^s$-topology, $s>\frac{d}{2}+1$, with respect to the weak $L^2$-metric defined by the fluid kinetic energy, \cite{ebin1970groups}. 

A natural framework for introducing parameterization schemes into fluid models that arise from variational principles is through the introduction of a parameterized perturbation at the level of the Lagrangian in the action functional. Critical paths of the parameterized action functional then satisfy modified fluid motion equations that preserve the fundamental properties of the unperturbed model inherited from the variational principle. In particular, variational parameterization schemes for fluid models possess a Kelvin-Noether theorem which governs their circulation dynamics, as well as any other conservation laws arising from unbroken Lie symmetries of the original model \cite{holm1998euler}. The stochastic setting for this variational approach was introduced in \cite{holm2015variational}, and many further developments of it have been made subsequently, \cite{gay2018stochastic, de2020implications, drivas2020lagrangian, street2021semi,cotter2019numerically}. In \cite{crisan2022variational}, the present authors extended \cite{holm2015variational} to obtain a class of variational principles for fluid dynamics on geometric rough paths \cite{MR2314753, FrVi10, friz2020course}. This extension to geometric rough paths was achieved by constraining the advective transport equation to incorporate a temporally rough vector field. Critical points of the corresponding action functionals satisfy a system of rough partial differential equations (RPDEs) whose dynamics incorporate both the resolved-scale fluid velocity and the modeled effects of the unresolved fluctuations. The paper \cite{crisan2022variational} provides a bridge between fluid dynamics and rough path theory. It draws upon knowledge from both areas, and we hope that it will impact both areas.  

In this work, we consider Euler's equations for perfect incompressible fluid flow on geometric rough paths. It was shown in \cite{crisan2022variational}[Section 4.1 and 4.2] that this system is a critical path of the Clebsch or Hamilton-Pontryagin action functionals and satisfies a Kelvin circulation balance law (see, also, Section \ref{sec:main_res_state_critical}).  On the $d$-dimensional torus $\bbT^d$, the rough Euler system is given by
\begin{equation*}
\begin{cases}
\rmd u +u\cdot \nabla u\,\rmd t -\pounds_{\xi_k}^*u\,\rmd \bZ^k_t = -\rmd\nabla  q_t - \rmd h_t \quad \textnormal{on}\;\; (0,T]\times \bbT^d ,\\
\operatorname{div}  u=0 \quad \textnormal{on}\;\; [0,T]\times \bbT^d ,\\
\int_{\bbT^d}u\, dV=0,\quad  \int_{\bbT^d}q\,dV=0 \quad \textnormal{on} \;\;[0,T],\\
u=u_0,\quad q=0,\quad  h=0 \quad \textnormal{on} \;\;\{0\}\times \bbT^d ,
\end{cases}
\end{equation*}
where $\bZ=(Z,\mathbb{Z})\in \mathcal{C}^{p-\textnormal{var}}_{g}(\bbR_+;\bbR^K)$ is an $\bbR^K$-valued geometric rough path with variation  $p\in [2,3)$,  $\xi: \bbT^d\rightarrow \{\bbR^d\}^K$ is a collection of sufficiently smooth divergence-free vector fields, and 
$
\pounds_{\xi_k}^*u=-(\xi_k^j\partial_{x^j}u^i+u^j\partial_{x^i}\xi_k^j)e_i.
$
Equation \eqref{eq:rough_velocity_unprojected} is to be solved for an unknown divergence and mean-free vector field (velocity) $u: [0,T]\times \bbT^d\rightarrow \bbR^d$, mean-free scalar field (`time-integrated' pressure) $q:[0,T] \times \bbT^d\rightarrow \bbR$, and harmonic constant (time-integrated) $h: [0,T]\rightarrow \bbR^d$.  The pressure $q$ and harmonic constant $h$ should be understood as Lagrange multipliers associated with the divergence-free and mean-free constraints, respectively.  The present paper aims to establish the local well-posedness of equations \eqref{eq:rough_velocity_unprojected} in $L^2$-Sobolev spaces $H^m$ with integer regularity $m\ge \lfloor d/2\rfloor+2$. 

The $p$-variation of the rough path provides a measure of roughness: the higher the value of $p$, the rougher the path. See Definition \ref{def:rough_path} for details. The analysis of the Euler equation perturbed by a bounded variation path (i.e., $p=1$) follows the same steps as that of the (unperturbed) Euler equation (see, e.g., \cite{majda2002vorticity}[Chapter 3] and \cite{bahouri2011fourier}[Chapter 7]). We refer to \cite{majda2002vorticity}[Notes for Chapter 3] and \cite{bahouri2011fourier}[References and remarks for Chapter 7] for a historical account of the solution theory for the deterministic incompressible Euler equations.

For $p\ge 2$,  the classical integration methodology is no longer applicable and we enter the realm of rough path theory (see e.g., \cite{friz2020course}). In this paper, we treat rough path perturbations with $p\in [2,3)$, in the first non-trivial regime. A similar treatment is possible for paths with variation $p\in (1,2)$ (i.e., the Young integration case) and  $p\ge 3$. 

The case $p=2$  includes the incompressible stochastic Euler system driven by Brownian motion. See, for example,  \cite{mikulevicius2000stochastic, glatt2014local, brzezniak2016existence, crisan2019solution, breit2018stochastically, lang2022well} for a non-exhaustive selection of papers covering various types of incompressible stochastic Euler systems. Having eliminated the need for stochastic integration, the RPDEs retain a pathwise interpretation.  Unlike in the stochastic setting, though, no Burkholder-Davis-Gundy inequality is available to estimate the rough integral. Consequently, we will apply the method of unbounded rough drivers \cite{BaGu15} (explained below). This method will enable us to establish a Wong-Zakai approximation for the solution. In addition, we will give an interpretation of the solution as a bona fide random dynamical system \cite{bailleul2017random}. It is worth mentioning that it is possible to prove Wong-Zakai approximation results for SPDEs driven by Brownian motion without the use of the rough paths (see, e.g., \cite{gyongy2013rate, brzezniak2019wong}). 

Our work also includes a solution theory for fractional Brownian motion driven equations,  which enables memory effects to be introduced through our formulation. In our previous work \cite{crisan2022variational}, we explained how fluid models on geometric rough paths can be used in the context of stochastic parameterization schemes and uncertainty quantification (see, e.g., \cite{cotter2019numerically, cotter2020data, cotter2020particle}. Our results also set the stage for investigating numerical schemes and developing geometric rough path parameterization schemes for fluid models to account for additional  properties such as unknown Lagrangian trajectory roughness and system memory.

An intrinsic theory of transport-type RPDEs was developed in the papers \cite{BaGu15, DiFrSt14}.  In \cite{BaGu15}, the authors use a priori estimates based on Davies' type expansions \cite{davie2008differential}, doubling of variables, and commutator estimates in the spirit of DiPerna-Lions to establish well-posedness and the analog of the renormalization property. In contrast, the authors of \cite{DiFrSt14} use a generalized Feynman-Kac formula and forward-backward duality type arguments. The method of unbounded rough drivers was extended in \cite{DeGuHoTi16} to RPDEs with non-linear drift terms, which influenced a series of papers \cite{hocquet2018energy, HN, hofmanova2019navier, hofmanova2019rough, HOCQUET20202159,gerasimovivcs2021non,gussetti2021pathwise,hocquet2020quasilinear,coghi2021rough, flandoli2020global}.  The most relevant to our current study are the papers  \cite{hofmanova2019navier, hofmanova2019rough, flandoli2020global} on the rough Navier-Stokes equations. 

In \cite{hofmanova2019navier}, a rough Navier-Stokes equation with a pure linear rough transport perturbation $\xi_k\cdot \nabla u \rmd \bZ_t^k$ was considered, which at the formal level, conserves energy. At the present moment, the only method the authors know to obtain estimates (after passing to the limit) and the uniqueness of solutions of rough transport type PDEs in the highest spatial norm for which one has a uniform bound in time is through the method of doubling variables and commutator estimates. Unfortunately, the incompressibility constraint present in the velocity formulation makes it challenging to obtain solution estimates using this method. In \cite{hofmanova2019navier}, the existence of ``weak'' solutions (i.e., $u\in L^2([0,T]; H^1)\cap L^\infty([0,T];L^2)$ was proved in 2d and 3d by establishing an energy equality along a smooth-noise approximating sequence, and thereby deriving an energy inequality for the solution by passing to the limit. In dimension two, uniqueness, continuity in $L^2$, and an energy equality were only proved in the special case of constant vector fields $\xi$ because of the difficulty that arises from the incompressibility constraint. It is worth noting that for constant $\xi$, scalar linear transport agrees with Lie transport.

In \cite{hofmanova2019rough}, a Navier-Stokes equation with Lie-advection in the noise, as in the present paper, was considered. In contrast with \cite{hofmanova2019navier}, the vorticity formulation of the Lie-advected equation does not include projections, which facilitated the use of the doubling of variables technique for ``strong'' solutions (i.e., $u\in L^2([0,T];H^2)\cap L^\infty([0, T]; H^1)$ (see Remark \ref{rem:vorticity_no_projection}). In \cite{hofmanova2019rough}, the local existence of strong solutions was proved in dimension three, and  global existence and uniqueness was proved for strong solutions in dimension two. Additional regularity was not investigated, and a rough version of an Aubin-Lions type compactness result was used in the proof of existence.

The present paper expands the scope of the theory of unbounded rough drivers.  The main tool from this theory is Theorem \ref{thm:remainder_est}, which extends the usual Davies' remainder estimates for RDEs to RPDEs with linear rough transport structure.   As in \cite{hofmanova2019rough},  we make use the vorticity formulation to obtain solution estimates. In the absence of the smoothing effect of viscosity (c.f., \cite{hofmanova2019rough}), we need to obtain  bounds on the higher-order derivatives of the solution.   Specifically, in Section \ref{sec:apriori}, we derive a system of equations for the vorticity and its derivatives  up to order $m$ with $m\ge \lfloor d/2\rfloor + 2$, which has a linear rough symmetric transport structure. We then derive an equation for the square of this system and apply Theorem \ref{thm:remainder_est}. Next, we obtain solution estimates by applying a rough version of Gr\"onwall's lemma (see Section \ref{sec:rough_gron}). The additional  regularity of the solutions allows us to avoid using  doubling-of-variables  (used in e.g., \cite{BaGu15, DeGuHoTi16,hofmanova2019rough}) at the expense of having to assume $\xi$ is slightly more regular whenever continuity in time in the highest norm is needed. This simplifies many of the techniques needed for obtaining a priori estimates (see Remark \ref{rem:double_variable}). As a result of the additional regularity, we use Arzela-Ascoli to obtain compactness, rather than Aubin-Lions as in \cite{hofmanova2019navier, hofmanova2019rough}.  In Section \ref{sec:maximally extended solution}, we show how to construct a maximally extended solution in the rough path setting, which  as far as we are aware is a new result. Further extensions of this work to bounded domains with boundary conditions is the subject of future research. However, the extension to bounded domains will encounter technical difficulties involving, for example, the method of doubling of variables, or applying the equation for the square of the vorticity when boundary conditions are present (provided additional regularity is known).

\textbf{Description of results.} The  goal of the present paper is to  establish solution properties of the rough  incompressible Euler system given in equation \eqref{eq:rough_velocity_unprojected}. Theorem \ref{thm:local_wp} states that for an initial condition of Sobolev regularity $H^m$ with $m\ge \lfloor \frac{d}{2}\rfloor + 2$ there exists a unique $H^m$ solution of the $d$-dimensional rough Euler system on the interval $[0, T_*],$ where the time $T_*>0$ depends on the initial condition of the equation, the driving rough path, and the coefficients of the rough driver. This solution can be extended to a unique maximal solution (Corollary 
\ref{cor:maximal_solution}). Theorem \ref{thm:BKM} gives a blow-up criterion in terms of the $L^1_tL^\infty_x$-norm of the vorticity, which extends the well-known Beale-Kato-Majda criterion \cite{beale1984remarks}. In dimension two, we show that the $L^p$-norms of the vorticity are conserved for all $p\in [2,\infty]$ and thereby establish global well-posedness (see Theorem \ref{thm:global2d}). Corollary \ref{cor:stability} states that the solution, the pressure term, and harmonic constant are jointly continuous as functions of the initial condition, the driving rough path, and the coefficients of the rough driver. In Section \ref{sec:main_res_state_wong_zakai}, we discuss applications of our results to stochastic partial differential equations (Sades), including Theorem \ref{thm:Wong-Zakai}, which gives a Wong-Zakai approximation result for the corresponding Stratonovich driven SPDE. In Section \ref{sec:main_res_state_critical}, we explain how our main results yield critical points of the Clebsch and Hamilton-Pontryagin action functionals introduced in \cite{crisan2022variational}. To do this, we derive a generalized Weber representation formula \eqref{eq:WeberFormula} of smooth (in space) solutions.

\vspace{1mm}

\textbf{Structure of the paper}: 
\begin{itemize}
\item Section \ref{sec:prelim} introduces the basic notation and summarizes the standard results that will be required throughout the paper.
\item In Section \ref{Main Results}, we formulate the rough incompressible Euler equation in velocity form and as well as vorticity form and state the main results of the paper: local well-posedness in any dimension, the Beale-Kato-Majda blow-up criteria, global well-posedness in 2d, and continuous dependence on data. We discuss various applications to stochastic equations, including a Wong-Zakai approximation result. Finally, we explain the correspondence of solutions with critical points of the action functionals presented in \cite{crisan2022variational}.
\item Section \ref{sec:apriori} contains a priori estimates for remainders, the solution, and differences of solutions.
\item Section \ref{sec:local_wp} contains the proof of local well-posedness.
\item Section \ref{sec:proof_remaining} contains the proof of the remaining results.
\item The appendices \ref{sec:URD} and \ref{sec:rough_gron} contain technical results that are used in establishing the a priori estimates.                                                  
\end{itemize}

\subsection*{Acknowledgments} 
All of the authors are grateful to our friends and colleagues who have generously offered their time, thoughts, and encouragement in the course of this work during the time of COVID-19. 
We are particularly grateful to T. D. Drivas  for thoughtful discussions, especially an interesting discussion regarding Section \ref{sec:main_res_state_critical}. We also thank the anonymous referee for the helpful comments and constructive remarks on this manuscript. DC and DH are grateful for partial support from ERC Synergy Grant 856408 - STUOD (Stochastic Transport in Upper Ocean Dynamics). JML is grateful for partial support from US AFOSR Grant FA9550-19-1-7043 - FDGRP (Fluid Dynamics of Geometric Rough Paths) awarded to DH as PI. TN is grateful for partial support from the DFG via the Research Unit FOR 2402. 

\section{Preliminaries}\label{sec:prelim}

\subsection{Basic notation, function spaces, and inequalities}\label{sec:prelim_func}

Let $d\in \{2,3,\ldots \}$ and $\bbT^d= \bbR^d/(2\pi\bbZ)^d\cong (S^1)^d$  denote the $d$-dimensional flat torus. The Riemannian covering map $\pi:\bbR^d\rightarrow \bbT^d$ induces global orthonormal frames $\{\partial_{x^i}\}_{i=1}^d$ and $\{dx^i\}_{i=1}^d$ of the tangent $T\bbT^d\cong \bbT^d\times \bbR^d$ and cotangent bundle $T^*\bbT^d\cong \bbT^d\times (\bbR^d)^*$, respectively, and a normalized  Haar measure  $dV=\textnormal{Vol}(\bbT^d)^{-1}dx^1\wedge \cdots \wedge dx^d$. 

Let $V$ denote an arbitrarily given finite dimensional real vector space with inner product $(\cdot,\cdot)_V$ and norm $|\cdot|_V$. Denote by $C^\infty(\bbT^d;V)$  the Fr\'echet space of smooth $V$-valued functions on $\bbT^d$. For given $m\in \bbN$, let  $C^m(\bbT^d;V)$ denote the Banach space of $m$-times continuously differentiable $V$-valued functions on $\bbT^d$. We will blur the distinction between $2\pi$-periodic functions and functions on $\bbT^d$.  Let $\clL(\bbR^d;V)$ denote the space of linear maps from vector space $\bbR^d$ to $V$ and $D : C^\infty(\bbT^d;V)\rightarrow C^{\infty}(\bbT^d;\clL(\bbR^d;V))$ denote the derivative operator. Let $\Delta: C^\infty(\bbT^d;V)\rightarrow C^{\infty}(\bbT^d;V)$ denote the Laplacian, which is defined by $\Delta f = \partial_{x^i}^2f$. Here and below we use the convention of summing repeated indices over their range of values. 

For given  $p\in [1,\infty]$, denote by  $L^p(\bbT^d;V)=L^p((0,2\pi)^d;V)$  the Banach space of  equivalence classes of $V$-valued of $L^p$-integrable functions on $\bbT^d$ with norm $|\cdot|_{L^p}$. We denote by $(\cdot,\cdot)_{L^2}$ the inner product on the Hilbert space $L^2(\bbT^d;V)$. Since it will always be clear from the context where a function takes its values, we drop the dependence of norms and inner products on $V$. It is well-known that the sequence $\{\psi_n\}_{n\in \bbZ^d}\subset C^\infty(\bbT^d;\bbC)$ defined by for all $n\in \bbZ^d$ and $x\in \bbT^d$ by $\psi_n(x)=e^{in\cdot x}$ forms an orthonormal basis of $L^2(\bbT^d;\bbC)$. 

Denote the Fourier transform $\clF: L^1(\bbT^d;V) \rightarrow \ell^\infty(\bbZ^d;V)$ by $\clF f(n)=\hat{f}(n)=\int_{\bbT^d}f\psi_n dV$ and its inverse $\clF^{-1}: l^\infty(\bbZ^d;V)\rightarrow L^1(\bbT^d;V)$ by $\clF^*\hat{f}=\sum_{n\in \bbZ^d}\hat{f}(n)\psi_n$. Let $S(\bbZ^d;V)=\{\{c_n\}_{n\in \bbZ^d}\subset V: \sup_{n\in \bbZ^d}(1+|n|)^N|c_n|<\infty, \;\forall N\in \bbN\}$ denote the Fr\'echet space of rapidly decreasing  multi-sequences. It follows that $\clF : C^\infty(\bbT^d;V)\rightarrow S(\bbZ^d;V)$, $\clF^{-1}: S(\bbZ^d;V)\rightarrow  C^\infty(\bbT^d;V)$, and moreover that $\clF$ extends to an isometric isomorphism $\clF: L^2(\bbT^d;V)\rightarrow \ell^2(\bbZ^d;V)$.

Denote by $\clD'(\bbT^d;V)=\{f: C^\infty(\bbT^d;\bbC)\rightarrow V; \;f\textnormal{ is linear and continuous}\}$  the Fr\'echet space of $V$-valued distributions. Equivalently, $\clD'(\bbT^d;V)$ may be characterized as $2\pi$-periodic distributions $\clD'(\bbR^d;V)$  or as distributional Fourier series. Indeed, the Fourier transform  extends via duality to an isomorphism $\clF: \clD'(\bbT^d;V)\rightarrow S'(\bbZ^d;V)$, where $S'(\bbZ^d;V)=\{\{c_n\}_{n\in \bbZ^d}\subset V: \exists N\in \bbN \textnormal{ s.t. } \sup_{n\in \bbZ^d}(1+|n|)^{-N}|c_n|<\infty\}$ denotes the space of slowly increasing  multi-sequences. 
Clearly, for all $p\in [1,\infty)$, $C^\infty(\bbT^d;V),L^p(\bbT^d;V)\subset \clD'(\bbT^d;V)$. The differential operators $\partial^{\alpha}$, $D$, and $\Delta$  extend to distributions via duality.

Denote by $\mathring{C}^\infty(\bbT^d;V), \mathring{L}^{p}(\bbT^d;V)$,  and $\mathring{\clD}'(\bbT^d;V)$ the corresponding subspace of distributions $f$ with $\hat{f}(0)=0$. If $f\in \clD'(\bbT^d;V)$  is such that $\Delta f = -\sum_{n\in \bbZ^d} |n|^2\hat{f}(n)\psi_n=0_V$ (i.e., harmonic), then $\hat{f}(n)=0$ for all $n\in \bbZ^d-\{0\}$, which implies $\operatorname{Ker}\Delta = V.$

For given $s\in \bbR$, define  $(I-\Delta)^{-\frac{s}{2}}: \clD'(\bbT^d;V)\rightarrow \clD'(\bbT^d;V)$ by  
$$(I-\Delta)^{-\frac{s}{2}}f = \sum_{n\in \bbZ^d}\langle n\rangle^{-s} \hat{f}(n)\psi_n\,,$$ where $\langle n\rangle = (1+|n|^2)^{\frac{1}{2}}.$ Let $s\in \bbR$ and $H^{s}(\bbT^d;V)=(I-\Delta)^{-\frac{s}{2}}L^2(\bbT^d;V)$ denote the Bessel potential spaces, which are Hilbert spaces with  inner products and norms given by
$$
(f,g)_{H^s}:= \left(\sum_{n\in \bbZ^d}\langle n\rangle^{s} (\hat{f}(n),\hat{g}(n))_V\right)^{\frac{1}{2}}\quad \textnormal{and} \quad
|f|_{H^{s,p}}:= \left|\sum_{n\in \bbZ^d}\langle n\rangle^{s} \hat{f}(n) \psi_n \right|_{L^p}.
$$
The duality pairing  $\langle \cdot,\cdot\rangle_{H^s}:H^{-s}(\bbT^d;V)\times  H^s(\bbT^d;V)\rightarrow \bbR$ given by
$$
\langle f, g \rangle_{H^s}=((I-\Delta)^{-\frac{s}{2}} f , (I-\Delta)^{\frac{s}{2}}g)_{L^2}
$$
induces an isomorphism $H^{-s}(\bbT^d;V)\cong H^{s}(\bbT^d;V)'$. It follows that for all $m\in \bbN_0$, $H^{m}(\bbT^d;V)=\{f\in L^2(\bbT^d;V): \sum_{0\le n\le m} |D^n f|_{L^2}^2<\infty\}$. For $m\in \bbN$, denote by $W^{m,\infty}(\bbT^d;V)$ the Banach space of functions $f\in L^\infty(\bbT^d;V)$ such that $|f|_{W^{m,\infty}}:=\max_{0\le |\alpha|\le m}|\partial^{\alpha} f|_{L^\infty}<\infty$.

For given $s\in \bbR$, we let $\mathring{H}^s(\bbT^d;V)=\{f\in H^s(\bbT^d;V): \hat{f}(0)=0_V\}$ and for $m\in \bbN_0$, let $\mathring{W}^{m,\infty}(\bbT^d;V)=\{f\in \mathring{W}^{m,\infty}(\bbT^d;V): \hat{f}(0)=\int_{\bbT^d}f dV=0_V\}.$

Henceforth, we use the notation $\lesssim_{c_1,\ldots, c_d}$ to denote less than equal to up to a constant $C=C(c_1,\ldots, c_d)$ depending only on parameters $c_1,\ldots, c_d$. Throughout the paper, we will make regular use of the following facts:
\begin{itemize}
\item  For given $N\in \bbN$, define $P_{\le N}: \clD'(\bbT^d;V)\rightarrow C^\infty(\bbT^d;V)$ by $P_{\le N}f =\sum_{|n|\le N}\hat{f}(n)\psi_n$. If $s_0<s_1$, then for all $f\in H^{s_0}(\bbT^d;V)$, $|P_{\le N}f|_{H^{s_1}}\le N^{s_1-s_0} |f|_{H^{s_0}}$ and for all $f\in H^{s_1}(\bbT^d;V)$, $|(I-P_{\le N})f|_{H^{s_0}}\le N^{s_0-s_1}|f|_{H^{s_1}}$.
\item  If $s>\frac{d}{2}+m$ for some $m\in \bbN_0$, then $C^{m}(\bbT^d;V)\subset H^{s}(\bbT^d;V)$.
\item (Poincar\'e) For all $f\in \mathring{H}^{s+1}(\bbT^d;V)$,  $|f|_{H^s}\lesssim_{d,s} |Df|_{H^s}$. For all $f\in \mathring{W}^{m+1,\infty}(\bbT^d;V)$, $|f|_{W^{m,\infty}}\lesssim_{d,m} |D f|_{W^{m,\infty}}$.
\item For all $m_1,m_2\in \bbN_0$ and $f,g\in L^\infty(\bbT^d;V)\cap H^{m_1+m_2}(\bbT;V),$
\begin{equation}\label{ineq:alg_sob_1}
||D^{m_1} f| |D^{m_2}g||_{L^2}\lesssim_{d,m_1,m_2}|f|_{L^\infty} |g|_{H^{m_1+m_2}}+|f|_{H^{m_1+m_2}}|g|_{L^\infty}.
\end{equation}
\item For all $m\in \bbN_0$, $f\in L^\infty(\bbT^d;V)\cap H^{m+1}(\bbT;V)$ and $g\in L^\infty(\bbT^d;V)\cap H^{m}(\bbT;V)$
\begin{equation}\label{ineq:alg_sob_2}
\begin{aligned}
\sum_{0\le |\alpha|\le m}|\partial^{\alpha}(f\nabla g)-f\partial^{\alpha}\nabla g|_{L^2}&\lesssim_{d,m}\left(|\nabla f|_{L^\infty}|g|_{H^m}+|f|_{H^{m+1}}|g|_{L^\infty} \right).
\end{aligned}
\end{equation}
\end{itemize}

Let $(E,|\cdot|_{E})$ be an arbitrarily given Banach space. For an interval $I\subset \bbR_+$, denote by $C(I;E)$ (resp.\ $C_w(I;E)$) the space of continuous (resp.\ weakly-continuous) $E$-valued functions. For $p\in [1,\infty]$, let  $L^p(I;E)$ the Banach space of equivalence classes of $E$-valued $L^p$-integrable strongly measurable $E$-valued functions on $I$.

\subsubsection{Hodge and Helmholtz decomposition and the Biot-Savart operator}\label{sec:prelim_Hodge}

Let $\{e_i\}_{i=1}^d$ denote the standard basis of $\bbR^d$ and let $\{e^i\}_{i=1}^d$ denote the dual basis. We identify vector fields $u=u^i\partial_{x^i}\in \Gamma(\bbT^d;T\bbT^d)$ with $\bbR^d$-valued maps $u=u^ie_i: \bbT^d\rightarrow \bbR^d$ and $k$-forms $\alpha=\sum_{i_1<\cdots <i_d}\alpha_{i_1\ldots i_k}dx^{i_1}\wedge \cdots \wedge dx^{i_k}\in \Gamma(\bbT^d;\Lambda^kT^*\bbT^d)$ with $\Lambda^k(\bbR^d)^*$-valued maps $\alpha=\sum_{i_1<\cdots <i_k}\alpha_{i_1\ldots i_k} e^{i_1}\wedge \cdots \wedge e^{i_k}: \bbT^d\rightarrow \Lambda^k(\bbR^d)^*$. Denote by $\flat$ the map that sends vector fields to one-forms given by $u^{\flat}=u_ie^i$, where $u_i=u^i$ since the metric is locally Euclidean. Let $\sharp$ denote the inverse of $\flat$. 

For a given diffeomorphism $\phi: \bbT^d\rightarrow \bbT^d$ and $\alpha\in \Gamma(\bbT^d;\Lambda^kT^*\bbT^d)$, denote the pullback and pushforward by
$$\phi^* \alpha=k!\underset{j_1<\cdots <j_k}{\sum_{i_1<\ldots <i_k}} \alpha_{j_1\ldots j_k}\circ \phi \partial_{x^{i_1}}\phi^{j_1}\cdots \partial_{x^{i_k}}\phi^{j_k}e^{i_1}\wedge \cdots \wedge e^{i_k}$$
and $\phi_*\alpha=\phi^{-1;*}\alpha$, respectively.

Denote by $\bd$, $\delta$, and $-\Delta=\bd \delta + \delta \bd $  the exterior derivative, co-differential (i.e., adjoint of $\bd$), and Hodge Laplacian operators, respectively. In particular, using the above identification, for $f\in \clD(\bbT^d;\bbR^d)$, $\alpha\in \clD'(\bbT^d;(\bbR^d)^*)$, $\omega\in \clD'(\bbT^d;\Lambda^2(\bbR^d)^*)$, and $\gamma\in \clD'(\bbT^d;\Lambda^3(\bbR^d)^*)$, 
\begin{gather*}
\bd f = \partial_{x^i}fe^i, \quad \bd \alpha  =\sum_{i<j} (\partial_{x^i} \alpha_j - \partial_{x^j} \alpha_i) e^i\wedge e^j,\\
\bd \omega  = \sum_{i<j<k}(\partial_{x^i} \omega_{jk}-\partial_{x^j}\omega_{ik}+\partial_{x^k}\omega_{ij}) e^i\wedge e^j\wedge e^k,\\
\delta f = 0, \quad  \delta \alpha = -\partial_{x^i}\alpha_i, \quad \delta \omega  = \partial_{x^j}\omega_{ij}e^i, \quad \delta \gamma = -\sum_{i<j}\partial_{x^k}\gamma_{ijl}e^i\wedge e^j,\\
-\Delta f=  -\partial_{x^j}^2f, \quad -\Delta \alpha =-\partial_{x^j}^2\alpha_ie^i, \quad 
-\Delta  \omega =-\partial_{x^k}^2 \omega_{ij} e^i\wedge e^j.
\end{gather*}
We recall that $\bd \phi_*=\phi_*\bd$, $\bd^2=0$, and $\delta^2=0$.

Let $s\in \bbR$ be arbitrarily given. For arbitrarily given  $\alpha\in H^s(\bbT^d;\Lambda^k(\bbR^d)^*)$, $k\in \{0,1,2\}$,  there exists a unique $\beta \in \mathring{H}^{s+2}(\bbT^d;\Lambda^k(\bbR^d)^*)$ such that $-\Delta \beta = \alpha-\hat{\alpha}(0)$ given by $
\beta=(-\Delta)^{-1}(\alpha-\hat{\alpha}(0))= \sum_{n\in \bbZ^d}|n|^{-2}\hat{\alpha}(n)\psi_n$, which yields the Hodge decomposition 
$$
\alpha = \bd \delta \beta +  \delta \bd \beta +\hat{\alpha}(0)=\bd \delta (-\Delta)^{-1}(\alpha-\hat{\alpha}(0))+ \delta \bd (-\Delta)^{-1}(\alpha-\hat{\alpha}(0)) + \hat{\alpha}(0).
$$
Let $$\mathring{H}^{s}_{\bd}(\bbT^d;\Lambda^{k}(\bbR^d)^*)=\bd \mathring{H}^{s+1}(\bbT^d;\Lambda^{k-1}(\bbR^d)^*)=\{\alpha\in \mathring{H}^s(\bbT^d;\Lambda^{k}(\bbR^d)^*): \bd \alpha =0\},$$
which is understood to be $\emptyset$ if $k=0$,
$$\mathring{H}^{s}_{\delta }(\bbT^d;\Lambda^{k}(\bbR^d)^*)=\delta \mathring{H}^{s+1}(\bbT^d;\Lambda^{k+1}(\bbR^d)^*)=\{\alpha\in \mathring{H}^s(\bbT^d;\Lambda^{k}(\bbR^d)^*): \delta \alpha =0\},$$ and  $\clH^k=\operatorname{Ker}\Delta =\Lambda^k(\bbR^d)^*$. Thus, we obtain the following orthogonal decomposition for $k\in \{0,1,2\}$:
\begin{equation}\label{eq:HodgeDecomp}
H^s(\bbT^d;\Lambda^k(\bbR^d)^*)=\mathring{H}^{s}_{\bd}(\bbT^d;\Lambda^{k}(\bbR^d)^*)\oplus \mathring{H}^{s}_{\delta}(\bbT^d;\Lambda^{k}(\bbR^d)^*) \oplus \clH^k.
\end{equation}
For  given $u\in \clD'(\bbT^d;\bbR^d)$, denote $\operatorname{div} u=\partial_{x^i}u^i=\partial_{x^i}u_i=-\delta u^{\flat}$. Moreover, let 
\begin{gather*}
\mathring{H}^s_{\sigma}(\bbT^d;\bbR^d)=\{u\in \mathring{H}^s(\bbT^d;\bbR^d): \operatorname{div} u =0\},\;\;  W^{m,\infty}_{\sigma}(\bbT^d;\bbR^d)=\{u\in W^{m,\infty}(\bbT^d;\bbR^d): \operatorname{div} u =0\},\\
\mathring{C}^\infty_{\sigma}(\bbT^d;\bbR^d)=\{u\in C^\infty(\bbT^d;\bbR^d): \operatorname{div} u =0\}, \;\;  \mathring{D}_{\sigma}'(\bbT^d;\bbR^d)=\{u\in \clD'(\bbT^d;\bbR^d): \operatorname{div} u =0\}.
\end{gather*}
Henceforth, we will simply write $\mathring{H}^s_{\sigma}, W^{m,\infty}_{\sigma}, \mathring{C}^\infty_{\sigma}$, and $\mathring{D}_{\sigma}'$.
Applying the sharp operator $\sharp$ to \eqref{eq:HodgeDecomp} in the case $k=1$, we derive the $d$-dimensional Helmholtz decomposition of vector fields:
$$
H^s(\bbT^d;\bbR^d)=\nabla  \mathring{H}^{s+1}(\bbT^d;\bbR) \oplus \mathring{H}^s_{\sigma}(\bbT^d;\bbR^d)\oplus \bbR^d.
$$
Let $Q\in \clL(H^s(\bbT^d;\bbR^d);\nabla  \mathring{H}^{s+1}(\bbT^d;\bbR))$, $P\in \clL(H^s(\bbT^d;\bbR^d); \nabla  \mathring{H}^{s}_{\sigma})$, and $H\in \clL(H^s(\bbT^d;\bbR^d);\bbR^d)$ denote the projections associated with the decomposition so that $I=P+Q+H$.

Let $$\operatorname{BS}= \sharp (-\Delta)^{-1}\delta =\sharp \delta (-\Delta)^{-1}$$ denote the inverse of $\bd \flat : H^{s+1}_{\sigma}(\bbT^d;\bbR^d)\rightarrow H^{s}_{\bd}(\bbT^d;\Lambda^{2}(\bbR^d)^*)$. The operator $\operatorname{BS}$ is called the Biot-Savart operator. Owing to Propositions 7.5 and 7.7 of  \cite{bahouri2011fourier}, for all $\omega \in \mathring{H}^{s-1}(\bbT^d;\Lambda^2(\bbR^d)^*)$ such that $s>\frac{d}{2}+1$, we have
\begin{align}
| \operatorname{BS}\omega|_{H^s}&\lesssim_d |\nabla \operatorname{BS}\omega|_{H^{s-1}}\lesssim_{d,s} |\omega|_{H^{s-1}}\lesssim_{d,s} | \operatorname{BS}\omega|_{H^s},\label{ineq:curl_equiv}\\
|\nabla \operatorname{BS}\omega|_{L^\infty}&\lesssim_{d,s} \ln\left(e+| \omega|_{H^{s-1}}\right)|\omega|_{L^{\infty}}, \;\;\forall \omega  \in \mathring{H}^{s-1}(\bbT^d;\Lambda^2(\bbR^d)^*)\label{ineq:Kato}.
\end{align}

\subsubsection{The Lie derivative}\label{sec:prelim_Lie}

Let $s>\frac{d}{2}+1$ so that $|\nabla f|_{L^\infty}<\infty$ for all $f\in H^s(\bbT^d;V)$. 
For $v\in H^{s}(\bbT^d; \bbR^d)$ and $k\in \{0,1,2\}$, the Lie derivative $\pounds_{v}\in \clL(H^s(\bbT^d;\Lambda^k(\bbR^d)^*);H^{s-1}(\bbT^d;\Lambda^k(\bbR^d)^*)$ acts on $f\in H^s(\bbT^d;\bbR)$, $\alpha\in H^s(\bbT^d;(\bbR^d)^*)$, $\omega\in H^s(\bbT^d;\Lambda^2(\bbR^d)^*)$ by
\begin{gather*}
\pounds_v f =v\cdot \nabla f= v^j\partial_{x^j}f, \qquad \pounds_v\alpha = (v^j\partial_{x^j} \alpha_i+\alpha_j\partial_{x^i}v^j)e^i,\\
\pounds_v \omega=\sum_{i<j}(v^{q}\partial_{x^{q}}\omega_{ij}+ \omega_{qj}\partial_{x^{i}}v^q +\omega_{iq}\partial_{x^{j}}v^q  )  e^i\wedge e^j.
\end{gather*}
Indeed, it is a direct consequence of \eqref{ineq:alg_sob_1} that the range of the Lie derivative lies in $H^{s-1}$. A simple computation shows that $\bd \pounds_{v}=\pounds_{v}\bd$. Moreover, if $\bd \omega=0$, then
$$
\pounds_{v}\omega=\sum_{i<j}\left(\partial_{x^i}(v^q\omega_{qj})-\partial_{x^j}(v^q\omega_{qi})\right)e^i\wedge e^j=\bd [\omega(v,\cdot)],
$$
which is  just a special case of Cartan's formula.

The Lie derivative and covariant derivative $\pounds_{v},\nabla_v=v\cdot \nabla\in \clL(H^s(\bbT^d;\Lambda^k(\bbR^d)^*);H^{s-1}(\bbT^d;\bbR^d)$ act on $u\in H^s(\bbT^d;\bbR^d)$ by
\begin{equation*}
\pounds_v u=(v^j\partial_{x^j}u^i-u^j\partial_{x^j}v^i)e_i, \quad  \nabla_{v}u=v\cdot \nabla u = v^j\partial_{x^j} u^ie_i.
\end{equation*}
Stoke's theorem implies that for all $w,u,v\in H^s(\bbT^d;\bbR^d)$, 
$$
(\pounds_v w, u)_{L^2} =-\int_{\bbT^d}w^i(v^j\partial_{x^j}u^i + u^i\operatorname{div} v  +u^j\partial_{x^i}v^j)dV=:(w,\pounds_{v}^* u)_{L^2}.
$$
In particular, if $\operatorname{div} v= 0$, then $$\pounds_{v}^*u=-(v^j\partial_{x^j}u^i   +u^j\partial_{x^i}v^j)e_i$$
and
\begin{gather}
-\pounds_{v}^*u=\sharp \pounds_v u^{\flat} \quad \Rightarrow \quad  -\bd \flat \pounds_{u}^* w =\bd \pounds_v u^{\flat}=-\pounds_v \bd u^{\flat}, \label{eq:ad_star_Lie}\\
P\pounds_{v}^*u=P\pounds_{v}^* Pu. \label{eq:projection_Lie}
\end{gather}
Moreover, 
\begin{equation}\label{eq:covd_Lie}
u\cdot \nabla u= -\pounds_{u}^*u +2^{-1}\nabla |u|^2.
\end{equation}

If $d=2$, then the Hodge Star (see, e.g, Section 6.2 of \cite{abraham2012manifolds}) map $\star:  H^s(\bbT^d;\Lambda^2(\bbR^2)^*)\rightarrow H^s(\bbT^2;\bbR)$ defined by $\tilde{\omega} =\omega_{12}$ is an isomorphism and for all $v\in H^s_{\sigma}(\bbT^d;\bbR^d)$, we have 
\begin{equation}\label{eq:star_commute_scalar}
\star \pounds_v \omega = \pounds_v \star \omega=v\cdot \nabla \omega. 
\end{equation}
Moreover, if $d=3$, then the Hodge Star map $\sharp\star  H^s(\bbT^d;\Lambda^2(\bbR^3)^*)\rightarrow H^s(\bbT^3;\bbR^3)$ given by $\sharp \star \omega =\omega_{23}e_1+\omega_{31}e_2+\omega_{12}e_3$ is an isomorphism and for all $v\in H^s_{\sigma}(\bbT^d;\bbR^d)$, we have 
\begin{equation}\label{eq:star_commute_vector}
\sharp \star \pounds_v \omega = \pounds_v \sharp\star \omega=[v,\sharp \star \omega]. 
\end{equation}

\subsection{Geometric rough paths and the sewing lemma}
For an arbitrarily given closed interval $I=[a,b]$,  denote $$\Delta_I := \{(s,t)\in I^2: s\le t\} \quad \textnormal{ and } \quad  \Delta^{2}_I := \{(s,\theta,t)\in I^3: s\le\theta\le t\}.$$  If the interval $I=[0,T]$, we write $\Delta_{T}=\Delta_{I}$. 

Let $(E, |\cdot|_E)$ be an arbitrarily given Banach space with norm $| \cdot |_E$. We say a two-parameter function $g: \Delta_I \rightarrow E$ has finite $p$-variation for some $p\in (0,\infty)$ on $I$ if 
$$
|g|_{p-\textnormal{var};I;E}:=\sup_{\mathfrak{p}=(t_i)\in \clP(I)}\left(\sum_{i=1}^{\# \mathfrak{p} -1}|g _{t_i t_{i+1}}|^p_E\right)^{\frac{1}{p}}<\infty, 
$$
where $\mathcal{P}(I)$ is the set of all  finite partitions of $I$ and  $\#\mathfrak{p}$ denotes the number of points in a given partition $\mathfrak{p}\in \mathcal{P}(I)$.  Denote by $C_2^{p-\textnormal{var}}(I; E)$ the set of all continuous functions with finite $p$-variation on $I$ equipped with the seminorm $|\cdot |_{p-\textnormal{var}; I;E}$. Denote by $C^{p-\textnormal{var}}(I; E)$ the set of all paths $z : I \rightarrow E$ such that $\delta z \in C_2^{p-\textnormal{var}}(I; E)$, where  $$\delta z_{st} := z_t - z_s, \quad (s,t)\in \Delta_I.$$

For an arbitrary interval $I\subset \bbR$ that is not necessarily closed, denote by $C_{2,\textnormal{loc}}^{p-\textnormal{var}}(I; E)$ the set of all continuous functions $g:\Delta_I\rightarrow E$ such that there exists a countable sequence of closed intervals  $\{I_k\}$ such that $\cup_k I_k=I$ and  $g\in C_{2}^{p-\textnormal{var}}(I_k; E)$. 

A continuous mapping $\varpi : \Delta_I \rightarrow [0,\infty)$ is called a control on $I$ if $\omega(s,s) = 0$  for all $s\in I$ and if for all $(s,\theta, t)\in \Delta^{2}_I$, 
$$
\varpi(s,\theta)+\varpi(\theta ,t)\le \varpi(s,t),
$$
which is referred to as superadditivity.
If $\varpi_1$ and $\varpi_2$ are controls, then for all $\alpha,\beta\in \bbR_+$ such that $\alpha+\beta\ge 1$, $\varpi_1^{\alpha}\varpi_2^{\beta}$ is a control (see, e.g., \cite{FrVi10}[Ex.\ 1.9]). 

If $g \in C^{p-\textnormal{var}}_2(I; E)$ for a given $p \in (0,\infty)$, then $\varpi_g: \Delta_I \rightarrow [0,\infty)$ defined for all $(s,t)\in \Delta_I$ by
$$
\varpi_g(s,t)= |g|_{p-\textnormal{var};[s,t];E}^p 
$$
is a control (see, e.g.,   \cite{FrVi10}[Prop. 5.8]). Moreover, it is straightforward to check that for all $(s,t)\in \Delta_I$
$$
\inf \{ \varpi(s,t)^{\frac{1}{p}} : \varpi \textnormal{ is a control s.t. }|g_{r\theta}|_E \leq \varpi(r,\theta)^{\frac{1}{p}}, \; \forall (r, \theta)  \in \Delta_{[s,t]} \} , 
$$
is an equivalent semi-norm  on $C_2^{p-\textnormal{var}}(I;E)$.  

For an arbitrarily given two-index map $g: \Delta_{I}\rightarrow \bbR$,  define the increment operator
$$
\delta g_{s \theta t} = g_{st} - g_{\theta t} - g_{s \theta}, \quad    (s, \theta , t)\in \Delta^2_I.
$$
We will now give the definition of a rough path and  geometric rough path. We refer the reader to \cite{MR2314753, FrVi10, friz2020course} for more thorough expositions.

\begin{definition}\label{def:rough_path}
Let $K \in\bbN $ and $p\in[2,3)$. A  $p$-variation rough path is a  pair 
$$
\bZ=(Z, \mathbb{Z}) \in \clC^{p-\textnormal{var}}(I;\bbR^K):=C_2^{p-\textnormal{var}}(I;\bbR^{K}) \times C^{\frac{p}{2}-\textnormal{var}}_2 (I; \bbR^{K\times K}) 
$$
that satisfies the Chen relations 
\begin{equation*}\label{eq:chen_rel_Z}
\delta Z_{s\theta t}=0 \quad \textnormal{ and } \quad  \delta \mathbb{Z}_{s\theta t}=Z_{s\theta} \otimes  Z_{\theta t},   \quad  \forall (s, \theta , t)\in \Delta^2_I.
\end{equation*}
For an arbitrarily given $\bZ\in  \clC^{p-\textnormal{var}}([0,T];\bbR^K)$,  denote
\begin{equation}\label{def:control_of_rough_path}
\varpi_{\bZ} = \inf \{\varpi :  \varpi \textnormal{ is a control s.t. } |Z_{st}|\le \varpi(s,t) \; \textnormal{and} \; |\bbZ_{st}|\le \varpi(s,t), \;\; \forall (s,t)\in \Delta_I\}.
\end{equation}
Given a smooth path $z: I\rightarrow \bbR^K$, we define its canonical lift $\bZ=(Z, \mathbb{Z})\in \clC^{p-\textnormal{var}}([0,T];\bbR^K)$  by
$$Z_{st}=\delta z_{st}\quad \textnormal{and} \quad \mathbb{Z}_{st}:=\int_s^tZ_{s r} \otimes \rmd z_r, \;\; (s,t)\in I.$$
An element $\mathbf{Z}=(Z, \mathbb{Z}) \in \clC^{p-\textnormal{var}}(I;\bbR^K)$ is said to be geometric if it can be obtained as the limit in the  product topology of a sequence of rough paths  $\{(Z^{n},\mathbb{Z}^{n})\}_{n=1}^\infty$ 
that are canonical lifts of smooth paths $z^n:I \to \bbR^K$.
We denote by $\clC^{p-\textnormal{var}}_g(I;\bbR^K)$ the set of geometric $p$-variation rough paths and endow it with the product topology. Finally, we denote by $\clC^{p-\textnormal{var}}_g(\bbR_+;\bbR^K)$ the corresponding Fr\'echet space of pairs  $\bZ=(Z,\bbZ): \Delta_{[0,\infty)}\rightarrow \bbR^K\times \bbR^{K\times K}$ belonging to $\clC^{p-\textnormal{var}}_g([0,T];\bbR^K)$ for all positive $T$.
\end{definition}

The following lemma, referred to as the \textit{sewing lemma}, lies at the very foundation of the theory of rough paths. The proof is a straightforward modification of  \cite{DeGuHoTi16}[Lemma 1.2] or  \cite{friz2020course}[Lemma 4.2].

\begin{lemma}[Sewing lemma] \label{lem:sewing}
Let  $\varpi_1$ and $\varpi_2$ be controls on $I$. Let $L\in (0,\infty)$, $\zeta\in [0,1)$ and $p\ge \zeta$. Assume that $h\in  C_{2,\textnormal{loc}}^{p-\textnormal{var}}(I;E)$  is such  that for all $(s,u,t)\in \Delta^2_I $ with $\varpi_1(s,t) \leq L$,
$$
|\delta h_{sut}|\le \omega_2(s,t)^{\frac{1}{\zeta}}.
$$
Then there exists a unique path $\clI h\in  C^{p-\textnormal{var}}(I;E) $ such that $\clI h_a=0$, $\Lambda h:=h-\delta \clI h\in C_{2,\textnormal{loc}}^{\zeta-\textnormal{var}}(I;E)$, and for all $(s,t)\in \Delta_I $ with $\varpi_1(s,t) \leq L$,
\begin{equation}\label{ineq:sewing_lemma}
|(\Lambda   h)_{st}|\le C_{\zeta}\varpi_2(s,t)^{\frac{1}{\zeta}}.
\end{equation}
for a  universal positive constant $C_{\zeta}$. Moreover, for all $(s,t)\in \Delta_I$,
$$
\delta \clI h_{st} =\lim_{|\mathfrak{p}|\rightarrow 0} \sum_{i=1}^{\# \mathfrak{p} -1} h_{t_it_{i+1}}, 
$$
where the limit is understood as a limit of nets over finite partitions $\mathfrak{p}\in \clP([s,t])$ of the interval $[s,t]\subset I$ partially ordered by inclusion with mesh size $|\mathfrak{p}|$ tending to zero. 
\end{lemma}

\section{Main results}\label{Main Results}
\subsection{Formulations of the rough incompressible Euler system}

Let $d\in \{2,3,\ldots\}$, $K\in \bbN$, $p\in [2,3)$, and  $m\in \bbN$ be such that  $m\ge m_*:= \lfloor \frac{d}{2}\rfloor + 2$. For an arbitrarily given initial condition $u_0\in \mathring{H}_{\sigma}^{m},$ geometric rough path $\bZ=(Z,\mathbb{Z})\in \mathcal{C}^{p-\textnormal{var}}_{g}(\bbR_+;\bbR^K)$, and collection of vector fields $\xi \in (W^{m+2,\infty}_{\sigma})^K$, we consider the rough incompressible Euler system given by
\begin{equation}
\begin{cases}\label{eq:rough_velocity_unprojected}
\rmd u +u\cdot \nabla u\,\rmd t -\pounds_{\xi_k}^*u\,\rmd \bZ^k_t = -\rmd\nabla  q_t - \rmd h_t \quad \textnormal{on}\;\; (0,T]\times \bbT^d ,\\
\operatorname{div}  u=0 \quad \textnormal{on}\;\; [0,T]\times \bbT^d ,\\
\int_{\bbT^d}u\, dV=0,\quad  \int_{\bbT^d}q\,dV=0 \quad \textnormal{on} \;\;[0,T],\\
u=u_0,\quad q=0,\quad  h=0 \quad \textnormal{on} \;\;\{0\}\times \bbT^d ,
\end{cases}
\end{equation}
where (see Section \ref{sec:prelim_Lie})
$
\pounds_{\xi_k}^*u=-(\xi_k^j\partial_{x^j}u^i+u^j\partial_{x^i}\xi_k^j)e_i.
$
Equation \eqref{eq:rough_velocity_unprojected} is to be solved for an unknown divergence and mean-free vector field (velocity) $u: [0,T]\times \bbT^d\rightarrow \bbR^d$, mean-free scalar field (`time-integrated' pressure) $q:[0,T] \times \bbT^d\rightarrow \bbR$, and harmonic constant (time-integrated) $h: [0,T]\rightarrow \bbR^d$.  The pressure $q$ and harmonic constant $h$ should be understood as Lagrange multipliers associated with the divergence-free and mean-free constraints, respectively. 

In  contrast to the unperturbed  system (i.e., $\xi\equiv 0$ or $\bZ\equiv 0$), the system \eqref{eq:rough_velocity_unprojected} does not preserve the mean of the initial condition if the $\xi$'s are not constant in space due to the term  $u^j\partial_{x^i}\xi_k^j$. Indeed, upon formally integrating \eqref{eq:rough_velocity_unprojected} over $\bbT^d$, all other terms vanish due to the periodic boundary conditions. Thus, the Lagrangian multiplier $h$ is required to enforce the constraint that the velocity $u$ remains mean-free. It is worth noting that we impose the mean-free constraint because it simplifies our analysis. For details on how to avoid this assumption, we refer to \cite{hofmanova2019rough}, which establishes the existence of a strong solution of the associated viscous version of equation \eqref{eq:rough_velocity_unprojected}.

Applying the divergence and mean-free projection operator $P$ to \eqref{eq:rough_velocity_unprojected}, we find 
\begin{equation}
\begin{cases}\label{eq:rough_velocity_projected} \tag{RE}
\rmd u+P(u\cdot \nabla u)\, \rmd t -P\pounds_{\xi_k}^*u\, \rmd \bZ^k_t=0
\quad \textnormal{on}\;\; (0,T]\times \bbT^d ,\\
u=u_0 \quad \textnormal{on} \;\;\{0\}\times \bbT^d ,
\end{cases}
\end{equation}
which is to be solved for $u:[0,T]\rightarrow  \mathring{H}_{\sigma}^{m}$. We will now define the notion of solution of \eqref{eq:rough_velocity_projected} we will use throughout the paper, which is simply a specific case of Definition \ref{def:urd_eqn}. Formally, the definition can be obtained by integrating \eqref{eq:rough_velocity_projected} over an arbitrary interval and  iterating the equation into the $\rmd \bZ^k$-integral twice. In fact, in the proof of local existence  (see Section \ref{sec:local_wp}), we smooth out the path $\bZ$ and iterate in such a manner.  
  
\begin{definition}[$H^{m}$-solution]\label{def:solution_velocity}  Let $m\ge m_*$ and $\xi \in (W^{m+2,\infty}_{\sigma})^K$.
We say a bounded path $u:[0,T]\rightarrow H^m$ is a $H^{m}$-solution  of   \eqref{eq:rough_velocity_projected} on the interval $[0,T]$ if $u|_{t=0}=u_0$ and 
\begin{equation}\label{eq:def_solution_velocity}
u^{P,\natural}_{st}:=\delta u_{st}+ \int_s^tP(u_r\cdot \nabla u_r)\, \rmd r-  P\pounds_{\xi_k}^* u_sZ^k_{st} -P\pounds_{\xi_k}^*P\pounds_{\xi_l}^* u_s\mathbb{Z}^{lk}_{st}, \quad (s,t)\in \Delta_T,
\end{equation}  
satisfies $u^{P,\natural} \in C^{\frac{p}{3}-\textnormal{var}}_{2,\textnormal{loc}}([0,T];\mathring{H}_{\sigma}^{m-3})$.
A bounded path $u:[0,T]\rightarrow H^m$ is said to be a $H^{m}$-solution of  \eqref{eq:rough_velocity_projected} on the interval $[0,T)$ if $u$  is a $H^{m}$-solution   of   \eqref{eq:rough_velocity_projected} on the interval $[0,T']$ for all $0<T'<T$. 
\end{definition}
\begin{remark}
An $H^m$-solution actually possesses additional time-smoothness properties, as proven in Theorem \ref{thm:local_wp} (see, also, Theorem \ref{thm:URDRemEst}).
\end{remark}

In Section \ref{sec:pressure_recovery}, given $H^{m}$-solution $u$ of \eqref{eq:rough_velocity_projected} on an interval $[0,T]$, we  apply the sewing lemma (i.e., Lemma \ref{lem:sewing}) to  construct the rough integral  $\int_0^\cdot \pounds_{\xi_k}^*u_s\rmd \bZ_s^k\in  C^{p-\textnormal{var}}([0,T];H^{m-3})$; that is, for all $t\in [0,T],$
$$
\int_0^t \pounds_{\xi_k}^*u_s\,\rmd \bZ_s^k:=\underset{\mathfrak{p}\in \clP([0,t])}{\lim_{|\mathfrak{p}|\rightarrow 0} }\sum_{i=1}^{\# \mathfrak{p} -1}\left( \pounds_{\xi_k}^* u_{t_i}Z^k_{t_it_{i+1}} +\pounds_{\xi_k}^*P\pounds_{\xi_l}^* u_{t_i}\mathbb{Z}^{lk}_{t_it_{i+1}}\right).
$$
It follows by the continuity and linearity of the divergence and mean-free projection map $P: H^{m}\rightarrow \mathring{H}_{\sigma}^{m}$ that for all $t\in [0,T]$,
\begin{equation}\label{eq:rv_int_proj}
u_t-u_0+ \int_0^tP(u_s\cdot \nabla u_s)\,\rmd s- \int_0^tP\pounds_{\xi_k}^*u_s\, \rmd \bZ^k_{s} =0,
\end{equation}
where $P\int_0^{\cdot} \pounds_{\xi_k}^*u_s\,\rmd \bZ_s^k\in C^{p-\textnormal{var}}([0,T];\mathring{H}_{\sigma}^{m-3})$ is defined from the projected expansion appearing in the right-hand-side of \eqref{eq:def_solution_velocity}.
The pressure $q$ and harmonic constant $h$ can then be recovered using the Helmholtz decomposition $P=I-Q-H$ (see Section \ref{sec:prelim_Hodge})
$$
\nabla q_t := \int_0^t Q(u_s\cdot \nabla u_s)\, \rmd s - \int_0^t Q\pounds_{\xi_k}^*u_s\,\rmd \bZ_s^k, \quad h_t := -\int_0^t H\pounds_{\xi_k}^*u_s\,\rmd \bZ_s^k, \quad t\in [0,T].
$$

\begin{proposition}[Recovery of pressure and harmonic constant]\label{prop:pressure_and_constant_recovery}
If $u$ is a $H^{m}$-solution of \eqref{eq:rough_velocity_projected} on the interval $[0,T]$, then there exists $\int_0^\cdot \pounds_{\xi_k}^*u_sdZ_s^k\in  C^{p-\textnormal{var}}([0,T];H^{m-3})$ and uniquely determined  $q \in C^{p-\textnormal{var}}([0,T];\mathring{H}^{m-2})$ and  $h\in C^{p-\textnormal{var}}([0,T];\bbR^d)$  initiating from zero such that for all $t\in [0,T]$,
\begin{equation}\label{eq:def_solution_velocity_pressure}
u_t-u_0+ \int_0^tu_s\cdot \nabla u_s\,\rmd s- \int_0^t\pounds_{\xi_k}^*u_s\, \rmd \bZ^k_{s} =-\nabla  q_{t} - h_{t}.
\end{equation}
\end{proposition}

Using \eqref{eq:covd_Lie} and \eqref{eq:ad_star_Lie}, we find that \eqref{eq:rough_velocity_unprojected} can expressed in terms of the associated co-vector $u^{\flat}$ (see Section \ref{sec:prelim_Lie}):
\begin{equation}\label{eq:rough_one_form}
\rmd u^{\flat}+\pounds_{u}u^{\flat}\,\rmd t +\pounds_{\xi_k}u^{\flat}\, \rmd \bZ_t^k = - \bd(\rmd q_t -2^{-1} |u|^2\rmd t)   -\rmd h^{\flat}_t.
\end{equation}
Let $\omega = \bd u^{\flat}=(\partial_{x^i}u_j-\partial_{x^j}u_i)e^i\wedge e^j$.
By applying the exterior derivative $\bd$ to \eqref{eq:rough_one_form} and using that $\bd \pounds_v=\pounds_v \bd$ we arrive at the vorticity formulation: 
\begin{equation}\label{eq:vorticity_two_form_2}
\begin{cases}
\rmd \omega +\pounds_{u}\omega\,\rmd t + 
\pounds_{\xi_k}\omega\,\rmd \bZ_t^k = 0\quad \textnormal{on}\;\; (0,T]\times \bbT^d ,\\
u=\operatorname{BS}\omega\quad \textnormal{on}\;\; (0,T]\times \bbT^d,\\
\omega=\bd u_{0}^{\flat} \quad \textnormal{on} \;\;\{0\}\times \bbT^d,
\end{cases}
\end{equation}
where  for a divergence-free vector field $v$ we have (see Section \ref{sec:prelim_Lie})
\begin{align}
\pounds_v \omega&=\sum_{i<j}(v^{q}\partial_{x^{q}}\omega_{ij}+ \omega_{qj}\partial_{x^{i}}v^q +\omega_{iq}\partial_{x^{j}}v^q  )  e^i\wedge e^j\notag\\
&=\sum_{i<j}\left(\partial_{x^i}(v^q\omega_{qj})-\partial_{x^j}(v^q\omega_{qi})\right)e^i\wedge e^j=\bd [\omega(v,\cdot)].\label{eq:CartanLie}
\end{align}

Recall that  (see Section \ref{sec:prelim_Hodge}) for given $s\in \bbR$, $\mathring{H}^s_{\bd}:=\{\omega\in \mathring{H}^{s}(\bbT^d;\Lambda^{2}(\bbR^d)^*): \bd \omega =0\}$ and that $\operatorname{BS}:\mathring{H}^{m-1}_{\bd}\rightarrow \mathring{H}_{\sigma}^{m}$ denotes the inverse of $\bd \flat$. The Cartan formulation \eqref{eq:CartanLie} of the Lie derivative $\pounds_v\omega=\bd[\omega(v,\cdot)]$ implies that the dynamics preserve the property that $\bd \omega=0$ and $\hat{\omega}(0)=0$, and thus no Lagrange multipliers are needed to enforce these constraints. In contrast, the velocity equation requires the pressure and the harmonic constant to enforce the divergence-free constraint.

We summarize the equivalence between the velocity and vorticity formulation in the following proposition. The direct implication is a simple consequence of \eqref{eq:covd_Lie} and \eqref{eq:ad_star_Lie} and the converse follows from the properties of the Biot-Savart operator presented in Section \ref{sec:prelim_Hodge}.
\begin{proposition}[Vorticity formulation]\label{prop:vorticity}
If  $u$  is a $H^{m}$-solution of  \eqref{eq:rough_velocity_projected} on the interval $[0,T]$, then $\omega=\bd u^{\flat}:[0,T]\rightarrow \mathring{H}^{m-1}_{\bd}$ is bounded 
and 
\begin{equation}\label{eq:def_solution_vorticity}
\omega^{\natural}_{st}:=\delta \omega_{st}+ \int_s^t\pounds_{u_r}\omega_r\,\rmd r+  \pounds_{\xi_k}\omega_sZ^k_{st} -\pounds_{\xi_k}\pounds_{\xi_l}\omega_s\mathbb{Z}^{lk}_{st}, \;\;(s,t)\in \Delta_T,
\end{equation}
satisfies 
$\omega^{\natural}\in C^{\frac{p}{3}-\textnormal{var}}_{2,\textnormal{loc}}([0,T];\mathring{H}^{m-4}_{\bd})$. Moreover, $\omega^{\natural}=\bd \flat  u^{P,\natural}$.
Conversely, if $\omega$ and $\omega^{\natural}$ belong to the aforementioned spaces and satisfy \eqref{eq:def_solution_vorticity} with $\omega_0:=\bd u_0^{\flat}$ and $u:=\operatorname{BS}\omega$, then $u$ is $H^{m}$-solution of  \eqref{eq:rough_velocity_projected} and $u^{P,\natural}= \operatorname{BS}\omega^{\natural}$.
\end{proposition}
\begin{remark}\label{rem:vorticity_scalar_vec}
In dimension two, the vorticity $\omega$ can be identified with a scalar-valued function $\tilde{\omega}=\operatorname{curl} u=\star\omega= \omega_{12}$. Using \eqref{eq:star_commute_scalar}, we find that
\begin{equation}\label{eq:vorticity_2d}
\rmd\tilde{\omega}+u\cdot \nabla \tilde{\omega}\,\rmd t + 
\xi_k\cdot \nabla \tilde{\omega}\,\rmd \bZ_t^k = 0.
\end{equation}
The scalar transport structure then implies that $L^p$-norms of $\tilde{\omega}$ are conserved, which is used to prove global well-posedness (see \eqref{eq:vort_Lp_conserved} in Theorem \ref{thm:global2d}).
In dimension three,  the vorticity $\omega$ can be identified with a vector field $\hat{\omega}=\operatorname{curl} u=\sharp \star \omega=\omega_{23}e_1+\omega_{31}e_2+\omega_{12}e_3$. Applying \eqref{eq:star_commute_vector}, we get
$$
\rmd \hat{\omega}+ [u,\hat{\omega}]\,\rmd t + [\xi_k, \hat{\omega}]\, \rmd \bZ_t^k=0.
$$
\end{remark}
\begin{remark}\label{rem:vorticity_no_projection}
In the vorticity equation, neither Lagrange multipliers, nor projections, appear. The structure of the operators (i.e., \eqref{eq:CartanLie}) imply that the mean-free and exterior-derivative-free condition is preserved by the dynamics. This fact plays an important role in the proof of Theorem \ref{thm:solution_est} in which we obtain a priori solution estimates of \eqref{eq:rough_velocity_projected} using the vorticity formulation. More precisely, the linear hyperbolic symmetric structure of the operator in the rough part of \eqref{eq:vorticity_two_form_2} enables us to develop an `equation' for the `square' $\omega\otimes \omega$. Unlike in the stochastic setting, there is no Burkholder-Davis-Gundy inequality that can be used to estimate the rough integral, and thus we make use of the method of unbounded rough drivers \cite{BaGu15}, and more precisely  Theorem \ref{thm:URDRemEst} to obtain  remainder estimates for the equation and the `squared' equation. It is not clear how to obtain a priori estimates of \eqref{eq:rough_velocity_projected} directly due to the projection operators, or equivalently the presence of the pressure and harmonic constant. As a consequence, altering the structure of the operator appearing in the $d\bZ$-term in \eqref{eq:rough_velocity_unprojected} even in a multiplicative way directly impacts the structure of the  vorticity equation and prevents us from obtaining a priori solution estimates.
\end{remark}

\subsection{Statement of the main results}
Recall that we always work under the assumption  $m\ge m_*:=\lfloor \frac{d}{2}\rfloor + 2$. Our first main results establish local well-posedness and the existence of a maximally extended solution  in $H^m$ of \eqref{eq:rough_velocity_projected}. Recall that $\varpi_{\bZ}$ is the control of $\bZ$, which is defined in \eqref{def:control_of_rough_path}.

\begin{theorem}[Local well-posedness]\label{thm:local_wp} 
Assume that $u_0\in \mathring{H}^m_{\sigma}$ and $\xi \in (W_{\sigma}^{m+2})^K$. For all $T_*>0$ satisfying
\begin{equation}\label{ineq:time_local_ex}
e^{C_1(1+\varpi_{\bZ}(0,T_*))}T_*<|u_{0}|_{H^{m_*}}^{-1}, 
\end{equation}
where $C_1=C_1(p,d, |\xi|_{W^{m_*+2,\infty}})$ depends in an increasing way on $|\xi|_{W^{m_*+2,\infty}}$,
there exists a unique $H^{m}$-solution  $u\in C_w([0,T_*];\mathring{H}_{\sigma}^{m})\cap C^{p-\textnormal{var}}([0,T_*];\mathring{H}^{m-1}_{\sigma})$
of   \eqref{eq:rough_velocity_projected} on the interval $[0,T_*]$ satisfying
\begin{align}\label{ineq:m_star_soln_est}
\sup_{t\le T_*}|u_t|_{H^{m_*}}&\le \frac{e^{C_1(1+\varpi_{\bZ}(0,T_*))}}{|u_{0}|_{H^{m_*}}^{-1}-e^{C_1(1+\varpi_{\bZ}(0,T_*))} T_*}.
\end{align}
Moreover, if $m>m_*$, there is a constant $C_2=C_2(p,d,m, |\xi|_{W^{m+2,\infty}})$ which increases with $|\xi|_{W^{m_*+2,\infty}}$, such that 
\begin{equation}\label{ineq:m_soln_est}
\sup_{t\le T_*}|u_t|_{H^{m}} \le \sqrt{2} \exp\left(C_2 \left(\int_{0}^{T_*}|\nabla u_s|_{L^\infty}\,\rmd s
+\varpi_{\bZ}(0,T_*)\right) \right)|u_{0}|_{H^{m}}.
\end{equation}
If $\xi \in (W_{\sigma}^{m+4,\infty})^K$, then $u\in C([0,T_*];\mathring{H}_{\sigma}^{m})$.
\end{theorem}
\begin{corollary}[Maximally extended solution]\label{cor:maximal_solution}
Assume that $u_0\in \mathring{H}^m_{\sigma}$ and $\xi \in (W_{\sigma}^{m+2})^K$. Then there exists a unique maximally extended $H^m$-solution $u\in C_w([0,T_{\textnormal{max}});\mathring{H}_{\sigma}^{m})\cap C^{p-\textnormal{var}}([0,T_{\textnormal{max}});\mathring{H}^{m-1}_{\sigma})$ of \eqref{eq:rough_velocity_projected} on the interval $[0,T_{\textnormal{max}})$. The time $T_{\textnormal{max}}$ is uniquely specified by the property if $T_{\textnormal{max}}<\infty$, then $\limsup_{t\uparrow T_{\textnormal{max}}} |u_t|_{H^{m_*}}=\infty$. If $\xi \in (W_{\sigma}^{m+4, \infty})^K$, then $u\in C([0,T_{\textnormal{max}});\mathring{H}_{\sigma}^{m})$.
\end{corollary}

The next theorem extends the Beale-Kato-Majda (BKM) blow-up criterion \cite{beale1984remarks}.

\begin{theorem}[BKM blow-up criterion]\label{thm:BKM}
Assume that $u_0\in \mathring{H}^m_{\sigma}$ and  $\xi \in (W_{\sigma}^{m_*+4,\infty})^K$ if $m=m_*$ and $\xi \in (W_{\sigma}^{m+2})^K$ if $m>m_*$. Let $u$ denote the maximally extended $H^{m}$-solution of \eqref{eq:rough_velocity_projected} and  $\omega =\bd u^{\flat}$ denote its vorticity. Then there are constants  $C_1=C_1(d,m)$ and $C_2=C_2(p,d,m, |\xi|_{W^{m+2,\infty}})$ depending in an increasing way on $|\xi|_{W^{m+2,\infty}}$ such that for all $T<T_{\textnormal{max}}$,
\begin{equation}\label{ineq:BKM_blowup_bound}
\sup_{t\le T}|u_t|_{H^{m}} \le C_1 (1+|u_{0}|_{H^{m}})\exp\left(C_2(1+\varpi_{\bZ}(0,T))\exp \left(C_2 \int_{0}^T|\omega_s|_{L^\infty}\,\rmd s\right)\right).
\end{equation}
Moreover,  $T_{\textnormal{max}}<\infty$ if and only if  $\int_0^{T_{\textnormal{max}}} |\omega_t|_{L^\infty} \,\rmd t =\infty$. 
\end{theorem}

In dimension two, we obtain global well-posedness. 

\begin{theorem}[Global well-posedness in $2d$]\label{thm:global2d} Let $d=2$. Assume that $u_0\in \mathring{H}^m_{\sigma}$ and  $\xi \in (W_{\sigma}^{m+2})^K$. Let $u$ denote the maximally extended $H^{m}$-solution and  $\tilde{\omega} =\operatorname{curl} u$ denote its scalar vorticity. Then $T_{\textnormal{max}}=\infty$ and for all $t\ge 0$ and $p\in [2,\infty]$,
\begin{equation}\label{eq:vort_Lp_conserved}
|\tilde{\omega}_t|_{L^p}=|\tilde{\omega}_0|_{L^p}.
\end{equation}
Moreover, there are constants $C_1=C_1(m)$ and $C_2=C_2(p,m, |\xi|_{W^{m+2,\infty}})$  that increase with $|\xi|_{W^{m+2,\infty}}$ such that  for all $t\ge 0$
\begin{equation}\label{ineq:double_exp_2d}
|u_t|_{H^{m}} \le C_1 (1+|u_{0}|_{H^{m}})\exp\left(C_2(1+\varpi_{\bZ}(0,t))\exp \left(C_2|\tilde{\omega}_0|_{L^\infty}t\right)\right).
\end{equation}
\end{theorem}

The following corollary establishes the continuity of the solution map with respect to the data.

\begin{corollary}[Continuous dependence on data]\label{cor:stability}
Assume that $u_0\in \mathring{H}^m_{\sigma}$ and $\xi \in (W_{\sigma}^{m+2})^K$. Further, assume that $
\{\bZ^n\}_{n=1}^\infty$ converges to $\bZ$ in $\mathcal{C}^{p-\textnormal{var}}_{g}$ and that $\{(u_0^n, \xi^n)\}_{n=1}^\infty $ is bounded in $\mathring{H}_{\sigma}^{m} \times W^{m+2,\infty}_{\sigma}$ and converges to $(u_0, \xi)$ in $\dot{L}^{2}_{\sigma}\times (W^{2,\infty}_{\sigma})^K$. Denote by $u$ and  $\{u^n\}_{n=1}^\infty$ the  maximally extended solutions corresponding to the data $(u_0,\xi,\bZ)$ and $\{(u_0^n,\xi^n,\bZ^n)\}_{n=1}^\infty$, and let $(q,h)$ and  $\{(q^n,h^n)\}_{n=1}^\infty$ denote the associated  pressures and harmonic constants. 
\begin{itemize}
\item If $d=2$, then $\{u^n\}_{n=1}^\infty$ converges to $u$ in $C([0,\infty); \mathring{H}^{m-\epsilon}_{\sigma})$ for any $\epsilon>0$ and in the weak-star topology of $L^\infty([0,\infty); \mathring{H}_{\sigma}^{m})$. Moreover,  $\{(q^n,h^n)\}_{n=1}^\infty$  converge to $(q,h)$ in  $C([0,\infty);\mathring{H}^{m-2-\epsilon}) \times  C([0, \infty);\bbR^d)$ for any $\epsilon>0$.
\item If $m>m_*$, then for all $T<T_{\operatorname{max}}$ there exists an $N(T)\in \bbN$ such that $\{u^n\}_{n=N(T)}^\infty$ converges to $u$ in $C([0,T]; \mathring{H}^{m-\epsilon}_{\sigma})$ for any $\epsilon>0$ and in the weak-star topology of $L^\infty([0,T]; \mathring{H}_{\sigma}^{m})$. Moreover,  $\{(q^n,h^n)\}_{n=N(T)}^\infty$  converge to $(q,h)$ in  $C([0,T];\mathring{H}^{m-2-\epsilon}) \times  C([0,T];\bbR^d)$ for any $\epsilon>0$.
\end{itemize}
\end{corollary}
\begin{remark}[Rough Navier-Stokes and inviscid limit]\label{rem:Navier-Stokes}
Let $\nu>0$.  In \cite{hofmanova2019rough}, two of the authors considered  the rough Navier-Stokes system given by
\begin{equation*}
\begin{cases}
\rmd u +u\cdot \nabla u\,\rmd t - \pounds_{\xi_k}^*u\,\rmd Z^k_t = \nu \Delta u-\rmd\nabla  q_t - \rmd h_t,\quad \textnormal{on}\;\; [0,T]\times \bbT^d ,\\
\operatorname{div}  u=0 \quad \textnormal{on}\;\; [0,T]\times \bbT^d ,\\
\int_{\bbT^d}u\, dV=0,\quad  \int_{\bbT^d}q\,dV=0 \quad \textnormal{on} \;\;[0,T],\\
u=u_0,\quad q=0,\quad  h=0 \quad \textnormal{on} \;\;\{0\}\times \bbT^d.
\end{cases}
\end{equation*}
for $d\in \{2,3\}$. We showed that for arbitrarily given $u_0\in \mathring{H}^{1}_{\sigma}$, there exists  a time $T_*=T_*(d,\varpi_{\bZ},|\xi|_{W^{3,\infty}})$ and a solution $u\in L^2([0,T_*];\mathring{H}^2_{\sigma}) \cap L^\infty([0,T_*];\mathring{H}^{1}_{\sigma})$. In fact, the solution was not constrained to be mean-free in \cite{hofmanova2019rough}, but the proof goes through in a simpler manner. In dimension two, we also proved that there is a unique global solution of the Navier-Stokes system. With minor changes in the present paper, we can derive the existence and uniqueness of a maximally extended solution $u \in L^2([0,T_{\textnormal{max}});\mathring{H}^{m+1}_{\sigma}) \cap C([0,T_{\textnormal{max}});\mathring{H}^{m}_{\sigma})$ in any dimension $d\ge 2$. Moreover, we can show that  if $d=2$ or $m>m_*$, then for an arbitrarily given $\{\nu^n\}_{n=1}^\infty$ converging to zero, the corresponding sequence of Navier-Stokes solutions $\{u^n\}_{n=1}^\infty$ converges to the Euler solution in  $C([0,T_{\textnormal{max}}); \mathring{H}^{m-\epsilon}_{\sigma})$ for any $\epsilon>0$ and in the weak-star topology of $L^\infty([0,T_{\textnormal{max}}); \mathring{H}_{\sigma}^{m})$. The inviscid limit  stochastic Naiver-Stokes equations with additive and multiplicative noise has been studied in  \cite{glatt2015inviscid, breit2018stochastically}.
\end{remark}

\subsubsection{Applications to stochastic partial differential equations}\label{sec:main_res_state_wong_zakai}
In what follows, we will discuss the Wong-Zakai approximation of the Euler stochastic partial differential equation (SPDE) driven by Brownian motion in dimension two.
Let $B=\{B^k\}_{k=1}^K$ denote  a collection of $K$-independent Brownian motions adapted to a filtered probability space  $(\clZ,\clF, \bbF=\{\clF_t\}_{t\ge 0},\bbP)$  satisfying the usual conditions.  In \cite{lang2022well}, it was shown that for every  $\clF_0$-adapted initial velocity $u_0\in \mathring{H}^{m-1}_{\sigma}$, there exists a unique $\bbF$-adapted process  $\bar{u} \in C([0,\infty);\mathring{H}_{\sigma}^{m})$ such that $\bbP$-a.s.\ for all $t\in [0,\infty)$,
\begin{equation}\label{eq:SPDE}
\bar{u}_{t}-u_0+ \int_0^{t}P(\bar{u}_s\cdot \nabla \bar{u}_s)\,\rmd s - \int_0^{t} P\pounds_{\xi_k}^*\bar{u}_s \circ \rmd B_s^k =0,
\end{equation}
where equality is understood in $L^2$ and the stochastic integral is understood in the Stratonovich sense. 

By Proposition 3.5 of \cite{friz2020course}, $\bbP$-a.s., $\bB=(\delta B,\bbB^{\operatorname{strat}})\in  \clC_g^{p-\operatorname{var}}(\bbR_+;\bbR^K)$, where for each $l,k\in \{1,\ldots, K\}$ and $(s,t)\in \Delta_{[0,\infty)}$:
$$
\bbB^{\operatorname{strat};lk}_{st}:=\int_s^t(B_r^l-B_s^l)\circ \rmd B_{r}^k.
$$
By Proposition 3.6 of \cite{friz2020course}, the canonical lift of a dyadic piecewise-linear approximation $\{B^{n}\}_{n=1}^\infty$ of the Brownian motion $B$, denoted by $\{\bB^{n}\}_{n=1}^\infty=\{(B^{n},\bbB^{n})\}_{n=1}^\infty\in  \clC_g^{1-\operatorname{var}}(\bbR_+;\bbR^K)$, converges $\bbP$-a.s.\ to $\bZ$ in $\clC_g^{p-\operatorname{var}}(\bbR_+;\bbR^K)$.

By Corollary \ref{thm:global2d}, $\bbP$-a.s., corresponding to the data $(u_0,\bB,\xi)$ and $\{(u_0, \bB^n,\xi)\}_{n=1}^\infty$, there exists unique solutions $u,\{u^n\}_{n=1}^\infty\in C([0,\infty),\mathring{H}^{m}_{\sigma})$ of \eqref{eq:vorticity_2d}. Owing to Corollary \ref{cor:stability},  the sequence  $\{u^n\}_{n=1}^\infty$ converges to $u$ in  $C([0,T_{\textnormal{max}}); \mathring{H}^{m-\epsilon}_{\sigma})$ for any $\epsilon>0$  and in the weak-star topology of $L^\infty([0,\infty); \mathring{H}_{\sigma}^{m})$.

For every $t\in \bbR_+$, the map $\bZ|_{[0,t]} \in \clC_g^{p-\textnormal{var}}([0,t];\bbR^K)\mapsto u\in C([0,t],\mathring{H}^{m-}_{\sigma})$ is continuous and the map $\zeta\in \clZ\mapsto  \bB(\zeta)|_{[0,t]} \in  \clC_g^{p-\textnormal{var}}([0,t];\bbR^K)$ is measurable. Thus, we conclude that the composition of the two maps is measurable, and hence that the solution $u$ is $\bF$-adapted. As explained above (see \eqref{eq:rv_int_proj}), $\bbP$-a.s., for all $t\in [0,T]$, we have
$$
u_t-u_0+ \int_0^tP(u_r\cdot \nabla u_r)\,\rmd r-  \int_0^tP\pounds_{\xi_k}^*u_s\, \rmd \bB^k_{s} =0.
$$
It can be shown  (see, e.g., \cite{friz2020course}[Corollary 5.2]) that $\bbP$-a.s.\ for all $t\in [0,T]$,
$$
\int_0^tP\pounds_{\xi_k}^*u_s\, \rmd \bB^k_{s}= \int_0^tP\pounds_{\xi_k}^*u_s \circ \, \rmd B^k_{s}.
$$
Therefore, we  obtain the following Wong-Zakai approximation result. 

\begin{theorem}[Wong-Zakai approximation]\label{thm:Wong-Zakai}
The stochastic process $u$ is indistinguishable from $\bar{u}$, and 
$\bbP$-a.s., $\{u^{n}\}_{n=1}^\infty$ converges to $u$ in $C([0,\infty); \mathring{H}^{m-\epsilon}_{\sigma})$ for any $\epsilon>0$ and in the weak-star topology of $L^\infty([0,\infty); \mathring{H}_{\sigma}^{m})$.
\end{theorem}

\begin{remark}[Strook-Varadhan support theorem, large deviations principle, and random dynamical system]
As in \cite{friz2020course}[Section 9.3], using the continuity of the solution map, one can characterize the support of the law of the SPDE \eqref{eq:SPDE} in  $C([0,T_{\textnormal{max}}); \mathring{H}^{m-\epsilon}_{\sigma})$ for any $\epsilon>0$ in terms of Cameron-Martin space and prove a large deviations principle (making use of contraction principle) for
$$
\rmd u+P(u\cdot \nabla u)\,\rmd t -  \epsilon P\pounds_{\xi_k}^*u \circ \rmd B_t^k =0,
$$
which concerns small-noise deviations about solutions of the Euler system
$$
\partial_t u+P(u\cdot \nabla u)=0.
$$
Moreover, if  the driving rough path $\bZ:\clZ\rightarrow \clC_g^{p-\textnormal{var}}$ is a continuous $p$-rough
path co-cycle\footnote{This is the appropriate notion of random shifts in the rough path which enables the construction of a random dynamical system.}, viz
$$
Z_{0,s+t}(\zeta) = Z_{0,s}(\zeta) + Z_{0,t}(\theta_s \zeta) , \quad \mathbb{Z}_{0,s+t}(\zeta) = \mathbb{Z}_{0,s}(\zeta) + \mathbb{Z}_{0,t}( \theta_s \zeta) + Z_{0,s}(\zeta) \otimes Z_{0,t}(\theta_s \zeta), \;\;\forall (s,t)\in \Delta_T,
$$
where $\theta_s$ is the time-shift $\theta_s \zeta_t = \zeta_{t+s} - \zeta_s$, then the $2D$-system \eqref{eq:rough_velocity_projected} generates a continuous random dynamical system on   $C([0,T_{\textnormal{max}}); \mathring{H}^{m-\epsilon}_{\sigma})$ for any $\epsilon>0$. See\cite{hofmanova2019rough}. 
\end{remark}
\begin{remark}[Gaussian rough paths]
Our main results also yield a solution theory in any dimension $d\in \{2,3,\dots, \}$ for a class of Euler SPDEs driven by fractional Brownian or more general Gaussian processes transport noise. We refer the reader to \cite{friz2020course} for more details about the lifts of Gaussian processes to the space of geometric rough paths. In the introduction, we have discussed the potential applications of such models to stochastic parameterizations of sub-grid scales of ideal fluids.
\end{remark}

\subsubsection{Critical points of the Clebsch and Hamilton-Pontryagin variational principles}\label{sec:main_res_state_critical}
For an arbitrarily given Fr\'echet space $E$ and time $T>0$, denote by $\clD_{Z}([0,T];E)$  the space of of $E$-valued $\bZ$-controlled rough paths on the interval $[0,T]$. We refer to \cite{friz2020course} for precise definitions. In \cite{crisan2022variational}, we introduced the Clebsch action functional 
$$
S^{\textnormal{Clb}_{\bZ}}(u,(\lambda^q)_{q=1}^d,(a^q)_{q=1}^d) =\int_0^T \frac{1}{2}|u_t|^2_{L^2}\,
\rmd t + \sum_{q=1}^d\left(\lambda_t^q,\rmd a_t^q + \pounds_{u_t}a_t^q\,\rmd t + \pounds_{\xi_k} a_t^q\,\rmd \bZ_t^k\right)_{L^2},
$$
defined for $u\in C^{p-\textnormal{var}}([0,T]; \mathring{C}^\infty_{\sigma})$ and $(\lambda^q)_{q=1}^d,(a^q)_{q=1}^d\in \clD_{\bZ}([0,T];C^\infty)$.
We showed that critical points are characterized by the following system of RPDEs:
\begin{equation*}
\begin{cases}
\rmd u -  P\pounds_{u}^* u\rmd t - P\pounds_{\xi_k}^*u\rmd \bZ_t^k =0, \;\;
u=P(\sum_{q=1}^d a^q\nabla \lambda^q), \\
\rmd \lambda +\pounds_{u} u \rmd t + \pounds_{\xi_k} \lambda \rmd \bZ_t^k=0,\\
\rmd a + \pounds_{u}a\rmd t + \pounds_{\xi_k} a\rmd \bZ_t^k=0.
\end{cases}
\end{equation*}
We also introduced the  Hamilton-Pontryagin action functional
\begin{equation}\label{def:HP_functional}
S^{\textnormal{HP}_{\bZ}}(u, \phi,\lambda) =\int_0^T \frac{1}{2}|u_t|^2_{L^2}
\,\rmd t + \left( \lambda_t,\rmd \phi_t \circ \phi_t^{-1}- u_t\,\rmd t - \xi_k \,\rmd \bZ_t^k\right)_{L^2},
\end{equation}
defined for $$(u,\lambda, \phi)\in C^{p-\textnormal{var}}([0,T]; \mathring{C}^\infty_{\sigma})\times \clD_{\bZ}([0,T];\mathring{C}_{\sigma}^\infty)\times \operatorname{Diff}_{\bZ}([0,T];\bbT^d).$$ The space  $\operatorname{Diff}_{\bZ}([0,T];\bbT^d)$ consists of all rough flows (i.e., diffeomorphisms on $\bbT^d$)  $\{\phi_t\}_{t\in [0,T]} $ (see  \cite{crisan2022variational}) of the form
\begin{equation*}
\begin{cases}
\rmd \phi = v\circ \phi\, \rmd t  + \sigma_k\circ \phi \,\rmd \bZ_t^k, \;\;t\in (0,T],\\
\phi_{0}=\operatorname{id},
\end{cases}
\end{equation*}
for some $(v,\sigma)\in C^{p-\textnormal{var}}([0,T];\mathring{C}^\infty_{\sigma})\times C^\infty([0,T];(C^\infty_{\sigma})^K)$. Moreover, for $\phi \in \operatorname{Diff}_{\bZ}([0,T];\bbT^d)$, the integral in \eqref{def:HP_functional} is defined by
$$
\int_0^T\langle \lambda_t,\rmd \phi_t \circ \phi_t^{-1}\rangle:=\int_0^T (\lambda_t,v_t)_{L^2}\,\rmd t + (\lambda_t,\sigma_t)_{L^2}\, \rmd \bZ_t^k.
$$
We refer the reader to \cite{crisan2022variational} for the specifics of how variations of $\operatorname{Diff}_{\bZ}([0,T];\bbT^d)$ are defined. In \cite{crisan2022variational}, we showed that if $\bZ$ is truly rough (see \cite{friz2020course}[Definition 6.3), then $(u,\phi,\lambda)$ is a critical point of $S^{\textnormal{HP}_{\bZ}}$ if and only if 
\begin{equation*}
\begin{cases}
\rmd u - P\pounds_{u}^* u\rmd t- P\pounds_{\xi_k}^*u\rmd \bZ_t^k =0, \;\;
u=\lambda, \\
\rmd \phi = u\circ \phi \,\rmd t  + \xi_k\circ \phi \,\rmd \bZ_t^k.
\end{cases}
\end{equation*}

In \cite{crisan2022variational}, we showed that critical points satisfy a Kelvin circulation balance law: for an arbitrarily given smooth closed curve $\gamma\subset \bbT^d$ and all $t\in [0,T]$,
$$
\rmd \oint_{\phi_t (\gamma)} u^{\flat}_t= \oint_{\phi_t (\gamma)} (\rmd u^{\flat} _t+\pounds_{u_t}u_t^{\flat}\,\rmd t+\pounds_{\xi_k}u_t^{\flat}\,\rmd \bZ^k_t)=0.
$$

Assume that $u_0\in \mathring{C}^\infty_{\sigma}$ and $\xi \in (\mathring{C}^\infty_{\sigma})^K$.  Denote by $u\in C^{p-\textnormal{var}}([0,T_{\textnormal{max}});\mathring{C}^\infty_{\sigma})$  the unique maximally extended $H^{\infty}$-solution of \eqref{eq:rough_velocity_projected} and let $T<T_{\textnormal{max}}$. Let  $\{\phi_t\}_{t\in [0,T]} \in \operatorname{Diff}_{\bZ}([0,T];\bbT^d)$  denote the rough flow  (see, e.g., \cite{bailleul2017random, friz2020course, crisan2022variational}) satisfying 
\begin{equation*}
\begin{cases}
\rmd \phi = u\circ \phi\, \rmd t  + \xi_k\circ \phi \,\rmd \bZ_t^k, \;\;t\in (0,T],\\
\phi_{0}=\operatorname{id}.
\end{cases}
\end{equation*}
It follows that $(u,\phi,u)$ is a critical point of  $S^{\textnormal{HP}_{\bZ}}$ provided $\bZ$ is truly rough.

To show that we can construct a critical point of $S^{\textnormal{Clb}_{\bZ}}$, we must find $\lambda$ and $a$ such that the so-called Clebsch representation $u^{\flat}=P(\sum_{q=1}^d a^q\bd\lambda^q)$ holds. Let $\lambda=\phi^{-1}$ and $a = u_0\circ \phi^{-1}$. By Theorem 3.3 in \cite{crisan2022variational}, we have
\begin{equation*}
\begin{cases}
\rmd \lambda +\pounds_{u} u \rmd t + \pounds_{\xi_k} \lambda \rmd \bZ_t^k=0,\\
\rmd a + \pounds_{u}a\rmd t + \pounds_{\xi_k} a\rmd \bZ_t^k=0.
\end{cases}
\end{equation*}
Proceeding as in \cite{drivas2020circulation}[Lemma 3] and applying the product rule for geometric rough paths (see, e.g., \cite{friz2020course, crisan2022variational}), it follows that $v  = \sum_{q=1}^d a^q\bd \lambda^q$ satisfies  the linear RPDE
$$
\rmd v  -\pounds_{u}^*v \, \rmd t -\pounds_{\xi_k}^* v  \,\rmd \bZ_t^k =0.
$$
By virtue of  identity \eqref{eq:projection_Lie}, we find that $\bar{v}:=Pv \in C^{p-\textnormal{var}}([0,T];\mathring{H}^m_{\sigma})$ is a solution of the constrained (to the space of divergence and mean-free vector-fields) linear RPDE given by
\begin{equation}\label{eq:proj_lin}
\rmd \bar{v} - P\pounds^*_{v}  \bar{v} \, \rmd t - P\pounds_{\xi_k}^*  \bar{v}  \,\rmd Z_t^k =0.
\end{equation}
As in Proposition \ref{prop:vorticity}, using the operators $\bd \flat$ and $\operatorname{BS}$ we can establish an equivalence between solutions of \eqref{eq:proj_lin} and solutions of the unconstrained linear RPDE
$$
\rmd \bar{\omega} + \pounds_{u} \bar{\omega} \,\rmd t + \pounds_{\xi_k} \bar{\omega}\, \rmd \bZ_t^k =0.
$$
Thus, by the uniqueness of solutions of unconstrained linear RPDEs (see, e.g., \eqref{ineq:diff_est_H1} in Theorem \ref{thm:diff_est} or \cite{BaGu15}), we can deduce uniqueness of solutions of \eqref{eq:proj_lin}. Since $u$ also solves \eqref{eq:projection_Lie}, and hence \eqref{eq:proj_lin}, we deduce that 
\begin{equation}\label{eq:WeberFormula}
u= P\left(\sum_{q=1}^d a^q\nabla \lambda^q\right)=P(\nabla \phi^{-1} u_0\circ \phi^{-1})  \quad \Leftrightarrow  \quad u^{\flat} =P\phi_*u_0^{\flat},
\end{equation}
where with abuse of notation we denote by $P\in \clL(H^s(\bbT^d;(\bbR^d)^*); H^s_{\delta}(\bbT^d;(\bbR^d)$ the corresponding projection associated with the one-form Hodge decomposition (see Section \ref{sec:prelim_Hodge}). It follows that  $(u,\lambda, a)$ is a critical point of $S^{\textnormal{Clb}_{\bZ}}$. The above representation extends the well-known Weber formula (see, e.g., \cite{constantin2017analysis}). Since the exterior derivative  $\bd$ is natural with the push-forward, we also obtain the following representation of the vorticity two-form
$
\omega =\phi_*\omega_0.
$

\section{A priori  estimates}\label{sec:apriori}

The goal of this section is to establish a priori estimates of  remainders and solutions of \eqref{eq:rough_velocity_projected}. 

\subsection{Remainder estimates}
\begin{theorem}[Remainder estimates]\label{thm:remainder_est}
Let $u$ be an $H^{m}$-solution of \eqref{eq:rough_velocity_projected} on the interval $[0,T]$ and $\varpi_{\mu}(s,t):=\int_s^t|\nabla u_r|_{L^\infty}|u_r|_{H^m}\,\rmd r$, $(s,t)\in \Delta_T$.
Then there exists a constant $C=C(p,d,m, |\xi|_{W^{m+2,\infty}})$ which increases with $|\xi|_{W^{m+2,\infty}}$ such that for all $(s,t)\in \Delta_T$ with $C\varpi_{\bZ}(s,t)\le 1$, it holds that
$$
|u^{P,\natural}|_{\frac{p}{3}-\textnormal{var};[s,t];H^{m-3}}^{\frac{p}{3}}\le C \left(\sup_{s\le r\le t}|u_r|_{H^{m}} \varpi_{\bZ}(s,t)^{\frac{3}{p}} +\varpi_{\mu}(s,t)\varpi_{\bZ}(s,t)^{\frac{1}{p}}\right).
$$
Moreover, for all $(s,t)\in \Delta_T$ with $C\varpi_{\bZ}(s,t)+C\varpi_{\mu}(s,t)\,\rmd r\le 1$, it holds that
\begin{align*}
|u^{P,\sharp}|_{\frac{p}{2}-\textnormal{var};[s,t];H^{m-2}}^{\frac{p}{2}}&\le  C\left( \varpi_{\mu}(s,t)+ \sup_{s\le r\le t}|u_r|_{H^{m}}\varpi_{\bZ}(s,t)^{\frac{2}{p}}\right)\\
|u|_{p-\textnormal{var};[s,t];H^{m-1}}^p&\le C \left(\varpi_{\mu}(s,t)+ \sup_{s\le r\le t}|u_r|_{H^{m}}\left(\varpi_{\mu}(s,t)^{\frac{1}{p}}+\varpi_{\bZ}(s,t)^{\frac{1}{p}}\right)\right),
\end{align*}
where 
$$
u^{P,\sharp}_{st}:=\delta u_{st} -   P\pounds_{\xi_k}^*u_sZ ^k_{st}=-\int_s^tP(u_r\cdot \nabla u_r)\,\rmd r+P\pounds_{\xi_k}^*P\pounds_{\xi_l}^*u_s\mathbb{Z}^{lk}_{st}+u^{P,\natural}_{st}.
$$
\end{theorem}
\begin{proof}
By Proposition \ref{prop:vorticity}, $\omega=\bd u^{\flat} :[0,T]\rightarrow \mathring{H}^{m-1}_{\bd})$ is bounded 
and 
\begin{equation}\label{eq:vorticity_in_apriori}
\omega^{\natural}_{st}:=\bd \flat  u^{P,\natural}=\delta \omega_{st}+ \int_s^t\pounds_{u_r}\omega_r\,\rmd r+  \pounds_{\xi_k}\omega_sZ^k_{st} -\pounds_{\xi_k}\pounds_{\xi_l}\omega_s\mathbb{Z}^{lk}_{st}, \;\;(s,t)\in \Delta_T,
\end{equation}
satisfies 
$\omega^{\natural}\in C^{\frac{p}{3}-\textnormal{var}}_{2,\textnormal{loc}}([0,T];\mathring{H}^{m-4}_{\bd})$. We recall that that for a vector-field $v\in H^{m},$
$$
\pounds_v \omega=\sum_{i<j}(v^{q}\partial_{x^{q}}\omega_{ij}+ \omega_{qj}\partial_{x^{i}}v^q +\omega_{iq}\partial_{x^{j}}v^q  )  e^i\wedge e^j.
$$ 
Moreover,
$$
\omega^{\sharp}_{st}=\bd \flat u^{P,\sharp} =\delta \omega_{st} +   \pounds_{\xi_k}\omega_sZ^k_{st}=-\int_s^t\pounds_{u_r} \omega_r\,\rmd r+\pounds_{\xi_k}\pounds_{\xi_l}\omega_s\mathbb{Z}^{lk}_{st}+\omega^{\natural}_{st}.
$$

The strategy of the proof is as follows. We form a system of equations for $\omega$ and its derivatives up to order $m-1$ of the form Definition \ref{def:urd_eqn} and then apply Theorem  \ref{thm:URDRemEst} and \eqref{ineq:curl_equiv}.  Let $\clI_{m-1}=\{\emptyset\}\cup \cup_{n=1}^{m-1}\{1,\ldots, d\}^n$.  For $I=\emptyset$, let $|I|=0$ and for given $I=(i_1,\ldots, i_n)\in \clI_{m-1}$  define $|I|=n$. For given $I=(i_1,\ldots, i_n)\in \clI_{m-1}$, define $\partial^{I}=\partial_{x^{i_1}}\circ \cdots \circ \partial_{x^{i_n}}$. Let $A_{d,m-1}=\oplus_{n=0}^{m-1}(\Lambda^2(\bbR^{d})^*)^{\otimes n}.$

For given $n\in \{0,1,2,3\}$, let $E_n= H^{n}(\bbT^d;A_{d,m-1})$. It follows that $(E_n)_{n=0}^3$  is a scale with a smoothing (see Section \ref{sec:prelim_func}) $J^{\eta}=P_{\le \lfloor \eta^{-1}\rfloor}$, $\eta \in (0,1)$, in the sense of Definitions \ref{def:scale} and \ref{def:smooth_op}.
For a given function $\Phi \in E_n$, $i,j\in \{1,\ldots, d\}$ with $i<j$, and $I\in \clI_{m-1}$, denote by $\Phi^{I}_{ij}$  its $(I,ij)$-th component.  

Define $\Omega\in L^\infty([0,T]; E_{-0})\cap  C([0,T]; E_{-3})$ by $\Omega^{I}=\partial^I \omega$ and $\Omega^{\natural} \in C^{\frac{p}{3}-\textnormal{var}}_{2,\textnormal{loc}}([0,T];E_{-3}) $ by $\Omega^{I}=\partial^I \omega^{\natural}$ for $I\in \clI_{m-1}$.

Applying the weak derivative operators $\partial^I$  to  \eqref{eq:vorticity_in_apriori}, we find that  for all $(s,t)\in \Delta_T$ and $\Phi\in E_3$,
$$
\langle \Omega^{\natural}_{st}, \Phi \rangle =\langle \delta \Omega_{st}, \Phi\rangle + \langle \mu_{st}, \Phi \rangle+ \langle \Omega_s, [A^{1,*}_{st}+A^{2,*}_{st}]\Phi\rangle ,
$$
where
\begin{itemize}
\item $\langle \mu_{st}, \Phi\rangle :=\sum_{I\in \clI_{m-1}}\int_s^t\left( (\partial^I[\pounds_{u_r}\omega_r] - u_r^q\partial_{x^q} \partial^{I}\omega_r ,\Phi^I)_{L^2} +(u_r^q \partial^{I}\omega_r ,\partial_{x^q}\Phi^I)_{L^2}\right) \,\rmd r$,
\item $A^{1}_{st}\Phi:=\left(\xi_k^q \partial_{x^q}\Phi+ \nu^{m}_k\Phi\right) Z^k_{st}$,
\item $A^{2}_{st}\Phi:=-\left(\left(\xi_k^q \partial_{x^q} + \nu^{m}_k\right)\left(\xi_l^q \partial_{x^q} + \nu^{m}_l\right)\Phi\right)\bbZ^{lk}_{st}$,
\end{itemize}
and where for each $k\in \{1,\ldots, K\}$, $\nu_k^{m} : \bbT^d \rightarrow \clL(A_{d,m-1};A_{d,m-1})$ is implicitly and inductively defined as follows:   for $i,j\in \{1,\ldots, d\}$, 
$$(\nu_k^{m}\Phi)^{ \emptyset}_{ij}=\Phi^{\emptyset}_{qj}\partial_{x^i}\xi_k^{q}+\Phi^{\emptyset}_{iq}\partial_{x^j}\xi_k^{q},$$ and for $I\in \clI_{m-1}$ with $|I|<m-1$,
\begin{align*}
\partial_{x^{l}} [\xi_k^q\partial_{x^q}\Phi +\nu_k^{m} \Phi]^{I}_{ij} &=\xi_k^q\partial_{x^q} \Phi^{(I,l)}_{ij}+ (\partial_{x^{l}}\xi_k^{j})\Phi^{(I,j)}_{ij} +((\partial_{x^l}\nu_k^{m} )\Phi)^{I}_{ij} +(\nu_k^{m} \partial_{x^l}\Phi)^{I}_{ij}\\
&=\xi_k^q\partial_{x^q} \Phi^{(I,l)}_{ij} +(\nu_k^{m}\Phi)^{(I,l)}_{ij}.
\end{align*}
Consequently, the components of the $\nu_k^{m}$ depend only at most $m$ derivatives of $\xi_k$. In order to apply Theorem \ref{thm:URDRemEst}, we need to  show that the pair $\bA=(A^{1}, A^{2})$ is an unbounded rough driver  (see Definition \ref{def:urd})  in the scale $(E_n)$ and that we have control of the non-linearity $\mu$ in $E_{-1}$ in the sense of \eqref{ineq:urd_drift_est}. 

Since $\bZ$ satisfies Chen's relation (i.e., \eqref{eq:chen_rel_Z}),  \eqref{eq:chen_rel_urd} holds for  $\bA$. Moreover, for all $\Phi \in E_3$ and $(s,t)\in \Delta_T$, it holds that 
\begin{align*}
A^{1,*}_{st}\Phi&=(-\xi_k \cdot \nabla \Phi+ \nu^{m,\top}_k\Phi) Z^k_{st},\\
|A^{1,*}_{st}\Phi|_{E_n}& \lesssim_{d,m} |\xi|_{W^{m+n,\infty}} |\Phi|_{E_{n+1}} \varpi_{\bZ}(s,t)^{\frac{1}{p}}, \quad  \forall n \in \{0,2\},\\
|A^{2,*}_{st}\Phi|_{E_n}&\lesssim_{d,m} |\xi|_{W^{m+1+n,\infty}}^2 |\Phi|_{E_{n+2}}\varpi_{\bZ}(s,t)^{\frac{2}{p}},\quad   
\forall n\in \{0,1\}.
\end{align*}
Thus, $\bA$ is an unbounded rough driver in the scale $(E_n)$. 

We will now show that \eqref{ineq:urd_drift_est} holds for $\mu$.  For all $I\in \clI_{m-1}$, we have
$$\partial^I[\pounds_{u}\omega] - u^q\partial_{x^q}
\partial^{I}\omega=\partial^{I}[ u^q\partial_{x^q}\omega]-  u^q\partial_{x^q}\partial^{I}\omega+ \partial^I[\omega_{qj}\partial_{x^{i}}u^q +\omega_{iq}\partial_{x^{j}}u^q].$$
Applying \eqref{ineq:alg_sob_1}, \eqref{ineq:alg_sob_2}, and \eqref{ineq:curl_equiv}, we get  
\begin{equation}\label{ineq:non-linearity_m_est}
\sum_{I\in \clI_{m-1}}|\partial^I[\pounds_{u}\omega] - u^q\partial_{x^q}
\partial^{I}\omega|_{L^2} \lesssim_{d,m}  |\nabla u|_{L^\infty} |\omega|_{H^{m-1}} + |u|_{H^{m}}|\omega|_{L^\infty}\lesssim_{d,m}  |\nabla u|_{L^\infty} |\omega|_{H^{m-1}}.
\end{equation}
Using \eqref{ineq:non-linearity_m_est} and  Poincare's inequality, we find that there exists $C=C(d,m)$ such that for all $(s,t)\in \Delta_T$, 
$$
|\mu_{st}|_{E_{-1}}\le C\int_s^t|\nabla u_r|_{L^\infty} |\omega_r|_{H^{m-1}}\,\rmd r.
$$
Therefore, in the sense of Definition \ref{def:urd_eqn}, $\Omega$ is a solution of 
$$
\rmd \Omega + \mu(\rmd t) + \bA(\rmd t)\Omega= 0.
$$
We complete the proof by invoking  Theorem  \ref{thm:URDRemEst} and \eqref{ineq:curl_equiv}.
\end{proof}

\subsection{Solution estimates}

We will now present the main solution estimates that we will use to prove local existence and establish the BKM blow-up criteria.

\begin{theorem}[Solution estimates]\label{thm:solution_est}
Let $u$ be an  $H^{m}$-solution of \eqref{eq:rough_velocity_projected} on the interval $[0,T]$. Then for all $m'\in \bbN_0$ such that  $1\le m'\le m-2$, there are constants  $C_1=C_1(d,m')$ and $C=C(p,d,m', |\xi|_{W^{m'+2,\infty}})$ depending in an increasing way on $|\xi|_{W^{m'+2,\infty}}$ such that
\begin{align}
\sup_{t\le T}|u_t|_{H^{m'}}& \le \sqrt{2} \exp\left(C \left(\int_{0}^T| \nabla u_r|_{L^\infty}\,\rmd r
+\varpi_{\bZ}(0,T)\right) \right)|u_{0}|_{H^{m'}}\label{ineq:weak_apriori}\\
\sup_{t\le T}|u_t|_{H^{m'}}& \le C_1 (1+|u_{0}|_{H^{m'}})\exp\left(C(1+\varpi_{\bZ}(0,T))\exp \left(C \int_{0}^T|\omega_r|_{L^\infty}\,\rmd r\right)\right).\label{ineq:BKM_apriori}
\end{align}
If $m_*\le m'\le m-2$, then for all $T_*\in [0,T]$ satisfying
$$
e^{C(1+\varpi_{\bZ}(0,T_*))}T_*<|u_{0}|_{H^{m'}}^{-1},
$$
it holds that
\begin{equation}\label{ineq:solution_est_small_time}
\sup_{t\le T_*}|u_t|_{H^{m'}}\le \frac{e^{C(1+\varpi_{\bZ}(0,T_*))}}{|u_{0}|_{H^{m'}}^{-1}-e^{C(1+\varpi_{\bZ}(0,T_*))}T_*}.
\end{equation}
\end{theorem}
\begin{proof}
Using the  notation of the proof of Theorem \ref{thm:remainder_est}, we let 
$$
\Omega= (\partial^I\omega)_{I\in \clI_{m'-1}}\in L^\infty([0,T]; H^2(\bbT^d;A_{d,m'-1})\cap C([0,T]; H^{-1}(\bbT^d;A_{d,m'-1})),
$$
$$
\Omega^{\natural}=(\partial^I\omega^{\natural})_{I\in \clI_{m'-1}}\in C^{\frac{p}{3}-\textnormal{var}}_{2,\textnormal{loc}}([0,T];H^{-1}(\bbT^d;A_{d,m'-1})),
$$
and 
$$
\Omega^{\sharp}=(\partial^{I}\omega^{\sharp})_{I\in \clI_{m'-1}}\in C^{\frac{p}{2}-\textnormal{var}}_{2,\textnormal{loc}}([0,T];L^2(\bbT^d;A_{d,m'-1})).
$$
Here, we have used the worst possible regularity at $m'=m-2$.

For given $n\in \{0,1,2,3\}$, let $E_n= W^{n,\infty}(\bbT^d;A_{d,m'-1}\otimes A_{d,m'-1})$.  Let $\rho:\bbR^d\rightarrow \bbR$ denote a non-negative radially symmetric function  with support in the unit ball that integrates to one. For $\eta \in (0,1)$, let $\rho_{\eta}=\eta^{-d}\rho(\frac{\cdot}{\eta})$ and define the smoothing operator $J^{\eta}: E_0\rightarrow C^\infty(\bbT^d;A_{d,m'-1}\otimes A_{d,m'-1})$ by $J^{\eta}\Phi=\Phi \ast \rho_{\eta}=\int_{\bbR^d}\rho_{\eta}(y)\Phi(\cdot-y)dV$. 
It follows that $(E_n)_{n=0}^3$  is a scale with smoothing (see Section \ref{sec:prelim_func}) $J^{\eta}=\rho_{\eta}\ast $, $\eta \in (0,1)$, in the sense of Definitions \ref{def:scale} and \ref{def:smooth_op}. For $\Phi \in E_3$ and $k\in \{1,\ldots, K\}$, denote
$$
L_{\xi_k}\Phi=\xi_k^q\partial_{x^q} \Phi + \nu_k^{m'}\Phi ,
$$
where $\nu_k^{m'}$ is defined as in the proof of Theorem \ref{thm:remainder_est}.

In order to obtain the desired estimates, we will proceed as follows. First, we will write down an unbounded rough driver equation (Definition \ref{def:urd_eqn})  for the bounded path 
$
\Omega^{\otimes 2} = \Omega \otimes \Omega: [0,T]\rightarrow E_{-0}.
$
Second, we will apply Theorem \ref{thm:URDRemEst} to obtain a bound on the associated remainder term $\Omega^{\otimes 2,\natural}$. In the final step, we will test against a $\Phi=\boldsymbol{I}$ such that $\langle \Omega\otimes \Omega, \boldsymbol{I}\rangle=|\omega|^2_{H^{m'-1}}$ and apply rough Gronwall's lemma (i.e., Lemma \ref{lem:RoughGronwall}) to obtain two solution estimates; one of which is used for local existence and other of which is used to derive the BKM blow-up criterion. 

For arbitrarily given $M_1,M_2\in A_{d,m'-1}$, denote $M_1\hat{\otimes}M_2=2^{-1}(M_1\otimes M_2+ M_2\otimes M_1)$. 
Note that since $\bZ$ is geometric, we have that for all $k,l\in \{1,\ldots, K\}$,
\begin{equation}\label{eq:geometricity_square}
L_{\xi_k}\Omega_s \otimes L_{\xi_l}\Omega_s Z^k_{st} Z^l_{st}=L_{\xi_k}\Omega_s \otimes L_{\xi_l}\Omega_s ( \bbZ^{lk}_{st} +\bbZ^{kl}_{st})=2L_{\xi_k} \Omega_s  \hat{\otimes}L_{\xi_l}  \Omega_s \bbZ^{lk}_{st}.
\end{equation} 

We will need the following fact. The pointwise tensor product of smooth functions extends canonically to a continuous bi-linear map
$$
\otimes : H^1(\bbT^d;A_{d,m'-1})\times  H^{-1}(\bbT^d;A_{d,m'-1})\rightarrow  E_{-1}.
$$
Moreover, for all $f\in H^1(\bbT^d;A_{d,m'-1})$, $g\in H^{-1}(\bbT^d;A_{d,m'-1})$, and $\Phi \in E_1$,
\begin{equation}\label{eq:evaluation}
	\langle f \otimes g ,\Phi \rangle :=\langle g , (f,\Phi)_{A_{d,m'-1}}\rangle.
\end{equation}
Applying \eqref{eq:evaluation} and  \eqref{eq:geometricity_square}, we find that for all $\Phi \in E_1$ and   $(s,t)\in \Delta_T,$
\begin{align}
\langle \delta\Omega^{\otimes 2}_{st}, \Phi\rangle &=2\langle \Omega_s\hat{\otimes}\delta \Omega_{st},\Phi\rangle  +\langle \delta \Omega_{st} \otimes \delta\Omega_{st}, \Phi\rangle \notag \\
&=-\langle \clM_{st},\Phi \rangle  -2\langle \Omega_s\hat{\otimes}L_{\xi_k}\Omega_s,\Phi \rangle Z^k_{st} \notag \\
&\quad +2\langle \Omega_s\hat{\otimes}L_{\xi_k}L_{\xi_l}\Omega_s + L_{\xi_k}\omega_s\hat{\otimes} L_{\xi_l}\omega_s, \Phi\rangle \bbZ_{st}^{lk}+ \langle \Omega^{\otimes 2,\natural}_{st}, \Phi \rangle \notag  \\
& = -\langle \clM_{st},\Phi \rangle  - \langle \Omega^{\otimes 2}_s, [\Gamma^{1,*}_{st}+\Gamma^{2,*}_{st}]\Phi \rangle +\langle \Omega^{\otimes 2,\natural}_{st}, \Phi \rangle, \label{eq:vorticity_squared}
\end{align}
where
\begin{align*}
\langle\clM_{st},\Phi\rangle& :=\sum_{I,J\in \clI_{m'-1}}\int_s^t \langle \partial^{I} [\pounds_{u_r}\omega_r] \otimes \partial^{J} \omega_r+\partial^{I} \omega_r \otimes \partial^{J} [\pounds_{u_r}\omega_r], \Phi^{IJ}\rangle \, \rmd r, \\
\Gamma^{1}\Phi &:= (\xi_k^q\partial_{x^q} \Phi+ 2( \nu^{m'}_k \hat{\otimes} \operatorname{id} )\Phi)Z^k_{st},\\
\Gamma^{2}\Phi &:= -(\xi_k^q\partial_{x^q}+ 2 \nu^{m'}_k \hat{\otimes}  \operatorname{id} )(\xi_l^q\partial_{x^q} + 2 \nu^{m'}_l \hat{\otimes}  \operatorname{id} )\Phi \bbZ^{lk}_{st},\\
\langle \Omega^{\otimes 2, \natural}_{st}, \Phi\rangle& :=2\langle \Omega_s\hat{\otimes}\Omega^{\natural}_{st}, \Phi\rangle  +\sum_{I,J\in \clI_{m'-1}}\int_s^t \langle \partial^{I} [\pounds_{u_r}\omega_r] \otimes \partial^{J} \delta \omega_{rs}+\partial^{I} \delta\omega_{rs} \otimes \partial^{J} [\pounds_{u_r}\omega_r], \Phi^{IJ}\rangle \, \rmd r \\
&\qquad +\langle \Omega^{\sharp}_{st} \otimes \delta \Omega_{st}, \Phi\rangle +\langle L_{\xi_k}\Omega_sZ^k_{st}\otimes \Omega_{st}^{\sharp}, \Phi\rangle.
\end{align*} 
Here, $2\nu^{m'}_k\hat{\otimes } \operatorname{id}: \bbT^d\rightarrow \clL(A_{d,m'-1}\otimes A_{d,m'-1};A_{d,m'-1}\otimes A_{d,m'-1})$ is the unique linear map induced by the bi-linear map $2\nu^{m'}_k(x)\hat{\otimes } \operatorname{id}:A_{d,m'-1}\times A_{d,m'-1}\rightarrow A_{d,m'-1}\otimes A_{d,m'-1}$, $x\in \bbT^d,$ defined for given $(M_1,M_2)\in A_{d,m'-1}\times A_{d,m'-1}$   by
$$
(2 \nu^{m'}_k(x) \hat{\otimes}  \operatorname{id})(M_1,M_2)=\nu^{m'}_k(x) M_1\otimes M_2 + M_1\otimes \nu^{m'}_k(x)M_2.
$$

We will now obtain control over the drift.
Proceeding as in \eqref{ineq:non-linearity_m_est}, we find that
\begin{equation}\label{ineq:non-linearity_mm3_est}
\sum_{I\in \clI_{m'-1}}|\partial^I[\pounds_{u}\omega] - u^q\partial_{x^q}\partial^{I}\omega|_{L^2} \lesssim_{d, m}|\nabla u|_{L^\infty} |\omega|_{H^{m'-1}}.
\end{equation}
Using \eqref{ineq:non-linearity_mm3_est} and  Poincare's inequality, we find that there exists a constant $C=C(d,m)$ such that for all $\Phi \in E_1$ and   $(s,t)\in \Delta_T,$
\begin{align*}
\langle\clM_{st},\Phi\rangle&=\sum_{I,J\in \clI_{m'-1}}\int_s^t \langle \left(\partial^{I} [\pounds_{u_r}\omega_r] - u^q\partial_{x^q}\partial^I \omega_r \right)\otimes \partial^J \omega_r, \Phi^{IJ}\rangle \, \rmd r \\
&\qquad +\sum_{I,J\in \clI_{m'-1}}\int_s^t \langle \partial^I \omega_r \otimes \left(\partial^{J} [\pounds_{u_r}\omega_r] - u^q_r\partial_{x^q} \partial^J\omega_r \right), \Phi^{IJ}\rangle \, \rmd r \\
&\qquad - \sum_{I,J\in \clI_{m'-1}}\int_s^t \langle u_r^q(\partial^I\omega_r \otimes \partial^J\omega_r),\partial_{x^q} \Phi^{IJ}\rangle \, \rmd r\\
&\le  C|\Phi|_{E_1}\int_s^t | \nabla u_r|_{L^{\infty}}|\omega_r|_{H^{m'-1}}^2\,\rmd r.
\end{align*}

\begin{claim}[Verification of `squared' remainder]
We have
$\Omega^{\otimes 2, \natural}\in C^{\frac{p}{3}-\textnormal{var}}_{2, \textnormal{loc}}([0,T];E_{-1})$.
\end{claim}
\begin{proof}
We will proceed by estimating term-by-term. By the embeddings  $H^{2} \otimes H^{_1}\hookrightarrow E_{-1}$ and $L^1\otimes L^1\hookrightarrow L^2$, we find that there is a constant $C=C(d,m',|\xi|_{W^{m',\infty}})$ such that for all $(s,t)\in \Delta_T$,
\begin{align*}
|\Omega_s\hat{\otimes}\Omega^{\natural}_{st}|_{E_{-1}} &\le  \sup_{ r\le T}|\Omega_r^{m'}|_{H^2}  |\Omega^{\natural}_{st}|_{\frac{p}{3}-\textnormal{var};[s,t];H^{-1}}^{\frac{p}{3}},\\
|\Omega^{\sharp}_{st} \otimes \delta \Omega_{st}|_{L^1}& \le|\Omega^{\sharp}_{st}|_{\frac{p}{2}-\textnormal{var};[s,t];L^2}^{\frac{p}{2}} |\Omega_{st}|_{p-\textnormal{var};[s,t];L^2}^{p} ,\\
|\pounds_{\xi_k}\Omega_sZ^k_{st}\otimes \Omega_{st}^{\sharp}|_{L^1} &\le  C \sup_{ r\le T}|\Omega_r|_{H^1}\varpi_{\bZ}(s,t)^{\frac{1}{p}} |\Omega_{st}^{\sharp}|_{\frac{p}{2}-\textnormal{var};[s,t];L^2}^{\frac{p}{2}},
\end{align*}
Upon using the decomposition
\begin{align*}
\langle \partial^{I} [\pounds_{u_r}\omega_r] \otimes \partial^{J} \delta \omega_{rs}, \Phi^{IJ}\rangle &=\langle \left(\partial^{I} [\pounds_{u_r}\omega_r] -u_r^q\partial_{x^q} \partial^I\omega_r\right)\otimes \partial^{J} \delta \omega_{rs}, \Phi^{IJ}\rangle \\
&\quad -\langle u^q_r\partial^I\omega_r\otimes \partial_{x^q}\partial^{J} \delta \omega_{rs}, \Phi^{IJ}\rangle-\langle u^q_r\partial^I\omega_r\otimes \partial^{J} \delta \omega_{rs},  \partial_{x^q}\Phi^{IJ}\rangle
\end{align*}
in \eqref{ineq:non-linearity_mm3_est} and applying Poincare's inequality, we find that for all $\Phi\in E_1$ and $(s,t)\in \Delta_T$,
\begin{gather*}
\sum_{I,J\in \clI_{m'-1}}\int_s^t \langle \partial^{I} [\pounds_{u_r}\omega_r] \otimes \partial^{J} \delta \omega_{rs}+\partial^{I} \delta\omega_{rs} \otimes \partial^{J} [\pounds_{u_r}\omega_r], \Phi^{IJ}\rangle \, \rmd r\\
\lesssim_{d, m} |\Phi|_{E_1}|\omega|_{p-\textnormal{var};[s,t];H^{m'}}\int_s^t|\nabla u_r|_{L^\infty}|\omega_r|_{H^{m'-1}} \,\rmd r .
\end{gather*}   
We then complete the proof by using the property that a product of  powers of regular controls whose powers sum to greater than or equal to one is a control (see, e.g., \cite{FrVi10}[Ex.\ 1.9].
\end{proof}

One can easily check that the pair $\boldsymbol{\Gamma}=(\Gamma^{1}, \Gamma^{2})$ is an unbounded rough driver in the scale $(E_n)$ and that   \eqref{ineq:urd_est} holds with  $\varpi_{\boldsymbol{\Gamma}}(s,t):= C \varpi_{\bZ}(s,t)$, $(s,t)\in \Delta_T$, for a given  constant $C=C(d,m,|\xi|_{W^{m'+2,\infty}})$ that depends in an increasing way on $|\xi|_{W^{m'+2}}$.  Therefore, $\Omega^{\otimes 2}$ is a solution of 
$$
\rmd \Omega^{\otimes 2} + \clM(\rmd t) + \boldsymbol{\Gamma}(\rmd t) \Omega^{\otimes 2}= 0
$$
in the sense of Definition \ref{def:urd_eqn}. Thus, by Theorem  \ref{thm:URDRemEst}, there is a constant $C=C(p,d,m', |\xi|_{W^{m'+2,\infty}})$ depending  in an increasing way on $|\xi|_{W^{m'+2,\infty}}$ and a constant $L=L(p)$ such that for all $(s,t)\in \Delta_T$ with $\varpi_{\boldsymbol{\Gamma}}(s,t)\le L$ it holds that
\begin{equation}\label{ineq:remainder_square_est}
|\Omega^{\otimes 2,\natural}|_{\frac{p}{3}-\textnormal{var};[s,t];E_{-3}}^{\frac{p}{3}}\le C \left(\sup_{s\le r\le t}|\omega_r|_{H^{m'-1}}^2 \varpi_{\bZ}(s,t)^{\frac{3}{p}} + \int_s^t|\nabla u_r|_{L^\infty} |\omega_r|_{H^{m'-1}}^2\,\rmd r\varpi_{\bZ}(s,t)^{\frac{1}{p}}\right).
\end{equation}

Let $\boldsymbol{I}\in A_{d,m'-1}\otimes A_{d,m'-1}$ be such that $\langle \Omega\otimes \Omega, \boldsymbol{I}\rangle=|\omega|^2_{H^{m'-1}}$.
Henceforth, let $C=C(p,d,m', |\xi|_{W^{m'+2,\infty}})$ denote a constant which increases with $|\xi|_{W^{m'+2,\infty}}$ and, thus, may change from line to line. Letting $\Phi = \boldsymbol{I}$ in  \eqref{eq:vorticity_squared}, we find that for all $(s,t)\in \Delta_T$,
$$
\delta (|\omega|_{H^{m'-1}}^2)_{st}= -2\sum_{I\in \clI_{m'-1}}\int_s^t \langle \partial^{I} [\pounds_{u_r}\omega_r] , \partial^{I} \omega_r\rangle  \, \rmd r  +\langle \Omega_s^{\otimes 2}, [\Gamma^{1,*}_{st}+\Gamma^{2,*}_{st}]\boldsymbol{I} \rangle  + \langle \Omega^{\otimes 2,\natural}_{st},\boldsymbol{I}\rangle.
$$
It follows from \eqref{ineq:remainder_square_est} and the fact that $\boldsymbol{\Gamma}$ is an unbounded rough driver that for all $(s,t)\in \Delta_T$ with $\varpi_{\boldsymbol{\Gamma}}(s,t)\le L\wedge 1$, 
\begin{equation}\label{ineq:solution_est_pregron}
\delta (|\omega|_{H^{m'-1}}^2)_{st}\le  C\left(\int_s^t | \nabla u_r|_{L^\infty}|\omega_r|_{H^{m'-1}}^2 \, \rmd r
+ \sup_{s\le r\le t}|\omega_r|_{H^{m'-1}}^2 \varpi_{\bZ}(s,t)^{\frac{1}{p}}\right).
\end{equation}
Upon applying rough Gronwall's lemma (i.e., Lemma \ref{lem:RoughGronwall}) and \eqref{ineq:curl_equiv}, we obtain \eqref{ineq:weak_apriori}. For  $t\in [0,T]$, define $y_t =\ln( e + |u_t|_{H^{m'}})$. By  virtue of the inequality \eqref{ineq:Kato} and \eqref{ineq:weak_apriori}, we have that for all $t\in [0,T]$,
$$
y_t\le \ln (\sqrt{2}y_{0}) +C \varpi_{\bZ}(0,t)+ C \int_{0}^t|\omega_r|_{L^\infty}y_r \,\rmd r.
$$
By using Gronwall's inequality, we find that  for all $t\in [0,T]$, 
$$
y_t \le\left( \ln (\sqrt{2}y_0) +C \varpi_{\bZ}(0,t)\right)\exp \left(C  \int_{0}^t|\omega_r|_{L^\infty}\,\rmd r\right),
$$
and hence \eqref{ineq:BKM_apriori} holds. 

Returning to \eqref{ineq:solution_est_pregron} and applying the Sobolev embedding and \eqref{ineq:curl_equiv} with $m'>m_*$, we find that  for all $(s,t)\in \Delta_T$ with $\varpi_{\boldsymbol{\Gamma}}(s,t)\le L\wedge 1$, 
$$
\delta (|\omega|_{H^{m'-1}}^2)_{st}\le  C\left(\int_s^t |\omega_r|_{H^{m'-1}}^3 \, \rmd r
+ \sup_{s\le r\le t}|\omega_r|_{H^{m'-1}}^2 \varpi_{\bZ}(s,t)^{\frac{1}{p}}\right).
$$
Applying rough Gronwall's lemma implies that, for all  $t\in [0,T]$ and $t'>t$,
$$
|\omega_t|_{H^{m'-1}}^2 \le e^{C(1+\varpi_{\bZ}(0,t'))}\left(|\omega_{0}|_{H^{m'-1}}^2+  \int_{0}^t |\omega_r|_{H^{m'-1}}^3 \, \rmd r\right)=:y_t.
$$
It follows that for all  $t\in [0,T]$ and $t'>t$, one has
$$
\frac{\rmd}{\rmd t}y^{-\frac{1}{2}}_t=-2^{-1}e^{C(1+\varpi_{\bZ}(0,t'))} \left(y_t^{-\frac{1}{2}}|\omega_r|_{H^{m'-1}}\right)^3\ge  -2^{-1}e^{C(1+\varpi_{\bZ}(0,t'))},
$$
which implies
$$
y^{-\frac{1}{2}}_t\ge \exp^{-\frac{C}{2}(1+\varpi_{\bZ}(0,t'))}|\omega_{0}|_{H^{m'-1}}^{-1}-2^{-1}e^{C(1+\varpi_{\bZ}(0,t')} t.
$$
Therefore, upon letting $t'\downarrow t$ and applying \eqref{ineq:curl_equiv}, we obtain \eqref{ineq:solution_est_small_time}.
\end{proof}
\begin{remark}\label{rem:double_variable}
In  \cite{BaGu15},  the method of doubling of variables was used to derive solution estimates and prove uniqueness for linear rough transport equations. The remainder term for linear rough transport equations can only be expected to belong  to $H^{-3}$ if its initial condition belongs to $L^2$. Thus, more care is needed to derive solution estimates and uniqueness since it is not possible, at least directly, to test the equation against the solution. In our setting, we can avoid doubling variables and obtain solution estimates by smoothing (e.g., the path an initial condition and  $\xi$) at the expense of having to assume $\xi$ is slightly more regular whenever continuity in time in the highest norm is needed. In particular, we prove uniqueness by considering differences only in the $L^2$-norm and not the highest norm.
\end{remark}

\subsection{Difference estimates and uniqueness}
The following theorem establishes a priori estimates of the difference of $H^m$-solutions in the $H^1$-norm and $H^m$-norm.

\begin{theorem}[Difference estimates and uniqueness]\label{thm:diff_est}
Let $u^1$ and $u^2$ be  $H^{m}$-solutions of \eqref{eq:rough_velocity_projected} on the interval $[0,T]$  with the same data $(\bZ,\xi)\in C^{p-\textnormal{var}}_g(\bbR_+;\bbR^K)\times (W^{m+2,\infty})^K$ starting from $u^1_0$ and $u^2_0$, respectively. 
Then there is a constant $C=C(p,d, |\xi|_{W^{3,\infty}})$ such that 
\begin{equation}\label{ineq:diff_est_H1}
\sup_{t\le T}|u_t^1-u_t^2|_{H^1} \le\sqrt{2} \exp\left(C\left(\int_{0}^T
\left( | u_r^1|_{H^{m_*}}+ | u_r^2|_{H^{m_*}}\right) \,\rmd r+\varpi_{\bZ}(0,T) \right)\right)|u_0^1-u_0^2|_{H^1}. 
\end{equation}
If $m_*\le m'\le m-2$, then there is a constant $C=C(p,d, m', |\xi|_{W^{m'+2,\infty}})$ such that 
\begin{equation}\label{ineq:diff_est_highest}
\sup_{t\le T}|u_t^1-u^2_t|_{H^{m'}}\le \sqrt{q_{m'}(T)} |u_{0}^1-u^2_{0}|_{H^{m'}}+ Tq_{m'}(T)|u_{0}^1-u^2_{0}|_{H^{m_*-1}}|u_{0}^1|_{H^{m'+1}}.
\end{equation}
where
$$
q_{m'}(T):=\sqrt{2} \exp\left( C\left(\int_{0}^{T}
\left( |u_r^1|_{H^{m_*}} + |u_r^2|_{H^{m'}} \right) \,\rmd r+\varpi_{\bZ}(0,T) \right)\right).
$$
\end{theorem}
\begin{proof}
Let $u=u^1-u^2$ and $\omega = \omega^1-\omega^2$. We will adopt the notation of the proof of Theorem \ref{thm:solution_est} and consider $m' \in \{1,m_*-1\}\cup \{m_*,m_*+1,\ldots \}$.  For $n\in \{0,1,2,3\}$, let $E_n= W^{n,\infty}(\bbT^d;A_{d,m'-1}\otimes A_{d,m'-1})$. By the analysis in the proof of Theorem \ref{thm:solution_est}, we have that $$\Omega^{\otimes 2} := \Omega^{1,\otimes 2}-\Omega^{2,\otimes 2}\in  L^\infty([0,T]; E_{-0})\cap  C^{p-\textnormal{var}}([0,T]; E_{-3})$$ and $$\Omega^{\otimes 2,\natural}=\Omega^{1,\otimes 2, \natural}-\Omega^{2,\otimes 2,\natural}\in C^{\frac{p}{3}-\textnormal{var}}_{2,\textnormal{loc}}([0,T];E_{-1})$$  satisfy for all $\Phi \in E_1$ and $(s,t)\in \Delta_T$,
$$
\langle \Omega^{\otimes 2,\natural}_{st}, \Phi \rangle = \langle \delta \Omega^{\otimes 2}_{st}, \Phi\rangle+\langle \clN_{st},\Phi \rangle  + \langle \Omega^{\otimes 2}_s, [\Gamma^{1,*}_{st}+\Gamma^{2,*}_{st}]\Phi \rangle,
$$ 
where 
\begin{align*}\langle\clN_{st},\Phi\rangle &:=\sum_{I,J\in \clI_{m'-1}}\int_s^t \langle \partial^{I} [\pounds_{u_r}\omega_r^{1}] \otimes \partial^{J} \omega_r+\partial^{I} \omega_r \otimes \partial^{J} [\pounds_{u_r}\omega_r^{1}] , \Phi^{IJ}\rangle\, \rmd r\\
&\qquad  +\sum_{I,J\in \clI_{m'-1}}\int_s^t \langle \left(\partial^{I} [\pounds_{u_r^2}\omega_r] - u_r^{2,q} \partial_{x^q} \partial^I \omega_r \right)\otimes \partial^J \omega_r, \Phi^{IJ}\rangle \, \rmd r\\
&\qquad  +\sum_{I,J\in \clI_{m'-1}}\int_s^t \langle\partial^I \omega_r \otimes \left(\partial^{J} [\pounds_{u_r^2}\omega_r] - u_r^{2,q}\partial_{x^q} \partial^J\omega_r \right), \Phi^{IJ}\rangle \, \rmd r\\
&\qquad-\sum_{I,J\in \clI_{m'-1}}\int_s^t \langle u_r^{2,q}\partial^I\omega_r \otimes \partial^J\omega_r,\partial_{x^q} \Phi^{IJ}\rangle \, \rmd r.
\end{align*}
For all $f\in H^1$ and $g\in H^s$ with $s>\frac{d}{2}-1$, the Sobolev embedding (see, e.g., \cite{benyi2013sobolev}) implies that
\begin{equation}\label{ineq:alg_sob_0}
	|fg|_{L^2} \lesssim_{d,s} |f|_{H^1}| g|_{H^{s}}.
\end{equation}
If $m'=1$, then using \eqref{ineq:alg_sob_0} and \eqref{ineq:alg_sob_1},  we find
\begin{align*}
|\pounds_{u}\omega^{1}|_{L^2} &\lesssim_{d} |u|_{H^1}| \omega^1|_{H^{m_*-1}} + |\nabla u|_{L^2}|\omega^1|_{L^2}\lesssim_{d}  |\omega|_{L^2} |u^1|_{H^{m_*}}\\
|\pounds_{u^2}\omega-u^2\cdot \nabla \omega|_{L^2} &\lesssim_{d}|\nabla u^2|_{L^\infty}|\omega|_{L^2} \lesssim_{d} |\omega|_{L^2}|u^2|_{H^{m_*}}.
\end{align*}
Making use of \eqref{ineq:alg_sob_1}  and \eqref{ineq:curl_equiv}, we obtain
\begin{align*}
\sum_{I\in \clI_{m'-1}}|\partial^I[\pounds_{u}\omega^{1}]|_{L^2}
&\lesssim_{d,m'} \sum_{l=0}^{m'-1}\left(||D^l u| |D^{m'-l}\omega^{1}||_{L^2}+||D^{l+1}u| |D^{m'-1-l} \omega^{1}||_{L^2} \right)\\
&\lesssim_{d,m'}|u|_{L^\infty} |\omega^{1}|_{H^{m'}} +|\omega^1|_{L^\infty}|u|_{m'}\\
&\lesssim_{d,m'}\begin{cases}
|\omega|_{H^{m'-1}}|u^1|_{H^{m_*}} & \textnormal{if} \;\; m'=m_*-1,\\
|\omega|_{H^{m_*-2}}|u^1|_{H^{m'+1}} + |\omega|_{H^{m'-1}}|u^1|_{H^{m_*}} & \textnormal{if} \;\; m'\ge  m_*,
\end{cases}
\end{align*}
and 
\begin{align*}
\sum_{I\in \clI_{m'-1}}|\partial^I[\pounds_{u^2}\omega]-u^{2,q}\partial_{x^q}\partial^I \omega|_{L^2}
&\lesssim_{d,m'} \sum_{l=1}^{m'-1}||D^{l} u^2| |D^{m'-l} \omega||_{L^2}+\sum_{l=0}^{m'-1}||D^{l+1} u^2| |D^{m'-1-l} \omega||_{L^2} \\
&\lesssim_{d,m'}
\begin{cases}
\sum_{l=0}^{m'-1}||D^{l}\nabla  u^2| |D^{m'-l} u||_{L^2}& \textnormal{if} \; m'=m_*-1,\\
\sum_{l=0}^{m'-1}||D^{l}\nabla u^2| |D^{m'-1-l} \omega||_{L^2} & \textnormal{if} \; m'\ge  m_*,
\end{cases}\\
&\lesssim_{d,m'}
\begin{cases}
|\nabla u^2|_{L^\infty} |u|_{H^{m'}} + |u|_{L^\infty} | u^2|_{H^{m_*}} & \textnormal{if} \;\; m'=m_*-1,\\
|\nabla u^2|_{L^\infty} |\omega|_{H^{m'-1}} + |\omega|_{L^\infty} | u^2|_{H^{m'}} & \textnormal{if} \; m'\ge  m_*,
\end{cases}\\
&\lesssim_{d,m'}|\omega|_{H^{m'-1}} |u^2|_{H^{m_*}}.
\end{align*}
Putting it all together and using $|u^2|_{L^\infty}\lesssim_d |u^2|_{H^{m_*}}$, we get that there exists a constant $C=C(d,m')$ such that for all $(s,t)\in \Delta_T,$
\begin{align*}
|\clN_{st}^{m'}|_{E_{-1}}&\le C\begin{cases}
\int_s^t  |\omega_r|_{L^2}^2\left( |u^1_r|_{H^{m_*}} + |u_r^2|_{H^{m_*}} \right) \rmd r& \textnormal{if} \; m'=1, \\
\int_s^t  |\omega_r|_{H^{m'-2}}^2\left( |u^1_r|_{H^{m_*}} + |u_r^2|_{H^{m_*}} \right) \rmd r& \textnormal{if} \; m'=m_*-1, \\
\int_s^t |\omega_r|_{H^{m'-1}}|\omega_r|_{H^{m_*-2}}|u^{1}_r|_{H^{m'+1}}+|\omega_r|_{H^{m'-1}}^2\left( |u^1_r|_{H^{m_*}} + |u_r^2|_{H^{m'}} \right)  \,\rmd r&\textnormal{if} \; m'\ge m_*,
\end{cases}\\
&\quad =:\varpi_{\clN^{m'}}(s,t).
\end{align*}
Therefore, $\Omega^{\otimes 2}$ is a solution of 
$$
\rmd \Omega^{\otimes 2} + \clN(\rmd t) + \boldsymbol{\Gamma}(\rmd t) \Omega^{\otimes 2}= 0
$$
in the sense of Definition \ref{def:urd_eqn}.
Owing to Theorem  \ref{thm:URDRemEst}, there is a constant $C=C(p,d,m, |\xi|_{W^{m'+2,\infty}})$ depending  in an increasing way on $|\xi|_{W^{m'+2,\infty}}$ and a constant $L=L(p)$ such that for all $(s,t)\in \Delta_T$ with $\varpi_{\boldsymbol{\Gamma}}(s,t)\le L$ it holds that
$$
|\Omega^{ \otimes 2,\natural}|_{\frac{p}{3}-\textnormal{var};[s,t];E_{-3}}^{\frac{p}{3}}\le C \left(\sup_{s\le r\le t}|\omega_r|_{H^{m'-1}}^2 \varpi_{\bZ}(s,t)^{\frac{3}{p}} + \varpi_{\clN^{m'}}(s,t)\varpi_{\bZ}(s,t)^{\frac{1}{p}}\right).
$$
Henceforth, let $C=C(p,d,m, |\xi|_{W^{m'+2,\infty}})$ denote a constant depending in an increasing way on $|\xi|_{W^{m'+2,\infty}}$ that may change from line to line. Letting $\Phi = \boldsymbol{I}$ in  \eqref{eq:vorticity_squared} as in the proof of Theorem \ref{thm:solution_est}, we get that for all $(s,t)\in \Delta_T$,
$$
\delta (|\omega|_{H^{m'-1}}^2)_{st}\le  -2\sum_{I\in \clI_{m'-1}}\int_s^t  \langle \partial^{I} [\pounds_{u_r}\omega_r^{1}+\pounds_{u^2_r}\omega_r] , \partial^I \omega_r\rangle  \, \rmd r  +\langle \Omega_s^{\otimes 2}, [\Gamma^{1,*}_{st}+\Gamma^{2,*}_{st}]\boldsymbol{I} \rangle  + \langle \Omega^{\otimes 2,\natural}_{st},\boldsymbol{I}\rangle.
$$

Thus, if $m'\in \{1,m_*-1\}$, then for all $(s,t)\in \Delta_T$ with $\varpi_{\boldsymbol{\Gamma}}(s,t)\le L\wedge 1$ ,  we have
$$
\delta (|\omega|_{H^{m'-1}}^2)_{st}\le C\int_s^t|\omega_r|_{H^{m'-1-2}}^2\left( |u^1_r|_{H^{m_*}} + |u_r^2|_{H^{m_*}} \right) \,\rmd r +C \sup_{s\le r\le t}|\omega_r|_{H^{m'-1}}^2 \varpi_{\bZ}(s,t)^{\frac{1}{p}}.
$$
Applying Lemma \ref{lem:RoughGronwall}, we get
\begin{equation}\label{ineq:diff_est_lower}
\sup_{t\le T}|\omega_t^1-\omega_t^2|_{H^{m'-1}} \le\sqrt{2} \exp\left(C\left(\int_{0}^T
\left( | u_r^1|_{H^{m_*}}+ | u_r^2|_{H^{m_*}}\right) \,\rmd r+\varpi_{\bZ}(0,T) \right)\right)|\omega_0^1-\omega_0^2|_{H^{m'-1}},
\end{equation}
and hence \eqref{ineq:diff_est_H1} follows from \eqref{ineq:curl_equiv}.

If $m'\ge m_*$, then for all $(s,t)\in \Delta_T$ with $\varpi_{\boldsymbol{\Gamma}}(s,t)\le L\wedge 1$ , 
\begin{align*}
\delta (|\omega|_{H^{m'-1}}^2)_{st}&\le  C \int_s^t\left( |\omega_r|_{H^{m'-1}}|\omega_r|_{H^{m_*-2}}|u^{1}_r|_{H^{m'+1}}+|\omega_r|_{H^{m'-1}}^2\left(|u^1_r|_{H^{m_*}} +|u^2_r|_{H^{m'}}\right)\right) \,\rmd r\\
&\quad +C \sup_{s\le r\le t}|\omega_r|_{H^{m'-1}}^2 \varpi_{\bZ}(s,t)^{\frac{1}{p}}.
\end{align*}
Using Lemma \ref{lem:RoughGronwall}, \eqref{ineq:weak_apriori}, and \eqref{ineq:diff_est_lower}, we obtain that for all $t\in [0,T]$ and $t'>t$,
$$
|\omega_t|_{H^{m'-1}}^2\le q_{m'}(t')\left( |\omega_{0}|_{H^{m'-1}}^2+ |\omega_{0}|_{H^{m_*-2}}|\omega^1_{0}|_{H^{m_*-1}}\int_{0}^t|\omega_r|_{H^{m'-1}} \,\rmd r \right),
$$
where  
$$
q_{m'}(t'):=\sqrt{2} \exp\left( C\left(\int_{0}^{t'}
\left( |u^1_r|_{H^{m_*}} + |u_r^2|_{H^{m'}} \right) \,\rmd r+\varpi_{\bZ}(0,t') \right)\right).
$$
Thus, letting $t'\downarrow t$, we find that for all $t\in [0,T]$,
$$
|\omega^1_t-\omega^2_t|_{H^{m'}}\le \sqrt{q_{m'}(t)} |\omega^1_{0}-\omega^2_{0}|_{H^{m'}}+ tq_{m'}(t)|\omega^1_{0}-\omega^2_{0}|_{H^{m_*-2}}|\omega^1_{0}|_{H^{m'+1}},
$$
which yields \eqref{ineq:diff_est_highest}.
\end{proof}

\section{Local well-posedness (Proof of  Theorem \ref{thm:local_wp})} \label{sec:local_wp}
\begin{proof}[Proof of Theorem \ref{thm:local_wp}]
There exists a sequence $\{(u_0^n,\mathbf{Z}^n,\xi^n)\}_{n=1}^\infty \in \mathring{C}^\infty_{\sigma}\times C_g^{1-\textnormal{var}}\times (C^\infty_{\sigma})^K$ such that for all $n\in \bbN,$ $\bZ^n=(\delta z^n, \bbZ^n)$ is the canonical lift of a $z\in C^\infty(\bbR_+;\bbR^K)$, $\{(u_0^n,\mathbf{Z}^n,\xi^n)\}_{n=1}^\infty$  converges to $(u_0,\bZ, \xi)$ in  $H^{m}_{\sigma}\times C_g^{p-\textnormal{var}}\times W^{m+1,\infty}$, and  for all $n\in \bbN$,
\begin{equation}\label{ineq:data_bounds}
|u_{0}^n|_{H^{m}}\le|u_{0}|_{H^{m}},\quad|\xi^n|_{W^{m+2,\infty}} \le |\xi|_{W^{m+2,\infty}},\quad
\varpi_{\bZ^n}(s,t)\le 1+ \varpi_{\bZ}(s,t), \; \forall (s,t)\in \Delta_{[0,\infty)}.
\end{equation}
For all $n\in \bbN$, there exists a maximally extended  solution  $u^n\in  C^{1}([0,T_{\textnormal{max}}^n);\mathring{C}^\infty_{\sigma})$  of 
\begin{equation}\label{eq:approximate_eq}
\begin{cases}
\rmd u^n + P(u^n\cdot \nabla u^n)\rmd t- P\pounds^*_{\xi_k^n}  u^n\rmd z^{n,k}_t=0\quad \textnormal{on}\;\; (0,T_{\textnormal{max}}^n)\times \bbT^d,\\
u^n=u^n_{0} \quad \textnormal{on}\;\; \{0\}\times \bbT^d
\end{cases}
\end{equation}
such that  if $T^n_{\textnormal{max}}<\infty$, then $\limsup_{t\uparrow T_{\textnormal{max}}^n}|u^n_t|_{H^{m_*}}= \infty$.\footnote{Indeed, to see this, for given $N\in \bbN$,  consider the equation
\begin{equation*}
\begin{cases}
\partial_t u^{N,n}+PP_{\le N}(P_{\le N}u^{N,n}\cdot \nabla P_{\le N}u^N) -PP_{\le N}(\pounds_{\xi_k}^*P_{\le N}u^{N,n})\, \dot{z}^{n,k}_t=0
\quad \textnormal{on}\;\; (0,\infty)\times \bbT^d ,\\
u^{N}=P_{\le N}u_0 \quad \textnormal{on} \;\;\{0\}\times \bbT^d.
\end{cases}
\end{equation*}
One may then derive solution estimates independent of $N$ but dependent on $\int_0^t|\dot{z}^n|\,\rmd t$ as in \cite{majda2002vorticity}[Chapter 3] or \cite{bahouri2011fourier}[Chapter 7]. We are not able to derive a priori estimates independent of $N$ and $n$ jointly because the presence of the projection $P_{\le N}$ prohibits us from deriving the unbounded rough driver equation (i.e. Definition \ref{def:urd_eqn}) for  $\partial_I\omega^{N,n}\otimes \partial_J\omega^{N,n}$, and hence applying Theorem \ref{thm:URDRemEst} to derive solution estimates independent of $n$ as in the proof of Theorem \ref{thm:solution_est}. For fixed $n\in \bbN$, one passes to the limit as $N\rightarrow \infty$ and establishes all other solution properties (e.g., BKM blow-up criterion) as in \cite{majda2002vorticity}[Chapter 3] or \cite{bahouri2011fourier}[Chapter 7] or as detailed below for the rough version of the equation.} 
Integrating \eqref{eq:approximate_eq} over an arbitrary interval $[s,t]\subset [0,T^n_{\textnormal{max}})$ and then substituting the equation into the $\rmd Z^{n,k}$-integral twice, we obtain
\begin{equation}\label{eq:rough_velocity_proj_approx}
u^{n,P,\natural}_{st}=\delta u^n_{st}+ \int_s^tP(u^n_r\cdot \nabla u^n_r)\, \rmd r-  P\pounds_{\xi_k^n}^*u^n_sZ^{n,k}_{st} -P[\pounds_{\xi_k^n}^*P[\pounds_{\xi_l^n}^*u_s^n]]\mathbb{Z}^{n,lk}_{st},
\end{equation}  
where
\begin{align*}
u^{n,P,\natural}_{st}&:=-\int_{s<t_2<t_1<t}P\pounds^*_{\xi_k^n}  Pu^n_{t_2}\cdot \nabla u^n_{t_2}\rmd t_2\, \rmd z^{n,k}_{t_1} -\int_{s<t_3<t_2<t_1<t}P\pounds_{\xi_k^n} P\pounds^*_{\xi_l^n}P (u^n_{t_3}\cdot \nabla u^n_{t_3})\rmd t_3\, \rmd z^{n,l}_{t_2}\, \rmd z^{n,k}_{t_1}\\
&\qquad + \int_{s<t_3<t_2<t_1<t}P\pounds^*_{\xi_k^n} P\pounds^*_{\xi_l^n}P\pounds^*_{\xi_q^n} u^n_{t_3}\, \rmd z^{n,q}_{t_3}\, \rmd z^{n,l}_{t_2}\, \rmd z^{n,k}_{t_1}.
\end{align*}
and hence $u^{n,P,\natural}\in C_{2,\textnormal{loc}}^{\frac{p}{3}-\textnormal{var}}([0,T^n_{\textnormal{max}}); \mathring{C}^\infty_{\sigma})$. In particular, $u^n$ is a $H^{m+2}$-solution  of \eqref{eq:rough_velocity_proj_approx} on the interval $[0,T^n_{\textnormal{max}})$ in the sense of Definition \ref{def:solution_velocity}.

By virtue of \eqref{ineq:data_bounds} and Theorem \ref{thm:solution_est}, there are constants $C_1=C_1(p,d,m, |\xi|_{W^{m_*+2,\infty}})$ depending in an increasing way on $|\xi|_{W^{m_*+2,\infty}}$ and constant $C_2=C_2(p,d,m, |\xi|_{W^{m+2,\infty}})$ such that for all $T_*< T^{n}_{\textnormal{max}}$ satisfying
$$
e^{C_1(1+\varpi_{\bZ}(0,T_*))}T_*<|u_{0}|_{H^{m_*}}^{-1},
$$
it holds that for all $n\in \bbN$, 
\begin{align}
\sup_{t\le T_*}|u_t^n|_{H^{m_*}}&\le \frac{e^{C_1(1+\varpi_{\bZ}(0,T_*))}}{|u_{0}^n|_{H^{m_*}}^{-1}-e^{C_1(1+\varpi_{\bZ}(0,T_*))} T_*},\label{ineq:mstar_seq_est}\\
\sup_{t\le T_*}|u_t^n|_{H^{m}} &\le \sqrt{2} \exp\left(C_2 \left(\int_{0}^{T_*}| \nabla u_s^n|_{L^\infty}\,\rmd r
+\varpi_{\bZ}(0,T_*)\right) \right)|u_{0}|_{H^{m}} \label{ineq:m_seq_est}\\
&\le \sqrt{2} \exp\left(C_2 \left(\int_{0}^{T_*}|u_s^n|_{H^{m_*}}\,\rmd r
+\varpi_{\bZ}(0,T_*)\right) \right)|u_{0}|_{H^{m}}.\notag
\end{align}
Furthermore, by Theorem \ref{thm:remainder_est}, there exists a constant $C=C(p,d,m, |\xi|_{W^{m+2,\infty}})$ such that for all $(s,t)\in \Delta_{T_*}$ with $C\varpi_{\bZ}(s,t)\le 1$, it holds that
$$
|u^{n,P,\natural}|_{\frac{p}{3}-\textnormal{var};[s,t];H^{m-3}}^{\frac{p}{3}}\le C \left(\sup_{s\le r\le t}|u_r^n|_{H^{m}} \varpi_{\bZ}(s,t)^{\frac{3}{p}} + \int_s^t |u_r^n|_{H^{m}}^2\,\rmd r\varpi_{\bZ}(s,t)^{\frac{1}{p}}\right)
$$
and for all $(s,t)\in \Delta_T$ with $C\varpi_{\bZ}(s,t)+C\int_s^t|u_r^n|_{H^{m}}^2\,\rmd r\le 1$, it holds that
$$
|u^n|_{p-\textnormal{var};[s,t];H^{m-1}}^p\le C \left(\int_s^t|u_r^n|_{H^{m}}^2\,\rmd r+ \sup_{s\le r\le t}|u_r|_{H^{m}}\left( \left(\int_s^t|u_r^n|_{H^{m}}^2\,\rmd r\right)^{\frac{1}{p}}+\varpi_{\bZ}(s,t)^{\frac{1}{p}}\right)\right).
$$

We deduce that $\{u^n\}_{n=1}^\infty$ is  bounded in $L^\infty([0,T_*]; \mathring{H}_{\sigma}^{m})\cap C^{p-\textnormal{var}}([0,T_*]; \mathring{H}^{m-1}_{\sigma})$. In particular,  $\{u^n\}_{n=1}^\infty$ is equicontinuous in $\mathring{H}^{m-1}_{\sigma}$. For all $\epsilon>0$,  there exists a $\theta\in (0,1)$ such that for all $(s,t)\in [0,T_*]^2$,
$$
|u^n_t-u^n_s|_{H^{m-\epsilon}} \lesssim_{d,m} |u^n_t-u^n_s|_{H^{m-1}}^{\theta}  \sup_{t\le T_*} |u^n_t|_{H^{m}}^{1-\theta},
$$
which implies that the sequence $\{u^n\}_{n=1}^\infty$ is equicontinuous in $H^{m-\epsilon}$. Moreover, by virtue of the boundedness of $\{u^n\}_{n=1}^\infty$ in  $L^\infty([0,T_*]; \mathring{H}_{\sigma}^{m}) $ and the compactness of $H^{m}$ in $H^{m-\epsilon}$,  the generalized Arzel\'a-Ascoli theorem implies that there exists a subsequence $\{u^{n_j}\}$ that converges to $u$ in $C([0,T_*]; \mathring{H}^{m-\epsilon}_{\sigma})$ for any $\epsilon>0$. Moreover, we may extract a further subsequence, also denoted by $\{u^{n_j}\}$ that converges to $u$ in the weak-star topology of $L^\infty([0,T_*]; \mathring{H}_{\sigma}^{m})$. Furthermore,
$
u \in C^{p-\textnormal{var}}([0,T_*]; \mathring{H}^{m-1}_{\sigma}) \cap L^\infty([0,T_*]; \mathring{H}_{\sigma}^{m}).
$
Using the lower semicontinuity of norms in the weak-star topology,  we may pass to the limit in \eqref{ineq:mstar_seq_est} and \eqref{ineq:m_seq_est} to obtain \eqref{ineq:m_star_soln_est} and \eqref{ineq:m_soln_est}.

\begin{claim}
$u \in  C_w([0,T_*]; \mathring{H}^{m}_{\sigma})$
\end{claim}
\begin{proof}
Fix an arbitrary $\epsilon>0$ and  $\phi \in \mathring{H}^{-m}_{\sigma}$. We need to show that there exists a $\delta_{\epsilon,\phi}>0$ such that for all $(s,t)\in [0,T_*]$ with  $|t-s|<\delta_{\epsilon,\phi}$, it holds that 
$$
\langle \phi, u_t-u_{s}\rangle _{H^m}<\epsilon
$$
Since $\mathring{H}^{-m+1}_{\sigma}$ is dense in $\mathring{H}^{-m}_{\sigma}$, there exists a $\psi_{\epsilon}\in \mathring{H}^{-m+1}_{\sigma}$  such that 
$$
|\phi-\psi_{\epsilon}|_{H^{-m}}<\frac{\epsilon}{4}\sup_{t\le T_*}|u_t|_{H^m}^{-1}.
$$
It follows that for all $(s,t)\in [0,T_*]^2$,
\begin{align*}
\langle \phi, u_t-u_{s}\rangle_{H^m}&= \langle\psi_{\epsilon}, u_t-u_{s}\rangle_{H^{m-1}}+\langle \phi-\psi_{\epsilon}, u_t-u_{s}\rangle_{H^m}\\
&\le \langle \psi_{\epsilon}, u_t-u_{s}\rangle_{H^{m-1}} + 2\sup_{t\le T_*}|u|_{H^m}|\phi-\psi_{\epsilon}|_{H^{-m}}\\
&< \langle \psi_{\epsilon}, u_t-u_{s}\rangle_{H^{m-1}} +\frac{\epsilon}{2}.
\end{align*}
Since $u\in C([0,T_*];\mathring{H}_{\sigma}^{m-1,2})\subset C_w([0,T_*];\mathring{H}_{\sigma}^{m-1})$ we can find $\delta_{\epsilon,\phi}>0$ such that for all $(s,t)\in [0,T_*]^2$ with $|t-s|<\delta_{\epsilon,\phi}$,
$
\langle \psi_{\epsilon}, u_t-u_{s}\rangle _{H^{m-1}} <\frac{\epsilon}{2},
$
which completes the proof of the claim.
\end{proof}

We will now show that $u$ is a $H^{m}$-solution of  \ref{eq:rough_velocity_projected} on the interval $[0,T_*]$ in the sense of Definition \ref{def:solution_velocity}. By \eqref{eq:rough_velocity_proj_approx}, for all $\phi \in \mathring{C}^\infty_{\sigma}$ and $(s,t)\in \Delta_{[0,T_*]}$, we have
$$
\langle u^{n_j,P,\natural}_{st},\phi\rangle=(\delta u^{n_j}_{st},\phi)_{L^2}- \int_s^t(u^{n_j}_r\otimes u^{n_j}_r,\nabla \phi)_{L^2}\,\rmd r+  (u^{n_j}_s,\pounds_{\xi_k^n}\phi)_{L^2}Z^{n_j,k}_{st}  -(u_s^{n_j},\pounds_{\xi_l^n}P\pounds_{\xi_k^n}\phi)_{L^2}\mathbb{Z}^{n_j,lk}_{st}. 
$$ 
Since $u^{n_j}\rightarrow u$ in $C([0,T_*];\mathring{L}^2_{\sigma})$, $\xi^{n_j}\rightarrow \xi$ in $(W^{2,\infty}_{\sigma})^K$,  and $\bZ^{n_j}\rightarrow \bZ$  in $C_g^{p-\textnormal{var}}$,    
the right-hand-side of the above converges for every $\phi \in \mathring{C}^\infty_{\sigma}$. Thus, $\langle u^{n,P,\natural}_{st},\phi\rangle $ converges to $\langle u^{\natural,P}_{st},\phi\rangle $, which implies that $u^{P,\natural} \in C^{\frac{p}{2}-\textnormal{var}}([0,T];\mathring{D}'_{\sigma})$.  Since $\{u^{n_j,P,\natural}\}$ is bounded uniformly in $C^{\frac{p}{3}-\textnormal{var}}_{2,\textnormal{loc}}([0,T_*]; \mathring{H}_{\sigma}^{m-3})$, we have $u^{P,\natural} \in C^{\frac{p}{3}-\textnormal{var}}_{2,\textnormal{loc}}([0,T_*]; \mathring{H}_{\sigma}^{m-3})$, which completes the proof that $u$ is an $H^m$-solution in the sense of Definition \ref{def:solution_velocity}. Uniqueness follows immediately from \eqref{ineq:diff_est_H1} in Theorem \ref{thm:diff_est}.

Finally, we will show that if $\xi \in (W^{m+4,\infty})^K$, then $u\in C([0,T_*];\mathring{H}^m_{\sigma})$.
For given $N\in \bbN$, let $u^N_0=P_{\le 2^N}u_0 \in \mathring{C}_{\sigma}$. By the above proof of local well-posedness, since $\xi \in W^{m+4,\infty}_{\sigma}$ and $|u_0^N|_{H^{m_*}}\le |u_0|_{H^{m_*}}$ for all $N\in \bbN$, there exists a sequence  of $H^{m+2}$-solutions $\{u^N\}_{N=1}^\infty\subset L^\infty([0,T_*];\mathring{H}^{m+2}_{\sigma})\cap C^{p-\textnormal{var}}([0,T_*];\mathring{H}^{m+1}_{\sigma})$ of \eqref{eq:rough_velocity_projected} on interval $[0,T_*]$ with initial conditions $\{u^{N}_0\}_{N=1}^\infty\subset \mathring{H}^{m+2}_{\sigma}$. Moreover, for all $N\in \bbN,$
\begin{align*}
\sup_{t\le T_*}|u_t^N|_{H^{m_*}}&\le\frac{e^{C_1(1+\varpi_{\bZ}(0,T_*))}}{|u_{0}^N|_{H^{m_*}}^{-1}-e^{C_1(1+\varpi_{\bZ}(0,T_*))} T_*} \\
&\le \frac{e^{C_1(1+\varpi_{\bZ}(0,T_*))}}{|u_{0}|_{H^{m_*}}^{-1}-e^{C_1(1+\varpi_{\bZ}(0,T_*))} T_*},\\
\sup_{t\le T_*}|u_t^N|_{H^{m}}&\le \sqrt{2} \exp\left(C_2 \left(\int_{0}^{T_*}| u_s^N|_{H^{m_*}}\,\rmd s
+\varpi_{\bZ}(0,T_*)\right) \right)|u_{0}|_{H^{m}}.
\end{align*}
By virtue of \eqref{ineq:diff_est_highest}, for all $N\in \bbN$, we have
$$
\sup_{t\le T_*}|u^N_t-u_t|_{H^{m}} \le \sqrt{q_{N}(T_*)} |u^N_{0}-u_{0}|_{H^{m}} + T_* q_{N}(T_*) |u^N_{0}-u_{0}|_{H^{m-1}} |u^N_{0}|_{H^{m+1}},
$$
where  
$$
q_{N}(T_*):=\sqrt{2} \exp\left(C\left(\int_{0}^{T_*}
\left( |u^N_s|_{H^m} + |u_s|_{H^{m}} \right) \,\rmd s+\varpi_{\bZ}(0,T_*) \right)\right).
$$
It follows that
\begin{align*}
2^{-N}|u^N_{0}|_{H^{m+1}}&=2^{-N}|P_{\le 2^N}u_{0}|_{H^{m+1}} \le |u_{0}|_{H^{m}}\\
\lim_{N\rightarrow \infty} 2^{N}|u^N_{0}-u_{0}|_{H^{m-1}}&\le  \lim_{N\rightarrow \infty} \left(\sum_{ |n|> 2^{N}} |n|^{2m}|\hat{u}_0(n)|^2\right)^{\frac{1}{2}}= 0,
\end{align*}
which implies that the sequence $\{u^N\}_{N=1}^\infty$ converges to $u$ in  $C([0,T_*]; \mathring{H}_{\sigma}^{m})$, and hence $u\in C([0,T_*]; \mathring{H}_{\sigma}^{m})$.
\end{proof}

\section{Proof of the remaining results}\label{sec:proof_remaining}
\subsection{Recovery of the pressure and harmonic part (Proof of Proposition \ref{prop:pressure_and_constant_recovery}) } \label{sec:pressure_recovery}
\begin{proof}[Proof of Proposition \ref{prop:pressure_and_constant_recovery}]
Define  $\Xi: \Delta_T\rightarrow H^{m-3}$ by
$$
\Xi_{st}=  \pounds_{\xi_k}^*u_sZ^k_{st} +\pounds_{\xi_k}^*P\pounds_{\xi_l}^*u_s\mathbb{Z}^{lk}_{st}.
$$
Using that $P=I-Q-H$ (see Section \ref{sec:prelim_Hodge}), we find that for all $(s,t)\in \Delta_T$, we have 
\begin{align*}
u^{P,\natural}_{st}&=\delta u_{st}+ \int_s^tu_r\cdot \nabla u_r\,\rmd r-\int_s^tQ(u_r\cdot \nabla u_r)\,\rmd r -\Xi_{st} + Q\Xi_{st} + H\Xi_{st}.
\end{align*}
We  will apply the sewing lemma (i.e., Lemma \ref{lem:sewing}) to construct the rough integral 
$$
I_t=\int_{0}^t \pounds_{\xi_k}^*u_r \rmd Z^k_r, \quad I_0=0,\\
$$
from the local expansion $\Xi$. For all $(s,\theta,t)\in \Delta^2_{T}$, 
$$
\delta \Xi_{s\theta t}=-\pounds_{\xi_k}^*u^{P,\sharp}_{s\theta}Z^k_{\theta t}- \pounds_{\xi_k}^*P\pounds_{\xi_l}^*\delta u_{s\theta}\bbZ^{lk}_{\theta t},
$$
where $u^{P,\sharp}$ is as defined in Theorem \ref{thm:remainder_est}. By Theorem \ref{thm:remainder_est}, for all $(s,t)\in \Delta_T$ with $C\varpi_{\bZ}(s,t)+C\int_s^t|u_r|_{H^{m}}^2\rmd r\le  1 $, it holds that
\begin{align*}
|\Xi_{st}|_{H^{m-3}}&\le C \sup_{ r\le T}|u_r|_{H^{m-1}} \varpi_{\bZ}(s,t)^{\frac{1}{p}}\\
|\delta \Xi_{s\theta t}|_{H^{m-3}}&\le C\left(|u^{P,\sharp}|_{\frac{p}{2}-\textnormal{var},[s,t],H^{m-2}}^{\frac{p}{2}} \varpi_{\bZ}(s,t)^{\frac{1}{p}} +|u|_{p-\textnormal{var},[s,t],H^{m-1}}^p\varpi_{\bZ}(s,t)^{\frac{2}{p}} \right).
\end{align*}
Thus, applying Lemma \ref{lem:sewing}, we find that there is a  paths $I\in C^{p-\textnormal{var}}([0,T];H^{m-3})$, $I^Q=QI\in C^{p-\textnormal{var}}([0,T];\nabla \mathring{H}^{m-2})$, and $h=HI\in C^{p-\textnormal{var}}([0,T];\bbR^d)$  such that $u^{\natural}:=\delta I- \Xi\in C^{\frac{p}{3}-\textnormal{var}}_{2,\textnormal{loc}}([0,T];H^{m-3})$, $u^{Q,\natural}:=\delta I^Q- Q\Xi\in C^{\frac{p}{3}-\textnormal{var}}_{2,\textnormal{loc}}([0,T];\nabla \mathring{H}^{m-2})$ and $u^{H,\natural}:=\delta h- H\Xi\in C^{\frac{p}{3}-\textnormal{var}}_{2,\textnormal{loc}}([0,T];\bbR^d)$.

For all $t\in [0,T]$, we have 
$$
\int_{0}^t|Q(u_r \cdot \nabla u_r)|_{H^{m-3}}\,\rmd r \le \int_{0}^t|u _r\cdot \nabla u_r |_{H^{m-3}}\,\rmd r \le \int_{0}^t |\nabla u_r|_{L^\infty} | u_r|_{H^{m-3}} \,\rmd r.
$$
We define $q\in  C^{p-\textnormal{var}}([0,T];\mathring{H}^{m-2})$ by 
$$\nabla q_t =-\int_{0}^{t} Q(u_r \cdot \nabla u_r)\, \rmd r + I^Q_{t}.$$ Therefore, we have that for all $(s,t)\in \Delta_T$, 
\begin{align*}
u^{P,\natural}_{st} + u^{\natural}_{st} +u^{Q,\natural}_{st} + u^{H,\natural}_{st}  &=\delta u_{st}+ \int_s^tu_r\cdot \nabla u_r\,\rmd r-\delta I_{st} + \nabla \delta q_{st} + \delta h_{st},
\end{align*}
from which we deduce  \eqref{eq:def_solution_velocity_pressure}  since the right-hand-side is a path of $\frac{p}{3}$-variation, and hence zero.
\end{proof}

\subsection{Maximally extended solution (Proof of  Corollary \ref{cor:maximal_solution})} \label{sec:maximally extended solution}

We will give a constructive proof of the maximally extended solution as it will be used in subsequent proofs.

\begin{remark}\label{rem:local_well_init_time}
Both Definition \ref{def:solution_velocity} and Theorem \ref{thm:local_wp} can be extended in the obvious manner to account for an arbitrary initial time $t_0\ge 0$. More precisely, for all $u_0\in \mathring{H}^m_{\sigma}$ and $\xi\in (W^{m+2}_{\sigma})^k$, there exists a unique $H^{m}$-solution  $u\in C_w([t_0,T_*];\mathring{H}_{\sigma}^{m})\cap C^{p-\textnormal{var}}([t_0,T_*];\mathring{H}^{m-1}_{\sigma})$ on an interval $[t_0,T_*]$  with initial condition $u|_{t=t_0}=u_0$ for any time $T_*$ satisfying
$$
e^{C_1(1+\varpi_{\bZ}(t_0,T_*))}(T_*-t_0)<|u_{0}|_{H^{m_*}}^{-1}
$$
and 
$$
\sup_{t\in [t_0,T_*]}|u_t|_{H^{m_*}}\le \frac{e^{C_1(1+\varpi_{\bZ}(t_0,T_*))}}{|u_{0}|_{H^{m_*}}^{-1}-e^{C_1(1+\varpi_{\bZ}(t_0,T_*))} (T_*-t_0)}.
$$
Moreover, if $m>m_*$,
$$
\sup_{t\in [t_0,T_*]}|u_t|_{H^{m}} \le \sqrt{2} \exp\left(C_2 \left(\int_{t_0}^{T_*}|\nabla u_s|_{L^\infty}\,\rmd s
+\varpi_{\bZ}(t_0,T_*)\right) \right)|u_{0}|_{H^{m}}.
$$
\end{remark}

\begin{proof}[Proof of Corollary \ref{cor:maximal_solution}]

Let $R>1$  and $\bar{C}>C_1(p,d,|\xi|_{W^{m_*+2,\infty}})$ be arbitrarily given, where $C_1$ is as given in \eqref{ineq:time_local_ex} in Theorem \ref{thm:local_wp}. Define $f: \bbR_+^2\rightarrow [0,\infty)$ by 
$$f(s,\delta)=e^{\bar{C}(1+\varpi_{\bZ}(s,s+\delta))}\delta.$$ Note that for all $s\in \bbR_+$, $f(s,0)=0$ and $\lim_{\delta\uparrow \infty}f(s,\delta)=\infty$. Since $\varpi_{\bZ}$ is non-decreasing in its second argument, $f$ is strictly increasing in $\delta$, and hence there exists an inverse of $f$ in its second argument $\delta_*:\bbR_+^2\rightarrow \bbR_+$; that is, for all $(s,M)\in \bbR_+^2$,
$$
f(s,\delta_*(s,M))=e^{\bar{C}(1+\varpi_{\bZ}(s,\delta_*(s,M)))}\delta_*(s,M)=M.
$$

Let $T_0=0$ and $u_{T_0}^0=u_0$. Let $T_1=T_0+\delta_*(T_{0},(R+|u_{T_0}^0|_{H^{m_*}})^{-1}))$ so that 
$$
e^{C_1(1+\varpi_{\bZ}(T_{0},T_{1}))}(T_{1}-T_{0})\le e^{\bar{C}(1+\varpi_{\bZ}(T_{0},T_{1}))}(T_{1}-T_{0})\le (R+|u_{T_{0}}^0|_{H^{m_*}})^{-1}< |u_{T_{0}}^0|_{H^{m_*}}^{-1}.
$$
By Theorem \ref{thm:local_wp}, there exists an $H^m$-solution $u^1$  of \eqref{eq:rough_velocity_projected} on the interval $[0,T_1]$. Let $T_2=T_1+\delta_*(T_{1},(R+|u_{T_1}^1|_{H^{m_*}})^{-1}))$. Since  
$$
e^{C_1(1+\varpi_{\bZ}(T_{1},T_{2}))}(T_{2}-T_{1})\le  e^{\bar{C}(1+\varpi_{\bZ}(T_{1},T_{2}))}(T_{1}-T_{0})= (R+|u_{T_{1}}^1|_{H^{m_*}})^{-1}< |u^1_{T_{1}}|_{H^{m_*}}^{-1},
$$
by Remark \ref{rem:local_well_init_time}, there exists an $H^m$-solution $\tilde{u}^2$  of \eqref{eq:rough_velocity_projected} on the interval $[T_1,T_2]$. Let $u^2=u^1$ on $[0,T_1]$ and $u^2=\tilde{u}^2$ on $[T_1,T_2]$ so that $u^2$ is a $H^m$-solution of \eqref{eq:rough_velocity_projected} on the interval $[0,T_2]$. Proceeding by induction, we define 
$$
T_{l+1}=T_{l}+\delta_*(T_{l},(R+|u^l_{T_{l}}|_{H^{m_*}})^{-1}), \;\;l \in \{3,\ldots, \},
$$
and appeal to Remark \ref{rem:local_well_init_time} to obtain an $H^m$-solution $u^n$ of \eqref{eq:rough_velocity_projected} on the interval $[0,T_l]$ for all $l\in \bbN$. Let $T_{\textnormal{max}}=\sup_{n\in \bbN_0}T_l$ and $u^{\textnormal{max}}=\lim_{l\rightarrow \infty}u^l$. It follows that $u^{\textnormal{max}}$ is the unique $H^m$-solution on the open $[0, T_{\textnormal{max}})$ by virtue of \eqref{ineq:diff_est_H1} in Theorem \ref{thm:diff_est}.

Assume that $T_{\textnormal{max}}<\infty$. Suppose, by contradiction, that $\limsup_{t\uparrow T_{\textnormal{max}}}|u_{t}|_{H^{m_*}}<\infty$. In particular, there exists $M>0$ such that $|u_t|_{H^{m_*}}<M$ for all $t<T_{\textnormal{max}}$. For arbitrarily chosen $\epsilon>0$, there exists a $L=L(\epsilon)>0$ such that for all $l>L$, $T_{l+1}-T_l<\epsilon$, which implies
$$
e^{\bar{C}(1+\varpi_{\bZ}(0,T_{\textnormal{max}}))}(R+M)^{-1}<e^{\bar{C}(1+\varpi_{\bZ}(T_{l},T_{l+1}))}(R+|u_{T_{l}}|_{H^{m}})^{-1}<\epsilon.
$$
Choosing $\epsilon>0$ smaller than the left-hand-side, we obtain a contradiction.

Suppose, by contradiction, there exists $\bar{T}\in (T_{\textnormal{max}},\infty]$ and an $H^m$-solution $\bar{u}$ of \eqref{eq:rough_velocity_projected} on the interval $[0,\bar{T})$. By virtue of uniqueness (i.e., \eqref{ineq:diff_est_H1} in Theorem \ref{thm:diff_est}), we have $u\equiv \bar{u}$ on $[0,T_{\textnormal{max}})$, which implies $$\limsup_{t\uparrow T_{\textnormal{max}}} |u_{t}|_{H^{m_*}}= |\bar{u}_{T_{\textnormal{max}}}|_{H^{m_*}}<\infty \quad \Rightarrow \quad T_{\textnormal{max}}=\infty.$$  This leads to a contradiction, which precludes the existence of an extension of $u$.
\end{proof}

\subsection{BKM blow-up criterion (Proof of Theorem \ref{thm:BKM})}
\begin{proof}[Proof of Theorem \ref{thm:BKM}]
The strategy of the proof is to first construct an approximation sequence of $H^{m+2}$-solutions that converges to $u$ on any interval $[0,T]\subset [0,T_{\textnormal{max}})$. We will then use the solution estimate given in  \eqref{ineq:BKM_apriori} in Theorem \ref{thm:solution_est} and then pass to the limit in both sides of the solution bound. The approximation sequence is constructed separately for $m=m_*$ and $m>m_*$. If $m=m_*$, we only approximate the initial condition, assume $(W^{m_*+4,\infty}_{\sigma})^K$, and use  \eqref{ineq:diff_est_highest} in Theorem \ref{thm:diff_est} to establish convergence in $C([0,T];\mathring{H}^{m_*}_{\sigma})$, which is needed to continue the approximating sequence up to $T_{\textnormal{max}}$. If $m>m_*$, we can avoid the assumption $\xi \in (W^{m+4,\infty})^K$ for all $m$ by smoothing the initial condition and $\xi$ and relying on compactness to obtain convergence in $C([0,T];\mathring{H}^{m_*}_{\sigma})$, which allows us to continue the approximating sequence up to $T_{\textnormal{max}}$. 

\textbf{Case $m=m_*$}. 
Let $\{T_n\}_{n=0}^\infty$ be the sequence of times specified in the proof of Corollary \ref{cor:maximal_solution} converging to $T_{\textnormal{max}}$. In the proof of Theorem \ref{thm:local_wp}, we showed that the sequence $\{u^n\}_{n=1}^\infty$ of $H^{m+2}$-solutions on the interval $[0,T_1]$ corresponding to the initial conditions $\{P_{\le 2^n}u_0\}_{n=1}^\infty \in C_{\sigma}^\infty$ converges to $u$ in $C([0,T_1];\mathring{H}^{m_*}_{\sigma})$. In particular, there exists an $N_2\in \bbN$ be such that for all $n\ge N_2$,
$
|u_{T_1}^n|_{H^{m_*}}< |u_{T_1}|_{H^{m_*}} + 2^{-1}R.
$
Thus, by Remark \ref{rem:local_well_init_time} and \eqref{ineq:diff_est_H1} (i.e., uniqueness), we can extend the solutions $\{u^n\}_{n=N_2}^\infty$ to the interval $[0,T_2]$ such that for all $n\ge N_2$,
\begin{align*}
\sup_{t\in [T_1,T_2]}|u_t^n|_{H^{m_*}}&\le \frac{e^{C_1(1+\varpi_{\bZ}(T_1,T_2))}}{|u_{T_1}^n|_{H^{m_*}}^{-1}-e^{C_1(1+\varpi_{\bZ}(T_1,T_2))} (T_2-T_1)}\\
&\le \frac{e^{C_1(1+\varpi_{\bZ}(T_1,T_2))}}{(2^{-1}R+|u_{T_1}|_{H^{m_*}})^{-1}-(R+|u_{T_1}|_{H^{m_*}})^{-1}}.
\end{align*}
Repeating the argument at the end of the proof of Theorem \ref{thm:local_wp} (i.e., applying \eqref{ineq:diff_est_highest}) on the interval $[0,T_2]$, we get that the sequence $\{u^n\}_{n=N_2}^\infty$ of $H^{m+2}$-solutions on the interval $[0,T_2]$ converges to $u$ in $C([0,T_2];\mathring{H}^{m_*}_{\sigma})$. Proceeding inductively, for all $l\in \bbN$, we find an $N_{l}\in \bbN$ such that the sequence $\{u^n\}_{n=N_l}^\infty$ converges to $u$ in $C([0,T_l];\mathring{H}^{m_*}_{\sigma})$.

\textbf{Case $m>m_*$}.
There exists a sequence $\{(u_0^n,\xi^n)\}_{n=1}^\infty \in \mathring{C}^\infty_{\sigma} \times (C^\infty_{\sigma})^K$ that converges to $(u_0,\xi)$ in  $H^{m}_{\sigma}\times W^{m+1,\infty}$ and such that  for all $n\in \bbN$,
$
|u_{0}^n|_{H^{m}}\le|u_{0}|_{H^{m}}, |\xi^n|_{W^{m+2,\infty}} \le |\xi|_{W^{m+2,\infty}}.
$
Let $\{T_n\}_{n=0}^\infty$ be the sequence of times specified in the proof of Corollary \ref{cor:maximal_solution} converging to $T_{\textnormal{max}}$. By Theorem \ref{thm:local_wp}, there exists a sequence of $H^{m+2}$-solutions $\{u^n\}_{n=1}^\infty$  on the interval $[0,T_1]$ corresponding to the data $\{(u_0^n,\xi^n)\}_{n=1}^\infty$ such that for all $n\in \bbN,$
\begin{align*}
\sup_{t\in [0,T_1]}|u_t^n|_{H^{m_*}}&\le \frac{e^{C_1(1+\varpi_{\bZ}(0,T_1))}}{|u_{0}|_{H^{m_*}}^{-1}-(R+|u_{0}|_{H^{m_*}})^{-1}}\\
\sup_{t\in [0,T_1]}|u_t^n|_{H^{m}}&\le \sqrt{2} \exp\left(C_2 \left(\int_{0}^{T_1}|u_s^n|_{H^{m_*}}\,\rmd s
+\varpi_{\bZ}(0,T_1)\right) \right)|u_{0}|_{H^{m}}.
\end{align*}
Following the proof of Theorem \ref{thm:local_wp} (i.e., applying Theorem \ref{thm:remainder_est}, compactness, and Arzel\'a-Ascoli, and then passing to the limit) and using uniqueness, we find that the full sequence $\{u^n\}_{n=1}^\infty$ converges to $u$ in $C([0,T_1];\mathring{H}^{m-\epsilon}_{\sigma})$ for any $\epsilon>0$ and in the weak-star topology of $L^\infty([0,T_1];\mathring{H}^m_{\sigma})$.
In particular, there exists an $N_2\in \bbN$ be such that for all $n\ge N_2$,
$
|u_{T_1}^n|_{H^{m_*}}< |u_{T_1}|_{H^{m_*}} + 2^{-1}R.
$
By Remark \ref{rem:local_well_init_time} and \eqref{ineq:diff_est_H1} (i.e., uniqueness), we may continue the solutions  to the interval $[0,T_2]$ such that for all $n\ge N_2$
\begin{align*}
\sup_{t\in [T_1,T_2]}|u_t^n|_{H^{m_*}}&\le \frac{e^{C_1(1+\varpi_{\bZ}(T_1,T_2))}}{(2^{-1}R+|u_{T_1}|_{H^{m_*}})^{-1}-(R+|u_{T_1}|_{H^{m_*}})^{-1}}\\
\sup_{t\in [0,T_2]}|u_t^n|_{H^{m}}&\le \sqrt{2} \exp\left(C_2 \left(\int_{0}^{T_2}|u_s^n|_{H^{m_*}}\,\rmd s
+\varpi_{\bZ}(0,T_2)\right) \right)|u_{0}|_{H^{m}}.
\end{align*}
Again, following the proof of Theorem \ref{thm:local_wp} and using uniqueness, we get that the full sequence $\{u^n\}_{n=N_2}^\infty$  converges to $u$ in $C([0,T_2];\mathring{H}^{m-\epsilon}_{\sigma})$ for any $\epsilon>0$ and in the weak-star topology of $L^\infty([0,T_2];\mathring{H}^m_{\sigma})$. Proceeding inductively, for all $l\in \bbN$, we find an $N_{l}\in \bbN$ such that the sequence $\{u^n\}_{n=N_l}^\infty$ converges to $u$ in $C([0,T_l];\mathring{H}^{m-\epsilon}_{\sigma})$ for any $\epsilon>0$ and in the weak-star topology of $L^\infty([0,T_l];\mathring{H}^m_{\sigma})$. 

We will now show \eqref{ineq:BKM_blowup_bound}.
Let $T<T_{\textnormal{max}}=\sup_{l\in \bbN}T_l$ be arbitrarily given and $N(T)\in \bbN$ be such that the sequences constructed above $\{u^n\}_{n=N(T)}^\infty$ of  converges to $u$ in $C([0,T];\mathring{H}^{m_*}_{\sigma})$ and in the weak-star topology of $L^\infty([0,T];\mathring{H}^m_{\sigma})$ if $m>m_*$. It follows that $\{\omega^n=\bd \flat u^n\}_{n=N(T)}^\infty$ converges to $\omega$ in  $C([0,T];L^{\infty})$.\footnote{We actually only need convergence of $\{u^n\}$ in $C([0,T];\mathring{H}^{m_*-\epsilon}_{\sigma})$ for any $\epsilon>0$ to conclude this. However, in order to construct the approximating sequence up to the maximal time $T_{\textnormal{max}}$, we needed the convergence in $C([0,T];\mathring{H}^{m_*}_{\sigma})$.} By \eqref{ineq:BKM_apriori} in Theorem  \eqref{ineq:BKM_apriori}, there exists constants   $C_1=C_1(d,m)$ and $C_2=C_2(p,d,m, |\xi|_{W^{m+2,\infty}})$ such that for all $n\ge N(T)$,
$$
\sup_{t\le T}|u_t^n|_{H^{m}} \le C_1 (1+|u_{0}|_{H^{m}})\exp\left(C_2(1+\varpi_{\bZ}(0,T))\exp \left(C_2 \int_{0}^T|\omega_s^n|_{L^\infty}\,\rmd s\right)\right).
$$
Using the lower semi-continuity of weak-star convergence if $m>m_*$, we pass to the limit as $n\rightarrow \infty$ on both sides of the inequality to obtain \eqref{ineq:BKM_blowup_bound}.  

If $T_{\textnormal{max}}<\infty$, then $\limsup_{t\uparrow T_{\textnormal{max}}} |u_t|_{H^m}=\infty$, which yields $\limsup_{t\uparrow T_{\textnormal{max}}}\int_0^t|\omega_s|_{L^{\infty}}=\infty$ by \eqref{ineq:BKM_blowup_bound}. Conversely, if $\limsup_{t\uparrow T_{\textnormal{max}}}\int_0^t|\omega_s|_{L^{\infty}}\rmd s=\infty$, then
$$\infty =\limsup_{t\uparrow T_{\textnormal{max}}}\int_0^t|\omega_s|_{L^{\infty}}\rmd s \le \limsup_{t\uparrow T_{\textnormal{max}}}\int_0^t |u_s|_{H^{m}}\rmd s,$$ 
which implies $\limsup_{t\uparrow T_{\textnormal{max}}}|u_t|_{H^m}=\infty$.
\end{proof}

\subsection{Global well-posedness in two-dimensions (Proof of Theorem \ref{thm:global2d})}
\begin{proof}[Proof of Corollary \ref{thm:global2d}]
Consider the sequence $\{u^n\}_{n=1}^\infty$ of $H^{\infty}$-solutions on $[0,T^n_{\textnormal{max}})$ corresponding to the data $\{(u_0^n,\mathbf{Z}^n,\xi^n)\}_{n=1}^\infty \in \mathring{C}^\infty_{\sigma}\times C_g^{1-\textnormal{var}}\times (C^\infty_{\sigma})^K$ from the proof of Theorem \ref{thm:local_wp}. For given $n\in \bbN$, let $\tilde{\omega}^n= \operatorname{curl}u^n$, so that (see Remark \ref{rem:vorticity_scalar_vec}),
\begin{equation*}
\begin{cases}
\rmd \tilde{\omega}^n + u^n\cdot \nabla \tilde{\omega}^n\rmd t+ \xi_k^n\cdot \nabla \tilde{\omega}^n\rmd z^{n,k}_t=0\quad \textnormal{on}\;\; [0,T^n_{\textnormal{max}})\times \bbT^d,\\
\tilde{\omega}^n=\tilde{\omega}^n_{0} \quad \textnormal{on}\;\; \{0\}\times \bbT^d.
\end{cases}
\end{equation*}
It follows that for all $n\in \bbN$, $t\in [0,\infty)$, and $p\in [2,\infty],$
\begin{equation}\label{eq:Lp_conserved_approx}
|\omega_t^n|_{L^p}= |\omega^n_0|_{L^p} \le |\omega_0|_{L^p}.
\end{equation}
Moreover, $T^n_{\textnormal{max}}=\infty$ since (see, e.g., \cite{majda2002vorticity}[Chapter 3] or \cite{bahouri2011fourier}[Chapter 7])  for all $n\in \bbN,$
$$
|u_t^n|_{H^{m}} \le C_1 (1+|u_{0}^n|_{H^{m}})\exp\left(C_2(1+\int_0^t|\dot{z}^{n}_s|\rmd s)\exp \left(C_2t |\omega_0^n|_{L^\infty}\right)\right).
$$
Owing to \eqref{ineq:BKM_apriori} in Theorem \ref{thm:solution_est}, there exists constants   $C_1=C_1(m)$ and $C_2=C_2(p,m, |\xi|_{W^{m+2,\infty}})$ such that for all $n\in \bbN$ and  $t\ge 0$, 
\begin{equation}\label{ineq:double_exp_2d_approx}
|u_t^n|_{H^{m}} \le C_1 (1+|u_{0}|_{H^{m}})\exp\left(C_2(1+\varpi_{\bZ}(0,t))\exp \left(C_2t |\omega_0|_{L^\infty}\right)\right).
\end{equation}
Proceeding as in the proof of Theorem \ref{thm:local_wp} and using uniqueness, we find that $\{u^n\}_{n=1}^\infty$ converges to $u$ in $C([0,\infty);\mathring{H}^{m-\epsilon}_{\sigma})$ for any $\epsilon>0$ and in the weak-star topology of $L^\infty([0,\infty);\mathring{H}^m_{\sigma})$. We may then pass to the limit \eqref{ineq:double_exp_2d_approx} using lower semicontinuity of weak-star convergence to obtain \eqref{ineq:double_exp_2d}. Since there always exists an $\epsilon>0$ such that $m-\epsilon>\frac{d}{2}+1,$ $\{\tilde{\omega}^n\}_{n=1}^\infty$ converges to $\omega$ in $C([0,T];L^{\infty})$, and thus passing to the limit in \eqref{eq:Lp_conserved_approx} yields \eqref{eq:vort_Lp_conserved}.
\end{proof}

\subsection{Continuous dependence on data (Proof of Corollary \ref{cor:stability})}

\begin{proof}[Proof of  Corollary \ref{cor:stability}]
Let $R>1$ be such that for all $n\in \bbN$,
$$
|u_{0}^n|_{H^{m}}\le R,\quad|\xi^n|_{W^{m+2,\infty}} +|\xi|_{W^{m+2,\infty}} \le R,\quad
\varpi_{\bZ^n}(s,t)\le R+ \varpi_{\bZ}(s,t), \; \forall (s,t)\in \Delta_{[0,\infty)}.
$$

We will establish the convergence of $\{u^n\}$ for the cases $d=2$ and $m>m_*$ separately.

\textbf{Case $d=2$}. Owing to \eqref{ineq:double_exp_2d}, there exists a constant  $\bar{C}=\bar{C}(p,m,R)$ such that for all  and $n\in \bbN$ and  $t\ge 0$,
$$
|u_t^n|_{H^{m}} \le C_1 (1+R)\exp\left(\bar{C}(1+\varpi_{\bZ}(0,t))\exp \left(\bar{C}|\tilde{\omega}_0|_{L^\infty}t\right)\right).
$$
Proceeding as in the proof of Theorem \ref{thm:local_wp}, we find that $\{u^n\}_{n=N_1}^\infty$ converges to $u$ in $C([0,\infty);\mathring{H}^{m-\epsilon}_{\sigma})$ for any $\epsilon>0$ and in the weak-star topology of $L^\infty([0,\infty);\mathring{H}^m_{\sigma})$.

\textbf{Case $m>m_*$}. Let $C_1(|\xi|_{W^{m_*+2,\infty}})=C_1(p,d,m,|\xi|_{W^{m_*+2,\infty}})$ denote the constant appearing in \eqref{ineq:time_local_ex}. Let $\bar{C}=RC_1(p,d,m,R)$ and note that for all $(s,t)\in \Delta_{[0,\infty)}$,
$$
C_1(|\xi_n|_{W^{m_*+2,\infty}})(1+\varpi_{\bZ^n}(s,t))\vee C_1(|\xi|_{W^{m_*+2,\infty}})(1+\varpi_{\bZ}(s,t)) \le \bar{C}(1+\varpi_{\bZ}(s,t)). 
$$
Let $\{T_n\}_{n=0}^\infty$ be the sequence of times specified in the proof of Corollary \ref{cor:maximal_solution} with  $R$ and $\bar{C}$ as just specified. We then proceed as in the proof of Theorem \ref{thm:BKM} to show that for all $T<T_{\textnormal{max}}$, there exists an $N(T)\in \bbN$  such that the sequence $\{u^n\}_{n=N(T)}^\infty$ of $H^m$-solutions converges to $u$ in $C([0,T];\mathring{H}^{m-\epsilon}_{\sigma})$ for any $\epsilon>0$ and in the weak-star topology of $L^\infty([0,T];\mathring{H}^m_{\sigma})$. 

We will now turn our attention to showing the convergence of the pressure and harmonic constant. Let $T<T_{\textnormal{max}}$ and $N(T)=1$ if $d=2$ and $N(T)$ as specified if $m>m_*$.
By Proposition \ref{prop:pressure_and_constant_recovery}, for all $n\ge N(T)$, there exists $q^n\in C^{p-\textrm{var}}([0,T];\nabla \mathring{H}^{m-2})$ and $h^n\in C^{p-\textrm{var}}([0,T];\bbR^d)$ such that
$$
u_t^n-u_0^n+ \int_0^tu_s^n\cdot \nabla u_s^n\,\rmd s- \int_0^t\pounds_{\xi_k^n}^*u_s^n\, \rmd \bZ^{n,k}_{s} =-\nabla  q^n_{t} - h^n_{t},
$$
where, upon adopting the notation in the proof of Proposition \ref{prop:pressure_and_constant_recovery}, for all $(s,t)\in \Delta_T$, we have
$$
\nabla \delta q^n_{st} := \int_s^t Q (u_r^n \cdot \nabla u_r^n)\rmd r + Q\Xi_{st}^{n} + u^{n,Q,\natural},\quad
\delta h_{st}^n = H\Xi_{st}^{n} + u^{n,H,\natural},
$$
and
$$
\Xi^n_{st}=  \pounds_{\xi_k^n}^*u_s^nZ^{n,k}_{st} +\pounds_{\xi_k^n}^*P\pounds_{\xi_l^n}^*u_s^n\mathbb{Z}^{n,lk}_{st}.
$$
Thus,  for all $(s,t)\in \Delta_{T}$
\begin{align*}
|\nabla \delta q^n_{st}|_{H^{m-3}}&\le (t-s)\sup_{s\le r\le t} |\nabla u_r^n|_{L^\infty} | u_r^n|_{H^{m-3}} \,\rmd r + |\Xi_{st}^{n}|_{H^{m-3}} +  |u^{n,Q,\natural}|_{H^{m-3}}\\
|\delta h_{st}^n| &\le |\Xi_{st}^{n}|_{H^{m-3}} +  |u^{n,Q,\natural}|_{H^{m-3}}.
\end{align*}
There exists a constant $C=C(p,d,m, R)$ such that
$$
|\Xi_{st}^n|_{H^{m-3}}\le C \sup_{ r\le T}|u_r^n|_{H^{m-1}}\varpi_{\bZ^n}(s,t)^{\frac{1}{p}}.
$$
Theorem \ref{thm:remainder_est} implies that there exists a constant $C=C(p,d,m, R)$ such that for all $(s,t)\in \Delta_T$ with $C\varpi_{\bZ^n}(s,t)+C\int_s^t|u_r^n|_{H^{m}}^2\rmd r\le  1 $, it holds that
\begin{align*}
|\delta \Xi_{s\theta t}^n|_{H^{m-3}}&\le C\left(|u^{n,P,\sharp}|_{\frac{p}{2}-\textnormal{var},[s,t],H^{m-2}}^{\frac{p}{2}} \varpi_{\bZ^n}(s,t)^{\frac{1}{p}} +|u^n|_{p-\textnormal{var},[s,t],H^{m-1}}^p\varpi_{\bZ^n}(s,t)^{\frac{2}{p}} \right)\\
|u^{n,P,\sharp}|_{\frac{p}{2}-\textnormal{var};[s,t];H^{m-2}}^{\frac{p}{2}}&\le  C\left( \int_s^t |u_r^n|_{H^{m}}^2\,\rmd r+ \sup_{s\le r\le t}|u_r^n|_{H^{m}}\varpi_{\bZ^n}(s,t)^{\frac{2}{p}}\right)
\end{align*}
The bound \eqref{ineq:sewing_lemma} in the sewing lemma then implies that
$$
u^{n,Q,\natural} := \delta I^{n,Q} - Q\Xi^{n} \quad \textnormal{and}  \quad u^{n,H,\natural}: = \delta h^n - H\Xi^{n},
$$
are  bounded independent of $n$ in $C_{2,\textnormal{loc}}^{\frac{p}{3}-\textrm{var}}([0,T];\nabla \mathring{H}^{m-2})$ and $C_{2,\textnormal{loc}}^{\frac{p}{3}-\textrm{var}}([0,T];\bbR^d),$ respectively. Therefore, 
$\{( \nabla q^n,h^n)\}_{n=N(T)}^\infty$ is bounded in $C^{p-\textrm{var}}([0,T];\nabla \mathring{H}^{m-2})\times C^{p-\textrm{var}}([0,T];\bbR^d)$. 
Following the proof of Theorem \ref{thm:local_wp} and using uniqueness of $u$ from which uniqueness of $(q,h)$ follows, we deduce that $\{ (\nabla q^n,h^n)\}_{n=N(T)}^\infty$ converges to $(q,h)$ in $C([0,T]; \nabla \mathring{H}^{m-2-\epsilon})\times C([0,T];\bbR^d)$ for any $\epsilon>0$.
\end{proof}

\begin{appendices}

\section{Unbounded rough drivers}\label{sec:URD}

In this section, we present some elements of the theory of unbounded rough drivers \cite{BaGu15} and the associated remainder estimates. We use Theorem \ref{thm:URDRemEst} in Section \ref{sec:apriori} to derive a priori estimates of the remainder and solution (i.e., Theorems \ref{thm:remainder_est}, \ref{thm:solution_est}, and \ref{thm:diff_est}).

\begin{definition}[Scale]\label{def:scale}
We say a sequence  of Banach spaces $(E_{n}, | \cdot |_{n})_{n=0}^3=(E_n)$ is a scale, if $E_{n+1}$ is continuously embedded into $E_{n}$ for all $n\in\{0,1,2,3 \}$. Denote by $E_{-n}$ the strong topological dual of $E_{n}$.
\end{definition}
\begin{definition}[Unbounded rough driver] \label{def:urd}
For a given interval $[0,T]$, let $A^i: \Delta_T\rightarrow \clL(E_{-n};E_{-(n+i)})$ for $n\in \{0,1,2\}$ and $i\in \{1,2\}$. 
For a given $p\in [2,3)$, a pair of $2$-index maps  $\mathbf{A} = (A^1,A^2)$ is called an unbounded $p$-rough driver on the interval $[0,T]$ with respect to the scale $(E_{n})$ if there exists a regular control $\varpi_{\bA}$ on $[0,T]$ such that for every $(s,t)\in \Delta_T$,
\begin{equation}
\begin{aligned}\label{ineq:urd_est}
| A^1_{st}|_{\clL(E_{-n};E_{-(n+1)})}^p &\leq\varpi_{\bA}(s,t) \ \  \text{for}\ \ n \in \{0,2\},\\ 
|A^2_{st}|_{\clL(E_{-n};E_{-(n+2)})}^{\frac{p}{2}} &\leq\varpi_{\bA}(s,t) \ \ \text{for}\ \ n \in \{0,1\},
\end{aligned}
\end{equation}
and, in addition, Chen's relations hold:
\begin{equation}\label{eq:chen_rel_urd}
\delta A^1_{s\theta t}=0 \quad  \textnormal{ and } \quad \delta A^2_{s\theta t}= -A^1_{\theta t}A^1_{s\theta},\qquad \forall (s,\theta,t)\in\Delta^{2}_{T}.
\end{equation}
\end{definition}

\begin{definition}[Smoothing]\label{def:smooth_op}
We say a family of  operators $(J^{\eta})_{\eta \in (0,1)}$ is a smoothing on a given scale $(E_n)$ if the following conditions are satisfied: 
\begin{align*}
|J^{\eta}-I|_{\clL(E_m;E_n)}&\lesssim \eta^{m-n} \quad \textnormal{ for } (n,m)\in \{(0,1),(0,2),(1,2)\}\\
|J^{\eta}|_{\clL(E_n;E_m)}&\lesssim \eta^{-(m-n)}\quad \textnormal{ for } (n,m)\in \{(1,1),(1,2),(2,2),(1,3),(2,3)\}.
\end{align*}
\end{definition}

\begin{definition}[Solution of unbounded rough driver equation]\label{def:urd_eqn}
Let $\mathbf{A}=(A^1,A^2)$ be a continuous unbounded $p$-rough driver on $[0,T]$ with respect to a scale $(E_n)$. Let $\mu \in C^{1-\textnormal{var}}([0,T]; E_{-1})$.  Assume that $(E_n)$ admits a smoothing. A  bounded path  $f:[0,T]\rightarrow E_{-0}$ is called a solution  of
\begin{equation}\label{eq:urd}
\rmd f+\mu(\rmd t)+\mathbf{A}(\rmd t)f=0
\end{equation}
on the interval $[0,T]$, provided $f^{\natural}: \Delta_T\rightarrow E_{-3}$  defined   for every $(s,t)\in \Delta_T$ and $\phi\in E_3$ by 
$$
\langle f^{\natural}_{st},\phi\rangle =\langle \delta f_{st}, \phi\rangle +\langle \delta \mu_{st}, \phi\rangle +\langle f_s, (A^{1,*}_{st}+A^{2,*}_{st})\phi\rangle
$$
satisfies $f^{\natural}\in C^{\frac{p}{3}-\textnormal{var}}_{2, \textnormal{loc}}([0,T]; E_{-3})$.
\end{definition}

Define the map $f^{\sharp}:\Delta_T\rightarrow E_{-3}$  for every $(s,t)\in \Delta_T$ and $\phi\in E_3$ by
$$
\langle f^{\sharp}_{st}, \phi\rangle =\langle \delta f_{st} , \phi\rangle +\langle f_s, A^{1,*}_{st}\phi \rangle  = -\langle \delta \mu_{st}, \phi\rangle - \langle f_s, A^{2,*}_{st}\phi \rangle +\langle f^{\natural}_{st},\phi\rangle.
$$
The first expression for $f^\sharp_{st}$ consists of terms that are less regular in time  and more regular in space than the second expression for  $f^\sharp_{st}$. The following theorem is proved using interpolation, and hence properties of the smoothing operators, and the sewing lemma (i.e., Lemma \ref{lem:sewing}). Its proof can be found, for example, in  \cite{hocquet2018energy}[Proposition 3.1].
\begin{theorem}[Unbounded rough driver estimates]\label{thm:URDRemEst}
Let $u$ be a solution of \eqref{eq:urd}  and  assume  that there exists a regular control $\varpi_{\mu}$  on the interval $[0,T]$ such that for all $(s,t)\in \Delta_T$,
\begin{equation}\label{ineq:urd_drift_est}
|\delta\mu_{st}|_{E_{-1}}\le \varpi_{\mu}(s,t). 
\end{equation}
Then there exists positive constants $C=C(p)$ and $L=L(p)$ such that for all $(s,t)\in \Delta_T$ with  $\varpi_{\bA}(s,t)\le L$ it holds that 
$$
|f^{\natural}|_{\frac{p}{3}-\textnormal{var};[s,t];E_{-3}}^{\frac{p}{3}}\le C\left(\sup_{s\le r\le t}|f_r|_{E_{-0}} \varpi_{\bA}(s,t)^{\frac{3}{p}} +  \varpi_{\mu}(s,t) \varpi_{\bA}(s,t)^{\frac{1}{p}}\right)\,.
$$
Furthermore, for all $(s,t)\in \Delta_T$ with  $\varpi_{\bA}(s,t)+\varpi_{\mu}(s,t)\le L$  it holds that
\begin{align*}
|f^{\sharp}|_{\frac{p}{2}-\textnormal{var};[s,t];E_{-2}}^{\frac{p}{2}}&\le C \left(\varpi_{\mu}(s,t)+ \sup_{s\le r\le t}|f_r|_{E_{-0}}\varpi_{\bA}(s,t)^{\frac{2}{p}})\right)\,,\\
|f|_{p-\textnormal{var},[s,t];E_{-1}}^{p}&\le C \left(\varpi_{\mu}(s,t)+ \sup_{s\le r\le t}|f_r|_{E_{-0}}( \varpi_{\mu}(s,t) ^{\frac{1}{p}}+\varpi_{\bA}(s,t)^{\frac{1}{p}})\right).
\end{align*}
\end{theorem}

\section{Rough Gronwall's lemma}\label{sec:rough_gron}
In this section, we state a rough version of Gronwall's lemma. The proof can be found, for example, in  \cite{DeGuHoTi16}.  

\begin{lemma}[Rough Gronwall's lemma]\label{lem:RoughGronwall}
Let  $L$ and $p$ denote positive constants. Let $\kappa \in L^1(I)$ and  $\varpi$ be a regular control on the  interval $[0,T]$.  Let $\phi : \Delta_T \rightarrow \bbR_+$ be such that $\phi(s,t) \leq \phi(0,T)$ for all $(s,t)\in \Delta_T$. Assume that $G: [0,T] \rightarrow \bbR_+$ is such that for every $(s,t)\in \Delta_T$ with $\varpi(s,t)\le L$,
$$
\delta G_{st} \leq \phi(s,t) + \int_s^t\kappa_r G_r \,\rmd r+\varpi(s,t)^{\frac{1}{p}}\sup_{ r\le t}G_t.
$$
Then there exists a positive constant $\beta=\beta(L,p)$ such that 
$$
\sup_{ 0\le  t \leq T} G_{t} \leq 2 \exp \left( \beta \left(\int_{0}^T\kappa_r \,\rmd r+ \varpi(0,T)\right)\right) \left( G_0 +  \phi(0,T)\right).
$$
\end{lemma}

\end{appendices}
\bibliographystyle{alpha}
\bibliography{bibliography}

\end{document}